\documentclass[12pt]{article}

\usepackage{arxiv}

\usepackage[utf8]{inputenc} 
\usepackage[T1]{fontenc}    
\usepackage{hyperref}       
\usepackage{url}            
\usepackage{booktabs}       
\usepackage{amsfonts}       
\usepackage{nicefrac}       
\usepackage{microtype}      
\usepackage{lipsum}		    
\usepackage{graphicx}
\usepackage{natbib}
\usepackage{doi}

\usepackage{mathtools,amsmath,amsfonts,amssymb,amsthm,subcaption}
\usepackage{xcolor}
\usepackage{enumitem}
\usepackage[ruled]{algorithm2e}

\theoremstyle{plain}
\theoremstyle{plain}\newtheorem{theorem}{Theorem}[section]
\theoremstyle{plain}\newtheorem{lemma}{Lemma}[section]
\theoremstyle{plain}\newtheorem{corollary}{Corollary}[section]

\theoremstyle{definition}\newtheorem{definition}{Definition}[section]

\newcommand{\ol}[1]{\overline{#1}}

\def\th{\theta}
\def\vphi{\varphi}
\def\eps{\varepsilon}

\def\Th{\Theta}
\def\Gma{\Gamma}

\def\ind{\mathbb{I}}

\def\bA{\boldsymbol{A}}

\def\mA{\mathcal{A}}
\def\mB{\mathcal{B}}
\def\mD{\mathcal{D}}
\def\mF{\mathcal{F}}
\def\mG{\mathcal{G}}
\def\mT{\mathcal{T}}
\def\mP{\mathcal{P}}
\def\mQ{\mathcal{Q}}
\def\mX{\mathcal{X}}
\def\mY{\mathcal{Y}}

\def\dX{d_{\mX}}
\def\dY{d_{\mY}}

\def\dGma{d_{\Gma}}
\def\dTh{d_{\Th}}

\def\olvphi{\bar{\vphi}}
\def\olF{\bar{F}}
\def\olP{\bar{P}}

\def\mbP{\mathbb{P}}

\def\bx{\boldsymbol{x}}
\def\bxo{\bx_0}
\def\bxi{\bx_i}
\def\bxj{\bx_j}
\def\bxk{\bx_k}

\def\by{\boldsymbol{y}}
\def\byo{\by_0}
\def\byi{\by_i}
\def\byk{\by_k}
\def\byl{\by_{\ell}}

\def\bz{\boldsymbol{z}}

\def\bth{\boldsymbol{\theta}}
\def\btho{\bth_0}
\def\bthk{\bth_k}

\def\R{\mathbb{R}}
\def\Rp{\R_+}
\def\N{\mathbb{N}}

\newcommand{\set}[1]{\{#1\}}
\newcommand{\Lset}[1]{\left\{#1\right\}}

\newcommand{\bset}[1]{[#1]}

\def\piw{\pi^{\omega}}
\def\pii{\pi_i}
\def\piix{\pi_{i,\bx}}
\def\piiw{\pii^{\omega}}
\def\piixw{\piix^{\omega}}
\def\piixow{\piw_{i,x_0}}

\def\bthx{\bth_{\bx}}
\def\bthw{\bth^{\omega}}
\def\bthi{\bth_i}
\def\bthix{\bth_{i,\bx}}
\def\bthiw{\bthi^{\omega}}
\def\bthixw{\bthix^{\omega}}
\def\bthixow{\bthw_{i,x_0}}

\def\bgma{\boldsymbol{\gamma}}
\def\bgmaw{\bgma^{\omega}}
\def\bgmai{\bgma_i}
\def\bgmak{\bgma_k}
\def\bgmal{\bgma_{\ell}}
\def\bgmao{\bgma_0}

\def\Vi{V_i}
\def\Vix{V_{i,\bx}}
\def\Viw{\Vi^{\omega}}
\def\Vixw{\Vix^{\omega}}

\def\Vx{V_{\bx}}
\def\Vw{V^{\omega}}

\def\Vjx{V_{j,\bx}}

\def\Vjxw{\Vjx^{\omega}}

\def\GX{G_{\mX}}
\def\Gx{G_{\bx}}
\def\Gxo{G_{\bxo}}
\def\Gw{G^{\omega}}
\def\Gxw{\Gx^{\omega}}
\def\Gxow{\Gxo^{\omega}}

\def\mPo{P^0}
\def\mPGam{\mP_{\Gma}}
\def\po{p^0}

\def\qGw{q^{\Gw}}
\def\qPo{q^{\mPo}}
\def\qn{q_n}
\def\qnwx{q_{n,\bx}^{\omega}}
\def\qw{q^{\omega}}

\def\mw{m^{\omega}}
\def\gw{g^{\omega}}

\def\Beta{\textsc{Beta}}

\def\alpx{\alpha_{\bx}}
\def\alpX{\alpha_{\mX}}

\def\parV{\Psi_V}
\def\parTh{\Psi_{\Th}}

\def\ev{\mathbb{E}}
\def\prob{\mathbb{P}}

\def\DDP{\textsc{DDP}}
\def\wDDP{w\textsc{DDP}}
\def\thDDP{\textsc{\(\th\)DDP}}

\def\NTh{N_{\Th}}
\def\Nf{N_f}

\def\Fw{F^{\omega}}

\def\TV{\mathrm{TV}}

\def\LInf{L^{\infty}}
\def\Cb{C_b}

\def\dHel{d_{H}}
\def\dLInf{d_{\LInf}}

\def\muY{\nu_{\mY}}
\def\muX{\nu_{\mX}}

\def\kmix{\psi}

\newcommand{\nrm}[1]{\|#1\|}
\newcommand{\nrmC}[1]{\nrm{#1}_C}
\newcommand{\nrmLInf}[1]{\nrm{#1}_{\LInf}}
\newcommand{\KL}[2]{\textsc{KL}(#1\|\, #2)}

\title{On Dependent Dirichlet Processes for General Polish Spaces}

\author{
        \href{https://orcid.org/0000-0000-0000-0000}{\includegraphics[scale=0.06]{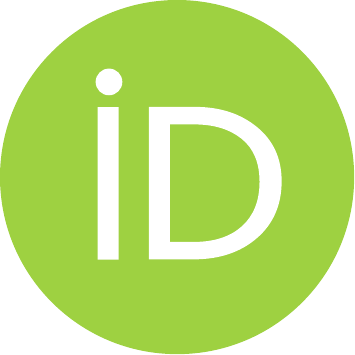}\hspace{1mm}Andr\'es~Iturriaga}\thanks{Andr\'es Iturriaga is Assistant Professor, Departamento de Matem\'atica y Ciencia de la Computaci\'on, Facultad de Ciencia, Universidad de Santiago de Chile, Santiago, Chile.} \\
         Departamento de Matem\'atica y \\Ciencia de la Computaci\'on \\
         Universidad de Santiago de Chile \\
         Santiago, Chile \\
         \texttt{andres.iturriaga@usach.cl} \\
        \And
        \href{https://orcid.org/0000-0002-2533-2509}{\includegraphics[scale=0.06]{orcid.pdf}\hspace{1mm}Carlos~A.~Sing~Long}\thanks{Carlos A. Sing-Long is Assistant Professor at the Institute for Mathematical and Computational Engineering and the Institute for Biological and Medical Engineering, Pontificia Universidad Cat\'olica de Chile, and the Center for the Discovery of Structures in Complex Data and the Center for Cardiovascular Magnetic Resonance.} \\
        Institute for Mathematical and \\ Computational Engineering and Institute for \\ Biological and Medical Engineering \\
        Pontificia Universidad Cat\'olica de Chile \\
        Santiago, Chile \\
        \texttt{casinglo@uc.cl} \\
        \And
        \href{https://orcid.org/0000-0000-0000-0000}{\includegraphics[scale=0.06]{orcid.pdf}\hspace{1mm}Alejandro~Jara}\thanks{Alejandro Jara is Associate Professor, Department of Statistics, Pontificia Universidad Cat\'olica de Chile, and Director, Center for the Discovery of Structures in Complex Data, Casilla 306, Correo 22, Santiago, Chile.} \\
        Department of Statistics\\
        Pontificia Universidad Cat\'olica de Chile\\
        Santiago, Chile \\
        \texttt{atjara@uc.cl} \\
}

\renewcommand{\headeright}{}
\renewcommand{\undertitle}{arXiv preprint}
\renewcommand{\shorttitle}{Dependent Dirichlet processes on Polish spaces}

\hypersetup{
    pdftitle={
                On Dependent Dirichlet Processes for General Polish Spaces
            },
    pdfsubject={
                stat.TH
            },
    pdfauthor={
                Andr\'es~Iturriaga, 
                Carlos~A.~Sing~Long, 
                Alejandro~Jara
            },
    pdfkeywords={
                    Related random probability distributions,
                    Bayesian nonparametrics,
                    Support of random measures,
                    Non standard spaces
                },
}

\begin{document}
\maketitle

\begin{abstract}
    We study Dirichlet process–based models for sets of predictor–dependent probability distributions, where the domain and predictor space are general Polish spaces. We generalize the definition of dependent Dirichlet processes, originally constructed on Euclidean spaces, to more general Polish spaces. We provide sufficient conditions under which dependent Dirichlet processes have appealing properties regarding continuity (weak and strong), association structure, and support (under different topologies). We also provide sufficient conditions under which mixture models induced by dependent Dirichlet processes have appealing properties regarding strong continuity, association structure, support, and weak consistency under i.i.d. sampling of both responses and predictors. The results can be easily extended to more general dependent stick-breaking processes.
\end{abstract}

\keywords{
            Related random probability distributions,
            Bayesian nonparametrics,
            Support of random measures,
            Non standard spaces
        }


\newpage
\section{Introduction}

This paper focuses on the properties of Bayesian nonparametric (BNP) prior distributions for sets of predictor--dependent probability measures, \(\mF=\{F_{\bx}: \bx \in  \mX\}\), where the \(F_{\bx}\)'s are probability measures defined on a common measurable Polish space \((\mY,\mB(\mY))\), indexed by a vector of exogenous predictors \(\bx \in \mX\), with \(\mX\) being also a Polish space, where \(\mB(\mY)\) is the Borel \(\sigma\)-field of \(\mY\). To date, most of the BNP priors to account for the dependence of predictors  on set of probability measures \(\mF\) are generalizations of the Dirichlet process (DP)~\citep{ferguson;73,ferguson;74} and Dirichlet process mixture (DPM) models~\citep{lo;84}. Let \(\mD(\mY)\) be the space of all probability measures, with density w.r.t. Lebesgue measure, defined on \((\mY,\mB(\mY))\). A DPM model is a stochastic process, \(F\), defined on an appropriated probability space \((\Omega,\mF,\mbP)\), such that for almost every \(\omega \in \Omega\), the density function of \(F\) is given by 
\begin{eqnarray}\label{DPM}
    f(\by \mid G(\omega)) =\int_{\Theta} \psi(\by,\bth) G(\omega)(d 
    \bth), \ \ \by \in \mY, 
\end{eqnarray}
where \(\psi(\cdot,\bth)\) is a continuous density function on \((\mY,\mB(\mY))\), for every \(\bth \in \Theta\), and \(G\) is a DP, whose sample paths are probability measures defined on \((\Theta,\mB(\Theta))\), with \(\mB(\Theta)\) being the co\-rres\-ponding Borel \(\sigma\)-field. If \(G\) is DP with
parameters \((M,G_0)\), where \(M\in \R_0^+\) and \(G_0\) is a probability measure on \((\Theta,\mB(\Theta))\), written as \(G\mid M, G_0
\sim \mbox{DP}(M G_0)\), then the trajectories of the process can be a.s. represented by the following stick-breaking representation \citep{sethuraman;94}: 
\(G(B)=\sum_{i=1}^\infty \pi_i \delta_{\bth_i}(B)\), \(B \in \mB(\Theta)\), where \(\delta_{\bth}(\cdot)\) is the Dirac measure at \(\bth\),  \(\pi_i=V_i \prod_{j<i}(1-V_j)\), with \(V_i\mid M \stackrel{iid}{\sim}\mbox{Beta}(1,M)\), and \(\bth_i \mid G_0 \stackrel{iid}{\sim} G_0\). Discussion of properties and applications of DP can be found, for instance, in \cite{mueller;quintana;jara;hanson;2015}.

Most of the BNP extensions incorporate dependence on predictors via the mixing distribution in~\eqref{DPM},  by replacing  \(G\) with \(G_{\bx}\), and the prior specification problem is related to the modeling of the collection of predictor-dependent mixing probability measures \(\set{G_{\bx} : \bx \in \mX}\)~\citep{quintana;etal;2022}. Some of the earliest developments on predictor-dependent DP models appeared in~\cite{cifarelli;regazzini;78}, who defined dependence across related random measures  by introducing a regression for the baseline measure  of marginally DP random measures. A more flexible construction was proposed by~\cite{maceachern;99}, called the dependent Dirichlet process (DDP). The key idea behind the DDP is to create a set of marginally DP random measures and to introduce dependence by modifying the stick-breaking representation of each element in the set. Specifically, \cite{maceachern;99} generalized the stick-breaking representation by assuming \(G_{\bx}(B)=\sum_{i=1}^\infty \pi_i(\bx) \delta_{\bth_i(\bx)}(B)\), \(B \in \mB(\Theta)\), where the point masses \(\bth_i(\bx)\), \(i\in\N\), are independent stochastic processes with index set \(\mX\), and the weights take the form \(\pi_i(\bx)=V_i(\bx) \prod_{j<i}[1-V_j(\bx)]\), with  \(V_i(\bx)\),  \(i\in\N\), being independent stochastic processes with index set \(\mX\) and \(\mbox{Beta}(1,M)\) marginal distribution. We refer the reader to~\cite{barrientos;jara;quintana;2012} for a formal definition of the DDP. Other extensions of the DP for dealing with related probability distributions include the DPM mixture of normals model for the joint distribution of the response and predictors~\citep{mueller;erkanli;west;1996}, the hierarchical mixture of DPM \citep{mueller;quintana;rosner;2004}, the predictor-dependent weighted mixture of  DP~\citep{dunson;pillai;park;2007}, the kernel-stick breaking process~\citep{dunson;park;2008}, the probit-stick breaking processes~\citep{chung;dunson;2009,rodriguez;dunson;2011}, the cluster-\(X\) model~\citep{mueller;quintana;2010}, the PPMx model \citep{mueller;quintana;rosner;2011}, and the general class of stick-breaking processes~\citep{barrientos;jara;quintana;2012}, among many others. Dependent neutral to the right processes and correlated  two-parameter Poisson-Dirichlet processes have been proposed by~\cite{epifani;lijoi;2010} and \cite{leisen;lijoi;2011}, respectively, by considering suitable L\'evy copulas. The general class of dependent normalized completely random measures has been discussed, for instance, by~\cite{lijoi;nipoti;pruenster;2014}. Based on a different formulation of the conditional density estimation problem,~\cite{tokdar;zhu;ghosh;2010} and~\cite{jara;hanson;2011} proposed alternatives to convolutions of dependent stick-breaking approaches.

All of the dependent BNP approaches described previously have focused on responses and parameters defined on Euclidean spaces, and are not appropriate for spaces in which the Euclidean geometry is not valid.  A relevant example of this situation arises in statistical shape analysis, where one of the main spaces of interest is Kendall's shape space~\citep{Kendall_1977}, which can be viewed as the quotient of a Riemannian manifold.  Kendall's space is a natural underlying  space for applications in different areas, including morphometry~\citep{claude2008morphometrics}, meteorology~\citep{mardia2000directional}, archeology~\citep{dryden2016statistical} and genetics~\citep{billera2001geometry}. In these contexts, to employ standard statistical procedures that do not take into account the geometrical properties of the underlying spaces can lead to wrong inferences, which  explains the increasing interest in the development of statistical models for more general Polish spaces.  

To date, the development of statistical procedures for non Euclidean spaces has focussed on the problem of mean estimation~\citep[see, e.g.,][]{bhattacharya2002nonparametric, bhattacharya2003large, bhattacharya2005large}, density estimation~\citep[see, e.g.,][]{pelletier2005kernel,bhattacharya2010nonparametric, bhattacharya2012strong}, and on the regre\-ssion problem for Euclidean responses based on non Euclidean predictors~\citep[see, e.g.,][]{pelletier2006non,bhattacharya2012nonparametric}.~\cite{bhattacharya2002nonparametric, bhattacharya2003large, bhattacharya2005large} studied the problem of nonparametric estimation of a location parameter on a Riemannian manifold, by means of the concept of Fr\'echet mean \citep{Frechet1948}, and derive its asymptotic distribution.~\cite{pelletier2005kernel} study the density estimation problem on a compact Riemannian manifold, by adapting kernel-type techniques.~\cite{bhattacharya2010nonparametric, bhattacharya2012strong} considered the problem of density estimation for data supported on a complete metric space from a BNP point of view.~\cite{pelletier2006non} consider the nonparametric regression problem for a real-valued response and with predictors supported on a closed Riemannian manifold.~\cite{bhattacharya2012nonparametric} study the problem of prediction of a categorial response based on predictors supported on a general manifold.

This work has two parts. In the first, we generalize the definition of a DDP, originally proposed on Euclidean spaces, to more general Polish spaces and establish its basic properties. It is important to stress that the existing DDP definitions given by~\cite{maceachern;99,maceachern;2000} and~\cite{barrientos;jara;quintana;2012} cannot be directly extended to more general spaces because they make use of the concept of cumulative distribution function. In our definition, the existence of a DDP in a general Polish space is justified by the extension of Kolmogorov's consistency theorem proposed by~\cite{neveu1965mathematical}. In Section~\ref{sec:ddp:definition} we define the DDP for general Polish spaces, and introduce some parsimonious variants that share similar properties. In Section~\ref{sec:ddp:continuity} we provide sufficient conditions under which the sample paths of a DDP are continuous under the weak, strong and uniform topologies on the space of probability measures. In Section~\ref{sec:ddp:support} we study the support of a DDP under different topologies, aiming to provide sufficient conditions under which a DDP has full support, or at least contains a sufficiently large set of functions of interest. Finally, in Section~\ref{sec:ddp:association} we provide conditions under which a measure of association is continuous. 

In the second part, we focus on mixtures, providing sufficient conditions under which they have a continuous density with respect to a base measure, and have appealing properties regarding continuity, support, asso\-ciation structure, and consistency under i.i.d. sampling. Section~\ref{sec:mixtures} introduces a framework to study mixtures of a DDP and its variants, with a focus on the regularity of the probability density of the mixture. In Section~\ref{sec:mixtures:continuity} we provide sufficient conditions for a mixture to have continuous sample paths. In particular, the strong regularizing properties of a mixture allows us to show that under mild conditions the sample paths are almost surely uniformly continuous. In Section~\ref{sec:mixtures:support} we leverage the fact that the mixtures we study have a density to study the support in the topologies induced by the Hellinger and \(L^{\infty}\) distances, and by the Kullback-Leibler divergence. In Section~\ref{sec:mixtures:association} we briefly discuss the regularity of some measures of association. Finally, in Section~\ref{sec:mixtures:posteriorConsistency} we discuss the conditions under which the mixture satisfies posterior consistency. We conclude the article with some brief remarks.

\section{Preliminaries}
\label{sec:preliminaries}

In this work, we let \((\Omega,\mF,\mbP)\) be a complete probability space. If \(S\) is a  measurable space and \(X:\Omega\to S\) is a \(S\)-valued random variable, we usually write \(X^{\omega}\) or \(X(\omega)\) to denote its value at \(\omega\). If \(\set{X_t:\, t\in T}\) is a process of \(S\)-valued random variables on \(T\) we write \(X_t^{\omega} := X_t(\omega) := X(t,\omega)\) to denote its values at \((t,\omega)\). In particular, outcomes always appear as superscripts and as the last argument of a function. 

If \(\mX,\mY\) are two sets we denote a generic element with the same letter in lowercase boldface. If \(F:\mX\times \mY \to S\) then we write \(F_{\bx}\) for a fixed \(\bx\in X\) to denote the function \(F_{\bx} : \mY\to S\) defined as \(F_{\bx}(\by) = F(\bx,\by)\). Similarly, we denote \(F_{\by}\) for fixed \(\by\in \mY\) the function defined as \(F_{\by}(\bx) = F(\bx,\by)\).

We recall that a Polish space is a separable and completely metrizable topological space. The sets \(\mX, \mY\) and \(\Theta\) are always assumed to be Polish spaces, with complete metrics \(\dX, \dY\) and \(\dTh\) respectively. 

The set \(\Th\) represents the set of parameters onto which we define a prior. Hence, we endow \(\Th\) with the Borel \(\sigma\)-algebra \(\mB(\Th)\) and let \(\mP(\Th)\) be the space of probability measures on \((\Th,\mB(\Th))\). Since \(\Th\) is Polish, the elements of \(\mP(\Th)\) are regular~\cite{Cohn2013}, i.e., for any \(P\in \mP(\Th)\) and \(B\in \mB(\Th)\) we have
\begin{align*}
    P(B) &= \inf\set{P(U):\, B\subset U,\, \mbox{\(U\) open}}\\
    P(B) &= \inf\set{P(K):\, B\supset K,\, \mbox{\(K\) compact}}.
\end{align*}
In particular, finite collections of elements in \(\mP(\Th)\) are always tight. We let \(\Cb(\Th)\) be the space of real-valued, bounded continuous functions on \(\Th\) endowed with the norm \(\nrmC{f} := \sup\set{|f(\bth)|:\, \bth\in\Th}\).

If \(n\) is a positive integer, we let \(\bset{n} := \set{1, \ldots, n}\). Finally, we use the abbreviations ``a.s.'' for {\em almost surely}, ``a.e.'' for {\em almost everywhere} and ``i.i.d.'' for {\em independent and identically distributed}.

\section{Dependent Dirichlet processes on Polish spaces}
\label{sec:ddp:definition}

Dependent Dirichlet processes (DDP) are a class of \(\mP(\Th)\)-valued stochastic processes on \(\mX\) defined on \((\Omega,\mF,\mbP)\). If \(\set{\Gx:\, \bx\in\mX}\) is a DDP then
\[
\Gxw = \sum_{i=1}^{\infty} \piixw \delta_{\bthixw}
\]
a.s. for a sequence \(\set{\set{\piix:\, \bx\in \mX}}_{i\in \N}\) of processes such that \(\piix \geq 0\) for every \(i\in\N\) and \(\bx\in \mX\) and \(\sum_{i\in \N} \piix \equiv 1\) a.s., and a sequence \(\set{\set{\bthix:\, \bx\in \mX}}_{i\in\N}\) of i.i.d. \(\Th\)-valued processes. The distinctive feature of the DDP is that the processes \(\set{\set{\piix:\, \bx\in \mX}}_{i\in \N}\) are defined in terms of a stick-breaking process. Let \(\set{\set{\Vix:\, \bx\in \mX}}_{i\in\N}\) be a sequence of i.i.d. processes with \(V_{i,\bx} \sim \Beta(1,\alpx)\)  for some \(\alpx\in \Rp\) and any \(\bx\in \mX\). Then, the stick-breaking process asociated to \(\set{\set{\Vix:\, \bx\in \mX}}_{i\in\N}\) is
\begin{equation}
    \label{eq:stickBreaking}
    \piixw := \begin{cases}
        \Vixw & i = 1\\
        \Vixw \prod_{j=1}^{i-1}(1 - \Vjxw) & i > 1.
    \end{cases}
\end{equation}
We now generalize the definition in~\cite{barrientos;jara;quintana;2012} to Polish spaces.

\begin{definition}
    \label{def:ddp}
    Let \(\GX\) be a \(\mP(\Th)\)-valued stochastic process on \(\mX\) and defined on \((\Omega,\mF,\mbP)\) given by \(\set{\Gx: \bx\in\mX}\). Suppose the following conditions hold: 
    \begin{enumerate}[leftmargin=24pt]
        \item{There exists a sequence \(\set{\set{\Vix:\, \bx\in\mX}}_{i\in\N}\) of separable i.i.d. processes, with a law characterized by a finite-dimensional parameter \(\parV\), and with marginal distribution \(\Beta(1,\alpx)\) for some \(\alpx \geq 0\) for any \(\bx\in \mX\). 
        }
        \item{There exists a sequence \(\set{\set{\bthix:\, \bx\in\mX}}_{i\in\N}\) of i.i.d. processes, with a law characterized by a finite-dimensional parameter \(\parTh\), and with marginal distribution \(\Gx^0\in \mP(\Th)\) for any \(\bx\in \mX\).
        }
        \item{There exists a null set \(N\subset \Omega\) such that for every \(\bx\in \mX\), \(B\in \mB(\Th)\) and \(\omega\in \Omega\setminus N\)
            \begin{equation}
                \label{eq:ddpRepresentation}
                \Gxw(B) = \sum_{i=1}^{\infty} \piixw \delta_{\bthixw}(B)
            \end{equation}
            where the sequence \(\set{\set{\piix:\, \bx\in\mX}}_{i\in\N}\) is given by the stick-breaking process in~\eqref{eq:stickBreaking}.
        }
    \end{enumerate}
    Then, the process \(\GX\) is called a {\em dependent Dirichlet process (DDP)} with parameters \((\parV,\parTh)\) and denoted as \(\GX \sim \DDP(\parV,\parTh)\). If \(\GX\) is a DDP we write \(\alpX:=\set{\alpx:\, \bx\in\mX}\) and \(\GX^{0} := \set{\Gx^{0}:\, \bx\in\mX}\).
\end{definition}

It is of interest to determine when a DDP can be constructed from a prescribed \(\alpX\) and \(\GX^0\). First, we can construct a sequence of processes \(\set{\set{\bthix:\, \bx\in\mX}}_{i\in\N}\) satisfying Condition 2 in Definition~\ref{def:ddp} from a prescribed \(\GX^0\) by constructing first a process \(\set{\bthx:\, \bx\in X}\) with marginals \(\set{\Gx^0:\, \bx\in X}\) using an extension of Kolmogorov's consistency theorem to Polish spaces ~\citep[see Section~III.3 in][]{neveu1965mathematical}. Remark that in this case the separability of the resulting process is not required.

Second, we can construct a sequence of processes \(\set{\set{\Vix:\, \bx\in\mX}}_{i\in\N}\) satisfying Condition 1 in Definition~\ref{def:ddp} from a prescribed \(\alpX\) using Kolmogorov's consistency theorem and a consistent families of copula functions as in~\citep{barrientos;jara;quintana;2012}. However, Definition~\ref{def:ddp} requires this process to be in addition separable. This ensures that the set of outcomes for which the left-hand side in~\eqref{eq:ddpRepresentation} is not a probability measure is a measurable set. In fact, if for every \(\bx\in\mX\) we define the event
\[
N_{\bx} := \Lset{\omega\in \Omega:\, \sum_{i=1}^{\infty}\piixw < 1} = \Lset{\omega\in \Omega:\, \sum_{i=1}^{\infty}\log(\ev(1 - \Vixw)) = -\infty},
\]
then the left-hand side in~\eqref{eq:ddpRepresentation} fails to be a probability measure if
\[
\omega\in \bigcup_{\bx\in\mX} N_{\bx}.
\]
However, the above set may not be measurable. The separability of the processes ensures there exists a countable set \(\mX_V \subset \mX\) such that
\[
\bigcup_{\bx\in\mX} N_{\bx} = \bigcup_{\bx\in\mX_V} N_{\bx}.
\]
Therefore, 
\[
\mbP\left(\bigcup_{\bx\in \mX} N_{\bx}\right) \leq \sum_{\bx\in \mX_V} \mbP(N_{\bx}).
\]
whence it suffices to ensure \(\mbP(N_{\bx}) = 0\) for every \(\bx\in \mX_V\). Since any stochastic process on a separable space with a.s. continuous sample paths is separable, an alternative, standard way to build a separable process on \(\mX\) with \(\Beta(1,\alpx)\) marginal distributions is to transform a real-valued process over \(\mX\) with a.s. continuous sample paths using the quantile function of the beta distribution~\citep{maceachern;99,maceachern;2000}. Specifically, let \(\set{Z_{\bx}:\bx\in\mX}\) be a real-valued stochastic process with a.s. continuous sample paths, and with continuous cumulative distribution function \(F_{Z,\bx}\) at \(\bx\in \mX\). Let \(F_{B,\bx}\) denote the cumulative distribution function of a \(\Beta(1,\alpx)\) distribution. Then, the process \(\set{\bx:\, \bx\in \mX}\) defined as
\[
\Vw_{\bx} := F_{B,\bx}^{-1}(F_{Z,\bx}(Z_{\bx}))
\]
is separable, has a.s. continuous sample paths, and has marginal distribution \(\Beta(1,\alpx)\). The choice for the base process \(\set{Z_{\bx}:\,\bx\in\mX}\) depends on the structure of \(\mX\). When it is a Gaussian process, there are known conditions under which it admits a modification with a.s. continuous sample paths. For example, Theorem~2.3.1 in~\cite{khoshnevisan2002multiparameter} and~\cite{Preston1972} provide sufficient conditions for the existence of such a modification when \(\mX\) is compact. Another possibility when \(\mX\) is a manifold is to use diffusion processes~\citep{hsu2002stochastic}.

Although the DDP is flexible, it is of interest to define parsimonious variants of the DDP for which either the support points or the weights are independent of \(\bx\). These parsimonious versions should be understood not only as simplifications of the DDP, but also as useful models with comparative advantages over the DDP that make them suitable in specific settings. The first parsimonious version removes the dependence of the weights on \(\bx\). 

\begin{definition}
    \label{def:wddp}
    Let \(\GX\) be a \(\mP(\Th)\)-valued stochastic process on \(\mX\) and defined on \((\Omega,\mF,\mbP)\) given by \(\set{\Gx: \bx\in\mX}\). Suppose the following conditions hold:
    \begin{enumerate}[leftmargin=24pt]
        \item{There exists a sequence \(\set{\Vi}_{i\in\N}\) of i.i.d. processes, with a common law \(\Beta(1,\alpha)\) for some \(\alpha \geq 0\). 
        }
        \item{There exists a sequence \(\set{\set{\bthix:\, \bx\in\mX}}_{i\in\N}\) of i.i.d. processes, with a law characterized by a finite-dimensional parameter \(\parTh\), and marginal distribution \(\Gx^0\in \mP(\Th)\) for any \(\bx\in \mX\).
        }
        \item{There exists a null set \(N\subset \Omega\) such that for every \(\bx\in \mX\), \(B\in \mB(\Th)\) and \(\omega\in \Omega\setminus N\)
            \[
            \Gxw(B) = \sum_{i=1}^{\infty} \piiw \delta_{\bthixw}(B)
            \]
            where the sequence \(\set{\pii}_{i\in \N}\) of random variables is defined as
            \begin{equation}
                \piiw := \begin{cases}
                    \Viw & i = 1\\
                    \Viw \prod_{j=1}^{i-1}(1 - \Vw_{j}) & i > 1.
                \end{cases}
            \end{equation}
        }
    \end{enumerate}
    Then, the process \(\GX\) is called a {\em single-weights dependent Dirichlet process (\(w\)DDP)} with parameters \((\alpha,\parTh)\) and denoted as \(\GX \sim \wDDP(\alpha,\parTh)\). If \(\GX\) is a \(w\)DDP we write \(G_{\mX}^{0} := \set{\Gx^{0}:\, \bx\in\mX}\).
\end{definition}

One of the advantages of the single-weights DDP is that it avoids any difficulty that may arise in the construction of a separable process \(\set{\Vix:\, \bx\in\mX}\) with marginals \(\Beta(1,\alpx)\) for \(\bx\in\mX\). However, it implicitly assumes we may still be able to construct the process \(\set{\bthx:\, \bx\in\mX}\) given the marginals \(\set{\Gx^0:\, \bx\in\mX}\). Hence, this variant is desirable when the structure of \(\mX\) is complex relative to that of \(\Th\).

The second parsimonious variant of the DDP relaxes the dependence of the support points on \(\bx\).

\begin{definition}
    \label{def:thddp}
    Let \(\GX\) be a \(\mP(\Th)\)-valued stochastic process on \(\mX\) and defined on \((\Omega,\mF,\mbP)\) given by \(\set{\Gx: \bx\in\mX}\). Suppose the following conditions hold:
    
    \begin{enumerate}[leftmargin=24pt]
        \item{There exists a sequence \(\set{\set{\Vix:\, \bx\in\mX}}_{i\in\N}\) of separable i.i.d. processes, with a law characterized by a finite-dimensional parameter \(\parV\), and with marginal distribution \(\Beta(1,\alpx)\) for some \(\alpx \geq 0\) for any \(\bx\in \mX\). 
        }
        \item{There exists a sequence \(\set{\bthi}_{i\in\N}\) of i.i.d. \(\Theta\)-valued random variables, with common law \(G^0\in\mP(\Th)\).
        }
        \item{There exists a null set \(N\subset \Omega\) such that for every \(\bx\in \mX\), \(B\in \mB(\Th)\) and \(\omega\in \Omega\setminus N\)
            \[
            \Gxw(B) = \sum_{i=1}^{\infty} \piixw \delta_{\bthiw}(B)
            \]
            where the sequence \(\set{\set{\bthix:\, \bx\in\mX}}_{i\in\N}\) is given by~\eqref{eq:stickBreaking}.
        }
    \end{enumerate}
    Then, the process \(\GX\) is called a {\em single-atoms dependent Dirichlet process (\(\theta\)DDP)} with parameters \((\parV,G^0)\) and denoted as \(\GX \sim \thDDP(\parV,G^0)\). If \(\GX\) is a \(\theta\)DDP we write \(\alpX := \set{\alpx:\, \bx\in\mX}\).
\end{definition}

A single-atoms DDP can be easier to construct in situations where the structure of \(\Th\) is complex. As matter of fact, the construction of a stochastic process \(\set{\bthx:\, \bx\in \mX}\) can be difficult for general spaces \(\mX\) and \(\Th\), particularly when it needs to satisfy additional properties, such as a.s. continuity of its sample paths. Some specific constructions are available in particular cases of interest. For example, if \(\Th\) is Kendall's planar shape space \citep[see, e.g.][]{Kendall1984}, diffusion processes have been proposed under two different approaches: (i) directly on the landmarks, on the space of configurations, which is referred to as Euclidean diffusion of shape \citep[see, e.g.,][]{Kendall_1977, kendall1988symbolic, kendall1990diffusion, le1991stochastic}, and (ii) directly on $\Theta$, via infinitesimal generators \citep[see, e.g.,][]{le1994brownian, kendall1998diffusion, ball2008brownian, golalizadeh2010useful} and the solution of partial differential equations \citep[see, e.g.,][]{hsu2002stochastic}.

Although neither the DDP nor its variants require the continuity of the sample paths of the processes \(\set{\set{\Vix:\, \bx\in\mX}}_{i\in\N}\)  or \(\set{\set{\bthix:\, \bx\in\mX}}_{i\in\N}\). However, this additional condition endows the DDP and its variants of some desirable properties that are the focus of this work.

\begin{definition}
    \label{def:continuousDDP}
    Let \(\GX\) be a DDP or any of its variants. 
    
    \begin{enumerate}[leftmargin=24pt]
        \item{If \(\GX\) is a DDP we say it is a {\em continuous parameter DDP} if both the separable i.i.d. processes \(\set{\set{\Vix:\,\bx\in\mX}}_{i\in\N}\) and the i.i.d. processes \(\set{\set{\bthix:\,\bx\in\mX}}_{i\in\N}\) have a.s. continuous sample paths.
        }
        \item{If \(\GX\) is a \(w\)DDP we say it is a {\em continuous parameter \(w\)DDP} if the i.i.d. processes \(\set{\set{\bthix:\,\bx\in\mX}}_{i\in\N}\) have a.s. continuous sample paths.
        }
        \item{If \(\GX\) is a \(\th\)DDP we say it is a {\em continuous parameter \(\th\)DDP} if the separable i.i.d. processes \(\set{\set{\Vix:\,\bx\in\mX}}_{i\in\N}\) have a.s. continuous sample paths.
        }
    \end{enumerate}
\end{definition}

\section{Properties of dependent Dirichlet processes}

\subsection{Continuity}
\label{sec:ddp:continuity}

The continuity of the sample paths of a process is an important property that also plays a critical role in statistical applications. On one hand, it allows us to determine suitable topologies for the spaces containing the sample paths. On the other, it ensures that the process will be able to borrow strength across sparse data sources regarding the predictors. In fact, continuity eliminates the need for replicates of the responses at every value of the predictors to obtain adequate estimates of the predictor-dependent probability distributions~\citep[see, e.g.,][]{barrientos;jara;quintana;2017, wehrhahn;barrientos;jara;2022}.

For the DDP and its variants, the sample paths are functions from \(\mX\) into \(\mP(\Th)\) and their continuity depends on the topologies on these spaces. Although we always assume \(\mX\) is endowed with its metric topology, there are several standard choices for the topology on \(\mP(\Th)\). Although we mostly focus on the weak topology, we will also study the effect of considering the strong (or weak-*) and uniform (or norm, or total variation) topologies on \(\mP(\Th)\).  

For the weak topology on \(\mP(\Th)\) we denote \(C_W(\mX,\mP(\Th))\) the space of {\em weakly continuous functions} from \(\mX\) into \(\mP(\Th)\). These are the functions \(P:\mX\to \mP(\Th)\) such that for any \(f\in \Cb(\Th)\) the function
\[
F(\bx) = \int_\Theta f(\bth) dP(\bth)
\]
is continuous on \(\mX\). The following theorem shows that when the underlying stochastic processes have a.s. continuous sample paths, the DDP and its variants have a.s. weakly continuous sample paths. We defer its proof to Appendix~\ref{thm:continuity:weak:proof}.

\begin{theorem}\label{thm:continuity:weak}
    Let \(\GX\) be a \(\mP(\Th)\)-valued process. Suppose that \(\GX\) is a continuous parameter DDP, a continuous parameter \(w\)DDP or a continuous parameter \(\th\)DDP. Then for a.e. \(\omega\in\Omega\)
    \begin{equation*}
        \forall\, \bx_0\in \mX,\, f\in C_b(\Th):\,\, \lim_{\bx\to \bx_0}\, \int f(\bth)d\Gxw(\bth) = \int f(\bth)\, d\Gw_{\bx_{0}}(\bth).
    \end{equation*}
\end{theorem}

Consequently, to construct a DDP or any of its variants with a.s. weakly continuous sample paths it suffices to construct suitable continuous processes on \(\mX\). As discussed earlier, a process \(\set{\Vx:\, \bx\in \mX}\) with the desired properties can be constructed from a real-valued base process \(\set{Z_{\bx}:\, \bx\in \mX}\) on \(\mX\) with a.s. continuous sample paths which can be, for instance, a suitable Gaussian process. 

To our knowledge, there is no similar, standard way to construct a process \(\set{\bthx:\, \bx\in \Theta}\) with the desired properties for general \(\mX\) and \(\Th\). When \(\mX = \R^d\) there are well-known sufficient conditions that ensure the there exists a modification of a \(\Th\)-valued process \(\set{\bthx:\, \bx\in \mX}\) with a.s. continuous sample paths~\cite[Theorem~2.23]{kallenberg1997}: this modification exists if there exists exponents \(\alpha,\gamma > 0\) and a constant \(C > 0\) such that
\begin{equation}
    \forall\, \bx,\bx'\in \mX:\quad \ev(\dTh(\bthx,\bth_{\bx'})^{\alpha})\leq C_{\theta}\dX(\bx,\bx')^{d + \gamma}
    \label{eq:3.5}
\end{equation}
where \(\dTh\) is a {\em complete} metric on \(\Theta\). This result can be applied to a Polish space \(\mX\) that is homeomorphic to \(\R^d\) for some \(d\). 

When \(\Th\) is a Riemannian manifold, processes with a.s. sample paths can be defined through diffusion processes  \citep[see, e.g.,][]{hsu2002stochastic}. When \(\Th\) is also a quotient space \(\Theta=\mA/\mG\) for a locally compact space \(\mA\) and a group \(\mG\), then a \(\Th\)-valued process with a.s. continuous sample paths can be constructed as follows. Let \(\mQ:\mA\rightarrow\Th\) be the canonical quotient map and let \(\bA:\mX\times\Omega\rightarrow \mA\) be a process with a.s. continuous sample paths. Since the canonical quotient map is continuous, the process \(\set{\mQ(\bA(\bx)):\, \bx\in \mX}\) has a.s. continuous sample paths.

As mentioned earlier, the variants of the DDP should not be thought only as simplifications of the DDP but also as processes with distinct properties. Endowing \(\mP(\Th)\) with the strong topology already allows us to distinguish the properties of these processes. Let \(C_S(\mX,\mP(\Th))\) be the vector space of {\em strongly continuous functions} from \(\mX\) into \(\mP(\Th)\). These are the functions \(P:\mX\to \mP(\Th)\) such that for any \(f\in \LInf(\Th)\) the function
\[
F(\bx) = \int_\Theta f(\bth) dP(\bth)
\]
is continuous on \(\mX\). The following theorem shows that, under the same hypothesis of Theorem~\ref{thm:continuity:weak}, the \(\th\)DDP has a.s. strongly continuous sample paths. Although this is really a corollary of Theorem~\ref{thm:continuity:uniform}, we present this statement independently for clarity. 

\begin{theorem}\label{thm:continuity:strong}
    Let \(\GX\) be a \(\mP(\Th)\)-valued process. Suppose \(\GX\) is a continuous parameter \(\th\)DDP with \(\GX\sim \thDDP(\parV,G^0)\). Then for a.e. \(\omega\in\Omega\)
    \begin{equation*}
        \forall\, \bx_0\in \mX,\, f\in \LInf(\Th):\,\, \lim_{\bx\to \bx_0}\, \int f(\bth)d\Gxw(\bth) = \int f(\bth)\, d\Gw_{\bx_{0}}(\bth).
    \end{equation*}
\end{theorem}

A natural question is whether the DDP or the \(w\)DDP can have a.s. strongly continuous sample paths under similar assumptions. To our knowledge, this cannot be the case unless substantially stronger conditions are imposed on these processes. We defer the proof of the following theorem to Appendix~\ref{thm:continuity:notStrongForDDP:proof}.

\begin{theorem}\label{thm:continuity:notStrongForDDP}
    Let \(\GX\) be a \(\mP(\Th)\)-valued process. Suppose that \(\GX\) is a continuous parameter DDP or a continuous parameter \(w\)DDP. Let \(\bxo\in \Omega\). If for a.e. \(\omega\in \Omega\) we have
    \[
    \forall\, f\in \LInf(\Th):\,\, \lim_{\bx\to\bxo}\, \int f(\bth) d\Gxw(\bth) = \int f(\bth) d\Gxow(\bth)
    \]
    then for a.e. \(\omega\in\Omega\) there exists an open neighborhood \(U^\omega\subset \mX\) of \(\bxo\) and at least one \(i^\omega\in\N\) such that \(\bthw_{i^\omega}\) is constant on \(U^\omega\). 
\end{theorem}

Since for the DDP and \(w\)DDP the sequence of processes \(\set{\set{\bthix:\, \bx\in \mX}}_{i\in \N}\) is independent and identically distributed, the above implies that when the process has a.s. strongly continuous paths at \(\bx_0\) the process \(\set{\bth_{1,\bx}:\,\bx\in\mX}\) must have a.s. constant sample paths near \(\bx_0\). Although this suggests that the main issue is the behavior of the atoms themselves, the proof shows that the main issue is the independence between the processes \(\set{\set{\Vix:\, \bx\in \mX}}_{i\in \N}\) and \(\set{\set{\bthix:\, \bx\in \mX}}_{i\in \N}\). We conjecture that for the DDP and \(w\)DDP to have a.s. strongly continuous paths, it is necessary to introduce dependence between these processes. 

Since the DDP and \(w\)DDP do not have a.s. strongly continuous paths, it is clear they will not have a.s. continuous paths with respect to stronger topologies on \(\mP(\Th)\). However, for the \(\th\)DDP we can strengthen the topology on \(\mP(\Th)\) while preserving this property. Consider the uniform topology on \(\mP(\Th)\) and denote as \(C_U(\mX,\mP(\Th))\) the set of {\em uniformly continuous functions} from \(\mX\) into \(\mP(\Th)\). The total variation norm for any signed finite measure \(Q\) on \((\Theta,\mB(\Th))\) is defined as
\[
\nrm{Q}_{\TV} := \sup\Lset{\int f(\bth)dQ(\bth):\, f\in \LInf(\Th),\, \nrm{f}_{\LInf} \leq 1}.
\]
Then the elements of \(C_U(\mX,\mP(\Th))\) are the functions \(P:\mX\to\mP(\Th)\) such that for any \(\bx_0\in \mX\)
\[
\lim_{\bx\to\bx_0} \nrm{P_{\bx} - P_{\bx_0}}_{\TV} = 0.
\]
By choosing indicator functions, it is clear the above is equivalent to
\[
\forall\, \bx_0\in \mX:\quad \lim_{\bx\to\bx_0}\sup_{B\in\mB(\Th)}|\Gxw(B) - \Gxow(B)| = 0.
\]
which is an expression that is typically more interpretable in statistical applications. For the uniform topology we can show that, under the same assumptions of Theorem~\ref{thm:continuity:strong}, the \(\th\)DDP has a.s. uniformly (or norm) continuous sample paths. This is also known as continuity in total variation. Its proof is deferrered to Appendix~\ref{thm:continuity:uniform:proof}.

\begin{theorem}\label{thm:continuity:uniform}
    Let \(\GX\) be a continuous parameter \(\th\)DDP. Then, for a.e. \(\omega\in\Omega\),
    \[
    \forall\, \bx_0\in \mX:\,\, \lim_{\bx\to \bx_0}\, \nrm{\Gxw - \Gw_{\bx_{0}}}_{\TV} = 0.
    \]
\end{theorem}

\subsection{Support}
\label{sec:ddp:support}

The sample paths of the DDP and its variants are elements of suitable spaces of functions from \(\mX\) into \(\mP(\Th)\). It is of interest to characterize the size, in a suitable sense, of the set containing the sample paths. This leads us to the concept of support. In applications in statistics, a large support is an important and basic property that any BNP model should possess. In fact, it is a minimum requirement, and almost a ``necessary'' property, for a BNP model to be considered ``nonparametric.'' This property is also important because it typically is a necessary condition for the consistency of the posterior distribution. In such settings, the full support of the prior implies that the prior probability model is flexible enough to generate sample paths sufficiently close to any element of the parameter space. 

Given a topology \(\mT\) on \(\mP(\Th)^{\mX}\) the support of a process is the smallest closed set, in the sense of set inclusion, such that the probability it contains a sample path is equal to one. We say it has full support, or that the support is full, if it is equal to \(\mP(\Th)^{\mX}\). When the support is not full, its complement is a non-empty open set. In particular, it contains a point with a neighborhood that is disjoint from the support for which the probability of containing a sample path is zero. Consequently, to prove that a process has full support with respect to \(\mT\) it suffices to show that the probability that any element of a neighborhood basis contains a sample path is positive. 

We characterize the support of the DDP and its variant for common choices of \(\mT\) starting from the weakest. We consider first the (weak) product topology, or pointwise topology~\cite{Kelley1975}, on \(\mP(\Th)^{\mX}\). For reasons that shall be clear soon, we call it the {\em product-weak topology}. In this topology, a neighborhood basis at \(P^0\in \mP(\Th)^{\mX}\) is given by sets of the form
\[
\Lset{P\in \mP(\Th)^{\mX}:\,\, \left|\int_\Theta f_{i,j}(\bth) dP_{\bx_j}(\bth) - \int_\Theta f_{i,j}(\bth) dP^0_{\bx_j}(\bth)\right| < \eps_{i,j},\,\, i,j\in \bset{n}}
\]
for \(\eps_{1, 1},\ldots, \eps_{n,n} > 0\), \(\bx_1,\ldots, \bx_n\in \mX\) and \(f_{1,1},\ldots, f_{n,n}\in \Cb(\Th)\). The following theorem shows the DDP and its variants have full support with respect to this topology. We defer its proof to Appendix~\ref{thm:support:product-weak:proof}.

\begin{theorem}\label{thm:support:product-weak}
    Let \(\GX\) be a \(\mP(\Th)\)-valued process on \(\mX\). Suppose that one of the following assertions holds.
    \begin{enumerate}[leftmargin=24pt]
        \item{\(\GX \sim \DDP(\parV,\parTh)\), for any \(\bx_1,\ldots, \bx_n\in \mX\) the law of the random vector
            \[
            (\bth_{1,\bx_1}, \ldots, \bth_{1,\bx_n})
            \]
            has full support on \(\Th^n\), and the law of the random vector 
            \[
            (V_{1,\bx_1}, \ldots, V_{1,\bx_n})
            \]
            has full support on \([0,1]^n\).
        }
        \item{\(\GX\sim \wDDP(\alpha,\parTh)\), for any \(\bx_1,\ldots, \bx_n\in \mX\) the law of the random vector
            \[
            (\bth_{1,\bx_1}, \ldots, \bth_{1,\bx_n})
            \]
            has full support on \(\Th\) and the law of the random variable \(V_1\) has full support on \([0,1]\).
        }
        \item{\(\GX\sim \thDDP(\parV,G^0)\), \(G^0\) has full support on \(\Th\), and or any \(\bx_1,\ldots, \bx_n\in \mX\) the law of the random vector 
            \[
            (V_{1,\bx_1}, \ldots, V_{1,\bx_n})
            \]
            has full support on \([0,1]^n\).
        }
    \end{enumerate}
    Then
    \[
    \mbP\left(\Lset{\omega\in\Omega:\,\, \left|\int_\Theta f_{i,j}(\bth) d\Gw_{\bxj}(\bth) - \int_\Theta f_{i,j}(\bth) dP^0_{\bxj}(\bth)\right| < \eps_{i,j},\,\, i,j\in \bset{n}}\right) > 0
    \]
    for \(\eps_{1, 1},\ldots, \eps_{n,n} > 0\), \(\bx_1,\ldots, \bx_n\in \mX\) and \(f_{1,1},\ldots, f_{n,n}\in \Cb(\Th)\). In consequence, the process has full support on \(\mP(\Th)^{\mX}\) endowed with the
    product-weak topology. 
\end{theorem}

The product-weak topology is often too coarse in statistical applications. The topology we consider next is the compact-open topology on \(\mP(\Th)^{\mX}\)~\cite{Kelley1975}. In this topology, a neighborhood basis at \(P^0\in \mP(\Th)^{\mX}\) is given by sets of the form
\begin{equation}
    \label{eq:support:compact-weak:basis}
    \Lset{P\in \mP(\Th)^{\mX}:\,\, \sup_{\bx\in K}\, \left|\int_\Theta f_i(\bth) dP_{\bx}(\bth) - \int_\Theta f_i(\bth) dP^0_{\bx}(\bth)\right| < \eps_i,\, i\in \set{1,\ldots, n}}
\end{equation}
for \(\eps_1,\ldots,\eps_n>0\), \(f_1,\ldots, f_n\in C_b(\Th)\) and \(K\subset \mX\) compact. As this topology is stronger, it is unlikely the DDP and its variants will still have full support on \(\mP(\Th)^{\mX}\). For this reason, we determine whether the support contains a subset of \(\mP(\Th)^{\mX}\) of functions of interest. For the weak topology on \(\mP(\Th)\), we consider the weakly continuous functions from \(\mX\) into \(\mP(\Th)\). 

If the support of a process does not contain \(C_W(\mX,\mP(\Th))\) then there is at least one \(P^0\in C_W(\mX,\mP(\Th))\) in the complement of the support. Since this set is open, it contains at least one set of the form~\eqref{eq:support:compact-weak:basis}. The following result shows that, under mild conditions, the support of both the DDP and \(\th\)DDP contains \(C_W(\mX,\mP(\Th))\). We defer its proof to Appendix~\ref{thm:support:compact-weak:proof}.

\begin{theorem}\label{thm:support:compact-weak}
    Let \(\GX\sim \DDP(\parV,\parTh)\) or \(\GX\sim \thDDP(\parV,G^0)\). Suppose the following conditions hold:
    \begin{enumerate}[leftmargin=24pt]
        \item{The processes \(\set{\set{\Vix:\, \bx\in \mX}}_{i\in\N}\) have a.s. continuous sample paths.
        }
        \item{For any \(\eps > 0\), continuous function \(h:\mX\to [0, 1]\) and \(K\subset \mX\) compact we have
            \begin{equation}
                \label{eq:vprocess:uniformApproximation}
                \mbP\left(\Lset{\omega\in \Omega:\, \sup_{\bx\in K} |\Vw_{1,\bx} - h(\bx)| < \eps}\right) > 0.
            \end{equation}
        }
        \item{If \(\GX\sim \DDP(\parV,G^0)\) then \(G^0\) has full support on \(\Theta\). 
        }        
        \item{If \(\GX\sim \DDP(\parV,\parTh)\) then for any open \(U\in\mB(\Th)\) and \(K\subset \mX\) compact we have
            \[
            \mbP(\set{\omega\in \Omega:\, \bthw_{1,\bx} \in U,\, \forall\, \bx\in K}) > 0.
            \]
        }
    \end{enumerate}
    Then, for any \(P^0\in C_W(\mX,\mP(\Th))\) we have
    \[
    \mbP\left(\Lset{\omega\in \Omega:\,\, \sup_{\bx\in K}\, \left|\int_\Theta f_i(\bth) d\Gxw(\bth) - \int_\Theta f_i(\bth) dP^0_{\bx}(\bth)\right| < \eps_i,\, i\in \set{1,\ldots, n}}\right) > 0
    \]
    for \(\eps_1,\ldots,\eps_n>0\), \(f_1,\ldots, f_n\in C_b(\Th)\) and \(K\) compact. In consequence, the support of the process in \(\mP(\Th)^{\mX}\) endowed with the compact-weak topology contains \(C_W(\mX,\mP(\Th))\). 
\end{theorem}

To construct a process \(\set{\Vx:\, \bx\in \mX}\) satisfying~\eqref{eq:vprocess:uniformApproximation} we use the same construction outlined in Section~\ref{sec:ddp:definition}. Let \(\set{Z_{\bx}:\, \bx\in \mX}\) be a Gaussian process with mean function \(\mu\) and covariance kernel \(\sigma\) and a.s. continuous sample paths. We define at any given \(\bx\in \mX\) the functions
\[
F_{B,\alpx}^{-1}(t) = 1 - (1 - t)^{1/\alpx}\quad\mbox{and}\quad F_{Z,\bx}(z) = \Phi\left(\frac{z - \mu_{\bx}}{\sigma_{\bx,\bx}}\right).
\]
Observe that
\[
|F_{B,\alpx}^{-1}(u) - F_{B,\alpx}^{-1}(v)| \leq \frac{1}{\alpx} |u-v|.
\]
If we let \(\Vix = F_{B,\alpx}^{-1}\circ F_{Z,\bx}(Z_{\bx})\) then \(\Vix\) has a.s. continuous sample paths and \(\Beta(1,\alpx)\) marginal distributions. Let \(K\subset \mX\) be compact, and let \(h:K\to (0, 1)\) be continuous. Then
\[
|\Vix - h(\bx)| \leq \frac{1}{\alpx}|F_{Z,\bx}(Z_{\bx}) - F_{B,\alpx}(h(\bx))| \leq \frac{1}{\sqrt{2\pi} \alpx\sigma_{\bx,\bx}}(Z_{\bx} - \ol{h}(\bx))
\]
where \(\ol{h} =  F_{Z,\bx}^{-1}\circ F_{B,\alpx}\circ h\). Note that if \(\alpx, \mu_{\bx}\)  and \(\sigma_{\bx,\bx}\)  are continuous, then so is \(\ol{h}:\mX\to \R\). Hence, if there exists \(c > 0\) such that \(\alpx\sigma_{\bx,\bx} \geq c\) then
\begin{align*}
    \sup_{\bx\in K}\,\left|\Vx - h(\bx)\right| \leq \frac{1}{c} \sup_{\bx\in K} |Z_{\bx} - \ol{h}(\bx)|.
\end{align*}
Therefore, it suffices to choose a process \(Z_{\bx}\) for which the event
\[
\mbP\Lset{\omega\in \Omega:\,\, \sup_{\bx\in K}|Z_{\bx} - \ol{h}(\bx)| < \eps} > 0
\]
for any continuous function \(\ol{h}:K\to \R\). Note the reproducing kernel Hilbert space (RKHS) associated to the covariance kernel
\begin{align*}
    \sigma_{\bx_1,\bx_2} = \sigma_0 e^{-\dX(\bx_1,\bx_2)^2/\tau^2}
\end{align*}
spans the space of all smooth functions if \(\tau > 0\) is allowed to vary freely \citep{choudhuri2007nonparametric}.

It is natural to characterize the support of the DDP or its variants in stronger topologies. One such topology arises when we endow \(\mP(\Th)\) with the {\em strong topology}. In this topology, a neighborhood basis at \(P^0\in\mP(\Th)\) is given by sets of the form
\[
\Lset{P\in \mP(\Th):\,\, \left|\int_\Theta f_i(\bth) dP(\bth) - \int_\Theta f_i(\bth) dP^0(\bth)\right| < \eps_i,\, i\in \bset{n}}
\]
for \(\eps_1,\ldots, \eps_n > 0\) and \(f_1,\ldots, f_n\in \LInf(\Th)\). Hence, we consider the (weak) product topology on \(\mP(\Th)^{\mX}\) when \(\mP(\Th)\) is endowed with the strong topology. We call this the {\em product-strong topology}. In this topology, a neighborhood basis at \(P^0\in \mP(\Th)^{\mX}\) is given by sets of the form
\[
\Lset{P\in \mP(\Th)X:\,\, \left|\int_\Theta f_{i,j}(\bth) dP_{\bx_j}(\bth) - \int_\Theta f_{i,j}(\bth) dP^0_{\bx_j}(\bth)\right| < \eps_{i,j},\, i,j\in \bset{n}}
\]
for \(\eps_{1,1},\ldots, \eps_{n,n} > 0\), \(\bx_1,\ldots, \bx_n\in \mX\) and \(f_{1,1},\ldots, f_{n,n}\in \LInf(\Th)\). By choosing simple functions, it becomes clear that the sets
\[
\Lset{P\in \mP(\Th)^{\mX}:\,\, P_{\bx_j}(B_{i,j}) \in I_j,\, i,j\in \bset{n}}
\]
for \(\bx_1,\ldots, \bx_n\in \mX\), \(I_{1,1},\ldots, I_{n,n}\subseteq [0, 1]\) open, and \(B_1,\ldots, B_n\in \mB(\Th)\) also form a neighborhood basis at \(P^0\). 

Theorem~\ref{thm:continuity:strong} suggests that neither the DDP not its variants will have full support on the product-strong topology. However, we are still able to characterize some key features of their support. We first introduce the following technical definition.

\begin{definition}
    Let \(\GX\) be a \(\mP(\Th)\)-valued random process on \(\mX\). 
    \begin{enumerate}[leftmargin=24pt]
        \item{If \(\GX \sim \DDP(\parV,\parTh)\) we define
            \[
            \mP(\Th)^{\mX}|_{\GX} := \Lset{P\in\mP(\Th)^{\mX}:\, \forall\, \bx\in \mX:\,\, P_{\bx} \ll G^0_{\bx}}.
            \]
        }
        \item{If \(\GX \sim \wDDP(\alpha,\parTh)\) we define
            \[
            \mP(\Th)^{\mX}|_{\GX} := \Lset{P\in\mP(\Th)^{\mX}:\, \forall\, \bx\in \mX:\,\, P_{\bx} \ll G^0_{\bx}}.
            \]
        }        
        \item{If \(\GX\sim \DDP(\parV,\parTh)\) we define
            \[
            \mP(\Th)^{\mX}|_{\GX} := \Lset{P\in\mP(\Th)^{\mX}:\, \forall\, \bx\in \mX:\,\, P_{\bx} \ll G^0}.
            \]
        }
    \end{enumerate}
\end{definition}

Therefore, we can associate to a DDP or to any of its variants a specific set of functions from \(\mX\) into \(\mP(\Th)\). The following theorem shows that, in fact, the support of DDP and its variants contain this set. We defer the proof of the theorem to Appendix~\ref{thm:support:product-strong:proof}

\begin{theorem}\label{thm:support:product-strong}
    Let \(\GX\) be a \(\mP(\Th)\)-valued process on \(\mX\). The following assertions are true. 
    \begin{enumerate}[leftmargin=24pt]
        \item{If \(\GX \sim \DDP(\parV,\parTh)\) and for any \(\bx_1,\ldots, \bx_n\in \mX\) the law of the random vector
            \[
            (\bth_{1,\bx_1}, \ldots, \bth_{1,\bx_n})
            \]
            has full support on \(\Th^n\) and the law of the random vector
            \[
            (V_{1,\bx_1}, \ldots, V_{1,\bx_n})
            \]
            has full support on \([0,1]^n\), then the support of \(\GX\) in the product-strong topology contains \(\mP(\Th)^{\mX}|_{\GX}\).
        }
        \item{If \(\GX \sim \wDDP(\alpha,\parTh)\) and for any \(\bx_1,\ldots, \bx_n\in \mX\) the law of the random vector
            \[
            (V_{1,\bx_1}, \ldots, V_{1,\bx_n})
            \]
            has full support on \([0,1]^n\), then the support of \(\GX\) in the product-strong topology contains \(\mP(\Th)^{\mX}|_{\GX}\).
        }        
        \item{If \(\GX\sim \thDDP(\parV,G^0)\), \(G^0\) has full support on \(\Th\), and for any \(\bx_1,\ldots, \bx_n\in \mX\) the law of the random vector
            \[
            (V_{1,\bx_1}, \ldots, V_{1,\bx_n})
            \]
            has full support on \([0,1]^n\) then the support of \(\GX\) in the product-strong topology contains \(\mP(\Th)^{\mX}|_{\GX}\).
        }
    \end{enumerate}
\end{theorem}

When \(\mP(\Th)\) is endowed with the strong topology, we can consider the associated compact-open topology on \(\mP(\Th)^{\mX}\). In this topology, the neighborhood basis at any \(P^0\in \mP(\Th)^{\mX}\) have the form~\eqref{eq:support:compact-weak:basis} for \(f_1,\ldots, f_n\in \LInf(\Th)\). 

In this case, the functions of interest are strongly continuous functions from \(\mX\) into \(\mP(\Th)\). In contrast to Theorem~\ref{thm:support:compact-weak} we cannot show that the support of the DDP not its variants contains this set. However, we can show the support contains the intersection between \(C_S(\mX,\mP(\Th))\) and the surrogate functions associated to a DDP or \(\th\)DDP. We defer the proof of the following result to Appendix~\ref{thm:support:compact-strong:proof}.

\begin{theorem}\label{thm:support:compact-strong}
    Let \(\GX\sim \DDP(\parV,\parTh)\) or \(\GX\sim \thDDP(\parV,G^0)\). Suppose the following conditions hold:
    \begin{enumerate}[leftmargin=24pt]
        \item{The processes \(\set{\set{\Vix:\, \bx\in \mX}}_{i\in\N}\) have a.s. continuous sample paths.
        }
        \item{For any \(\eps > 0\), continuous function \(h:\mX\to [0, 1]\) and \(K\subset \mX\) compact we have
            \[
            \mbP\left(\Lset{\omega\in \Omega:\, \sup_{\bx\in K} |\Vw_{1,\bx} - h(\bx)| < \eps}\right) > 0.
            \]
        }
        \item{If \(\GX\sim \DDP(\parV,G^0)\) then \(G^0\) has full support on \(\Theta\). 
        }        
        \item{If \(\GX\sim \DDP(\parV,\parTh)\) then there exists a measure \(G^0\) on \(\Th\) such that \(G^0_{\bx} \ll G^0\) for every \(\bx\in \mX\) and that for any \(A\in\mB(\Th)\) and \(K\subset \mX\) compact we have
            \[
            G^0(A) > 0 \quad\Rightarrow\quad\mbP(\set{\omega\in \Omega:\, \bthw_{1,\bx} \in A,\, \forall\, \bx\in K}) > 0.
            \]
        }
    \end{enumerate}
    Then \(C_S(\mX,\mP(\Th)) \cap \mP(\Th)^{\mX}|_{\GX}\) is in the support of \(\GX\) with respect to the compact-strong topology.
\end{theorem}

\subsection{Association structure}
\label{sec:ddp:association}

In statistical applications, it is of interest to study the behavior of the process \(\set{\Gx(B):\,\, \bx\in\mX}\) for some fixed \(B\in \mB(\Th)\). If \(\GX \sim \thDDP(\parV,G^0)\) the hypothesis of Theorem~\ref{thm:continuity:strong} ensure the process \(\set{\Gx(B):\,\, \bx\in\mX}\)  has a.s. continuous sample paths. As a consequence, for any \(d\in \N\) and \(f:[0,1]^d\to \R\) the process
\[
\set{f(\Gw_{\bx_1}(B),\ldots, \Gw_{\bx_d}(B)):\,\, (\bx_1,\ldots,\bx_d)\in \mX^d}
\]
has a.s. continuous sample paths. Furthermore, this holds for its expectation
\begin{equation}
    \label{eq:continuity:measureOfAssociationGeneric}
    F_B(\bx_1,\ldots, \bx_d) := \ev(f(G_{\bx_1}(B),\ldots, G_{\bx_d}(B))).
\end{equation}
Some functions of this form that are of statistical interest are the {\em measures of association}. For instance, the Pearson correlation coefficient is given by
\[
\rho(G_{\bx_1}(B)G_{\bx_2}(B)) = \frac{\ev(G_{\bx_1}(B)G_{\bx_2}(B)) - \ev(G_{\bx_1}(B))\ev(G_{\bx_2}(B))}{\ev(G_{\bx_1}(B)^2)^{1/2}\ev(G_{\bx_2}(B)^2)^{1/2}}.
\]
It is clear that it is continuous whenever the denominator is non-zero. Continuity implies
\[
\lim_{\bx\to\bx_0}\rho(G_{\bx}(B)G_{\bx_0}(B)) = 1.
\]
On the other hand, if
\[
\lim_{\dX(\bx,\bx_0) \to \infty} \ev(G_{\bx}(B)G_{\bx_0}(B)) = \ev(G_{\bx}(B))\ev(G_{\bx_0}(B))
\]
then it follows that
\[
\lim_{\dX(\bx,\bx_0) \to \infty}\rho(G_{\bx}(B)G_{\bx_0}(B)) = 0.
\]
Since the DDP and \(w\)DDP may not have a.s. strongly continuous paths, the above argument does not hold and, with positive probability, the process \(\set{\Gx(B):\,\, \bx\in\mX}\) may have discontinuous sample paths. In this case, a measure of association can act as a surrogate to study the regularity of this process, on average, at any point. The following theorem states that, under mild conditions, any function of the form~\eqref{eq:continuity:measureOfAssociationGeneric} is continuous. Its proof is given in Appendix~\ref{thm:continuity:association:proof}.

\begin{theorem}\label{thm:continuity:association}
    Let \(\GX\) be a \(\mP(\Th)\)-valued process on \(\mX\) and let \(B\in \mP(\Th)\). Suppose that one of the following assertions holds.
    \begin{enumerate}[leftmargin=24pt]
        \item{\(\GX\) is a continuous sample paths DDP with \(\GX \sim \DDP(\parV,\parTh)\).
        }
        \item{\(\GX\) is a continuous sample paths \(w\)DDP with \(\GX\sim \wDDP(\alpha,\parTh)\).
        }
    \end{enumerate}
    Furthermore, suppose that for any \(d\in \N\) the function
    \[
    (\bx_1,\ldots, \bx_d) \mapsto \mbP(\set{\omega\in\Omega:\,\, \bthw_{1,\bx_1}\in B,\ldots, \bthw_{1,\bx_d}\in B})
    \]
    is continuous. Then, for any \(d\in \N\), and continuous \(f:[0,1]^d\to \R\), the function \(F:\mX^d\to \R\) defined as
    \[
    F(\bx_1,\ldots, \bx_d) := \ev(f(G_{\bx_1}(B),\ldots, G_{\bx_d}(B))).
    \]
    is continuous.
\end{theorem}

\section{Mixtures induced by dependent Dirichlet processes}
\label{sec:mixtures}

From now on, we let \(\mY\) be a Polish space, and we let \(\muY\) be a base measure on \((\mY,\mB(\mY))\). To allow for flexible statistical models, we also consider a Polish space \(\Gma\) representing the parameters of the induced mixture. We always assume \(\Gma\) is endowed with the Borel \(\sigma\)-algebra \(\mB(\Gma)\). The  dependent Dirichlet processes mixture models that we study are constructed from a fixed measurable function \(\psi:\mY\times\Gma\times \Th\to \Rp\) such that
\[
\forall\, (\bgma,\bth)\in \Gma\times\Th:\,\, \int_\mY \psi(\by,\bgma,\bth)\, d\muY(\by) = 1.
\]
The mixture induced by \(P\in \mP(\Th)^{\mX}\) is the map \(M^P:\Gma\times\mX\to \mP(\mY)\) formally defined as
\[
\forall\, (\bgma,\bth)\in \Gma\times\Th,\, B\in \mB(\mY):\,\, M^P_{\bgma,\bx}(B) := \int_B \int_{\Th} \psi(\by,\bgma,\bth) dP_{\bx}(\bth)d\muY(\by).
\]
In particular, a dependent Dirichlet processes mixture model is a map
\[
M : \mP(\Th)^{\mX} \to \mP(\mY)^{\Gma\times\mX}.
\]
By construction, the measure \(M^P_{\bgma,\bx}\) is absolutely continuous with respect to \(\muY\) for any \((\bgma,\bx)\in \Gma\times \mX\). For this reason, we distinguish the set \(\mD(\mY)\subset \mP(\mY)\) of probability measures on \(\mY\) that admit a density with respect to \(\muY\). We often use the identification
\[
\mD(\mY) \cong \Lset{p\in L^1(\mY,\mB(\mY), \muY):\,\, \int p(\by)\, d\muY(\by) = 1,\, p \geq 0}.
\]
In particular, for the dependent Dirichlet processes mixture models we study we have an explicit for their density. For this reason, for \(P\in \mP(\mY)\) we define the function \(\rho^P:\mY\times\Gma\times\mX\to \Rp\)
\[
\rho^P(\by,\bgma,\bx) := \int_{\Th} \psi(\by,\bgma,\bth) dP_{\bx}(\bth)
\]
representing the density of \(Q^P_{\bgma,\bx}\) with respect to \(\muY\). We sometimes write
\[
\rho^P_{\bgma,\bx}(\by) := \rho^P(\by,\bgma,\bx).
\]
Hence, the dependent Dirichlet processes mixture model \(M\) induces a map
\[
\rho : \mP(\Th)^{\mX} \to \mD(\mY)^{\Gma\times\mX}. 
\]
Depending on the choice of \(\muY\) and \(\psi\) the dependent Dirichlet processes mixture model may have regularizing properties and the density \(\rho^P\) may be, for instance, continuous. The following lemma shows that, under mild regularity and decay assumptions on \(\psi\), we can characterize points of continuity of \(\rho^P\) when \(P:\mX\to\mP(\Th)\) is weakly continuous. The proof of the following result is deferred to Appendix~\ref{lem:mixtures:continuousDensityAll:proof}.

\begin{lemma}\label{lem:mixtures:continuousDensityAll}
    Let \(P:\mX\to \mP(\Th)\) be weakly continuous, and suppose that \(\psi\) is continuous. Let \((\byo,\bgmao)\in \mY\times \Gma\). If for every \(\eps >0\) there exists an open neighborhood \(U_{\byo}\subset \mY\) of \(\byo\), an open neighborhood \(U_{\bgmao}\subset \Gma\) of \(\bgmao\), and a compact \(K_{\Th}\subset \Th\) such that
    \begin{equation}
        \label{eq:mixture:likelihoodDecayConditionTheta}
        \sup\set{\psi(\by,\bgma,\bth):\, (\by,\bgma,\bth)\in U_{\byo}\times U_{\bgmao}\times K_{\Th}^{c}} < \eps
    \end{equation}
    then \(\rho^P\) is continuous on \(U_{\byo}\times U_{\bgmao}\times \mX\). 
\end{lemma}

The hypotheses imply that near \(\byo\) and \(\bgmao\) the function \(\psi\) tends to zero ``at infinity'' in \(\bth\). To gain insight into the consequences of these assumptions, we consider the following example. Let \(\mY = [0, 1]\) be endowed with the standard topology, and let \(\muY\) be the Lebesgue measure restricted to \([0, 1]\). Let
\[
\Th := \set{(\alpha,\beta)\in \R^2:\, \alpha,\beta \geq 1}
\]
be endowed with the standard subspace topology, and let \(\Gma = \emptyset\). If we consider the function
\[
\psi(y,\alpha,\beta) = \frac{y^{\alpha -1}(1 - y)^{\beta -1}}{B(\alpha,\beta)},
\] 
associated to a family of \(\Beta(\alpha,\beta)\) probability distributions on \([0, 1]\), then the induced mixture model would not satisfy the properties of the lemma. In fact, if \(y_0 = 1/2\) we can choose \(\alpha = \beta = t\) to see that, from Stirling's approximation,
\[
\psi(y_0,t,t) = \frac{1}{2^{2t - 2} B(t,t)} \sim \frac{1}{2^{2t - 2}}\frac{2^{2t - 1/2} t^{2t - 1/2}}{\sqrt{2\pi} t^{2t - 1}} = \frac{2^{3/2}}{\sqrt{2\pi}} t^{1/2}
\]
for \(t \gg 1\). Hence, \(\psi\) does not decay over \(\Th\) near \(y_0\). This can be mitigated by restricting the values of both \(\alpha\) and \(\beta\) to a compact set. A middle ground can be achieved if, for example, one parameter is constrained to a compact set, whereas the other becomes a parameter of the  induced mixture. For example,
\begin{align*}
    \Gma &:= \set{\alpha\in \R:\, 1\leq \alpha},\\
    \Th &:= \set{\beta\in \R:\, 1\leq\beta \leq \beta_{\max}}.
\end{align*}
In this case, the resulting induced mixture model satisfies the desired properties. Finally, note that failure to satisfy this condition is not always due to a lack of compactness. For instance, we could consider the model
\[
\Th := \set{(\alpha,\beta)\in \R^2:\, \alpha,\beta\in [1/4, 1/2]}. 
\]
In this case, not only \(\psi\) is discontinuous, we also have
\[
\lim_{y\to 0}\, \psi(y,\alpha,\beta) = \infty
\]
for any choice of \(\alpha,\beta\).

Due to the continuity properties of a DPP and its variants,  the conclusions of Lemma~\ref{lem:mixtures:continuousDensityAll} follow from milder hypotheses. In fact, in this case the same conclusion follows by only imposing  boundedness. We defer the proof of this result to Appendix~\ref{lem:mixtures:continuousDensity:proof}

\begin{lemma}\label{lem:mixtures:continuousDensity}
    Let \(\GX\) be a \(\mP(\Th)\)-valued process on \(\mX\) and let \(\psi: \mY\times\Gma\times \Th \to \Rp\). Suppose one of following conditions hold:
    \begin{enumerate}[leftmargin=24pt]
        \item{\(\GX\) is a continuous parameter DDP, with \(\GX \sim \DDP(\parV,\parTh)\).
        }
        \item{\(\GX\) is a continuous parameter \(w\)DDP, with \(\GX\sim \wDDP(\alpha,\parTh)\).
        }
        \item{\(\GX\) is a continuous parameter \(\th\)DDP, with \(\GX\sim \thDDP(\parV,G^0)\).
        }
    \end{enumerate}
    Furthermore, suppose that \(\psi\) is continuous. Let \((\byo,\bgmao)\in \mY\times \Gma\). If there exists an open neighborhood \(U_{\byo}\subset \mY\) of \(\byo\) and an open neighborhood \(U_{\bgmao}\subset \Gma\) of \(\bgmao\) such that
    \[
    \sup\set{\psi(\by,\bgma,\bth):\, (\by,\bgma,\bth)\in U_{\byo}\times U_{\bgmao}\times \Th} < \infty
    \]
    then for a.e. \(\omega\in \Omega\) the function \(\qGw\) is continuous at any \((\by,\bgma,\bx)\in \mY\times \Gma\times \mX\).
\end{lemma}

\section{Properties of dependent Dirichlet processes mixture models}

\subsection{Continuity}\label{sec:mixtures:continuity}

Mixture models have a regularizing effect. Under the same assumptions of Lemma~\ref{lem:mixtures:continuousDensity} the dependent Dirichlet processes mixture model \(M\) maps weakly continuous into uniformly continuous functions from \(\mX\) into \(\mP(\Th)\). We defer the proof of this result to Appendix~\ref{thm:mixtures:continuity:uniformAll:proof}.

\begin{theorem}\label{thm:mixtures:continuity:uniformAll}
    Suppose that \(\muY\) is locally finite and that \(\psi\) is continuous. Then, for every \(P\in C_W(\mX,\mP(\Th))\) the induced mixture \(M(P)\) is uniformly continuous, i.e., 
    \[
    \lim_{(\bgma,\bx)\to(\bgmao,\bxo)}\,\nrm{M^P_{\bgma,\bx} -M^P_{\bgmao,\bxo}}_{\TV} = 0
    \]
    for any \(\bgmao\in\Gma\) and \(\bxo\in \mX\).
\end{theorem}

As a consequence of this result, the induced mixture of a continuous parameter DDP or its variants has uniformly continuous sample paths.

\begin{corollary}
    Let \(\GX\) be a \(\mP(\Th)\)-valued process on \(\mX\) and let \(\psi: \mY\times\Gma\times \Th \to \Rp\). Suppose \(\GX\) is a continuous parameter DDP, a continuous parameter \(w\)DDP or a continuous parameter \(\th\)DDP. Then for a.e. \(\omega\in \Omega\) the map \(M(\Gw)\) is uniformly continuous, i.e., 
    \[
    \lim_{(\bgma,\bx)\to(\bgmao,\bxo)}\,\nrm{M^{\Gw}_{\bgma,\bx} -M^{\Gw}_{\bgmao,\bxo}}_{\TV} = 0
    \]
    for any \(\bgmao\in\Gma\) and \(\bxo\in \mX\).
\end{corollary}

\subsection{Support}
\label{sec:mixtures:support}

As in the case of a DDP or any of its variants, it is of interest to determine the effect that an induced mixture has on the support. As for the induced mixture models that we study the probability measures on \(\mY\) admit a density with respect to \(\muY\), we may interpret the sample paths of the induced mixture as elements of \(\mD(\mY)^{\Gma\times\mX}\). This allows us to consider other topologies defined in terms of the density of the induced mixture model. 

On \(\mD(\mY)\) we consider the topology induced by the Hellinger distance
\[
\dHel(p_1,p_2)^2 := \frac{1}{2}\int_{\mY}(\sqrt{p_1(\by)} - \sqrt{p_2(\by)})^2d\muY(\by) = 1 - \int_{\mY} \sqrt{p_1(\by)p_2(\by)}d\muY(\by),
\]
by the \(\LInf\) distance
\[
\dLInf(p_1,p_2):= \sup_{\by \in \mY}\,\, |p_1(\by) - p_2(\by)|,
\]
and by the Kullback-Leibler (KL) divergence
\[
\KL{p_1}{p_2} := \int_{\mY}p_1(\by) \log\left(\frac{p_1(\by)}{p_2(\by)}\right)d\muY(\by).
\]
for \(p_1,p_2\in \mD(\mY)\).

\subsubsection{The Hellinger distance}

We define the {\em product-Hellinger} topology on \(\mD(\mY)^{\Gma\times\mX}\) as follows. In this topology, a neighborhood basis at \(\mPo\in \mD(\mY)^{\Gma\times\mX}\) is given by sets of the form
\[
\Lset{P\in \mD(\mY)^{\Gma\times\mX}:\,\, \dHel(p_{\bgmai,\bxi}, \po_{\bgmai,\bxi}) < \eps_i}
\]
for some \(\eps_1,\ldots, \eps_n > 0\), \(\bgma_1,\ldots,\bgma_n\in\Gma\) and \(\bx_1,\ldots,\bx_n\in \mX\).  The following result shows that any neighborhood of the image of \(P\in\mP(\Th)^{\mX}\) under the induced mixture on the product-Hellinger topology contains, with positive probability, the image of a sample path of a DDP or its variants under the same induced mixture. We defer the proof of the following result to Appendix~\ref{thm:mixtures:support:hellingerWeakStrong:proof}.

\begin{theorem}\label{thm:mixtures:support:weakHellinger}
    Suppose that \(\muY\) is locally finite, that \(\psi\) is continuous and satisfies~\eqref{eq:mixture:likelihoodDecayConditionTheta} for any \((\by,\bgma)\in \mY\times\Gma\), and that the hypotheses of Theorem~5 in Part I hold. Then, for any \(\mPo\in \mP(\Th)^{\mX}\) the event
    \[    \Lset{\omega\in\Omega:\,\, \dHel(\qPo_{\bgmai,\bxi}, \qGw_{\bgmai,\bxi}) < \eps_i}
    \]
    has positive probability for any \(\eps_1,\ldots, \eps_n > 0\), \(\bgma_1,\ldots,\bgma_n\in\Gma\) and \(\bx_1,\ldots,\bx_n\in \mX\).
\end{theorem}

A stronger topology induced by the Hellinger distance is what we call the {\em compact-Hellinger} topology on \(\mD(\mY)^{\Gma\times\mX}\). In this topology, a neighborhood basis at \(\mPo\in \mD(\mY)^{\Gma\times\mX}\) is given by sets of the form
\[
\Lset{P\in \mD(\mY)^{\Gma\times\mX}:\,\, \sup_{(\bgma,\bx)\in K_{\Gma}\times K_{\mX}}\, \dHel(p_{\bgma,\bx}, \po_{\bgma,\bx}) < \eps}
\]
for some \(\eps > 0\), \(K_{\Gma}\subset\Gma\) compact, and \(K_{\mX}\subset\mX\) compact. Note these neighborhoods include sets of the form
\[
\Lset{P\in \mD(\mY)^{\Gma\times\mX}:\,\, \sup_{\bx\in K_{\mX}}\, \dHel(p_{\bgmai,\bx}, \po_{\bgmai,\bx}) < \eps_i,\,i\in\bset{n}}
\]
for \(\eps_1,\ldots,\eps_n > 0\) and \(\bgma_1,\ldots,\bgma_n \in \Gma\). The following result shows that any neighborhood of the image of \(P\in\mP(\Th)^{\mX}\) under the induced mixture on the product-Hellinger topology also contains, with positive probability, the image of a sample path of the DDP or its variants under the same induced mixture. We defer the proof of the following result to Appendix~\ref{thm:mixtures:support:hellingerWeakStrong:proof}.

\begin{theorem}\label{thm:mixtures:support:strongHellinger}
    Suppose that \(\muY\) is locally finite, that \(\psi\) is continuous and satisfies~\eqref{eq:mixture:likelihoodDecayConditionTheta} for any \((\by,\bgma)\in \mY\times\Gma\), and that the hypotheses of Theorem~7 in Part I hold. Then, for any \(\mPo\in \mP(\Th)^{\mX}\) the event
    \[
    \Lset{\omega\in\Omega:\,\,\sup_{(\bgma,\bx)\in K_{\Gma}\times K_{\mX}}\,\dHel(\qPo_{\bgma,\bx}, \qGw_{\bgma,\bx}) < \eps}
    \]
    has positive probability for any \(\eps > 0\), and compact \(K_{\Gma}\subset\Gma\) and \(K_{\mX}\subset\mX\).
\end{theorem}

\subsubsection{The \(\LInf\) distance}

We define the {\em product-\(\LInf\)} topology on \(\mD(\mY)^{\Gma\times\mX}\) as follows. In this topology, a neighborhood basis at \(\mPo\in \mD(\mY)^{\Gma\times\mX}\) is given by sets of the form
\[
\Lset{P\in \mD(\mY)^{\Gma\times\mX}:\,\, \dLInf(p_{\bgmai,\bxi}, \po_{\bgmai,\bxi}) < \eps_i}
\]
for some \(\eps_1,\ldots, \eps_n > 0\), \(\bgma_1,\ldots,\bgma_n\in\Gma\) and \(\bx_1,\ldots,\bx_n\in \mX\). Similarly to the Hellinger distance, any neighborhood of the image of \(P\in\mP(\Th)^{\mX}\) under the induced mixture on the product-\(\LInf\) topology contains, with positive probability, the image of a sample path of a DDP or its variants under the same induced mixture. However, we require the additional hypothesis of compactness of \(\mY\). We defer the proof of the following result to Appendix~\ref{thm:mixtures:support:LInfWeakStrong:proof}.

\begin{theorem}\label{thm:mixtures:support:weakLInf}
    Suppose that \(\mY\) is compact, that \(\psi\) is continuous and satisfies~\eqref{eq:mixture:likelihoodDecayConditionTheta} for any \((\by,\bgma)\in \mY\times\Gma\), and that the hypotheses of Theorem~5 in Part I hold. Then, for any \(\mPo\in \mP(\Th)^{\mX}\) the event
    \[
    \Lset{\omega\in\Omega:\,\, \dLInf(\qPo_{\bgmai,\bxi}, \qGw_{\bgmai,\bxi}) < \eps_i}
    \]
    has positive probability for any \(\eps_1,\ldots, \eps_n > 0\), \(\bgma_1,\ldots,\bgma_n\in\Gma\) and \(\bx_1,\ldots,\bx_n\in \mX\).
\end{theorem}

The stronger {\em compact-\(\LInf\)} topology on \(\mD(\mY)^{\Gma\times\mX}\) can be defined similarly as for the Hellinger distance. In this topology, a neighborhood basis at \(\mPo\in\mD(\mY)^{\Gma\times\mX}\) is given by sets of the form
\[
\Lset{P\in \mD(\mY)^{\Gma\times\mX}:\,\, \sup_{(\bgma,\bx)\in K_{\Gma}\times K_{\mX}}\, \dLInf(p_{\bgmai,\bxi}, \po_{\bgma,\bx}) < \eps}
\]
for some \(\eps > 0\), \(K_{\Gma}\subset\Gma\) compact, and \(K_{\mX}\subset\mX\) compact. Note these neighborhoods include sets of the form
\[
\Lset{P\in \mD(\mY)^{\Gma\times\mX}:\,\, \sup_{\bx\in K_{\mX}}\, \dLInf(p_{\bgmai,\bx}, \po_{\bgmai,\bx}) < \eps_i,\,i\in\bset{n}}
\]
for \(\eps_1,\ldots,\eps_n > 0\) and \(\bgma_1,\ldots,\bgma_n \in \Gma\). The following result shows that any neighborhood of the image of \(P\in\mP(\Th)^{\mX}\) under the induced mixture on the product-\(\LInf\) topology also contains, with positive probability, the image of a sample path of a DDP or its variants under the same induced mixture. In this case, we also assume \(\mY\) is compact. We defer the proof of the following result to Appendix~\ref{thm:mixtures:support:LInfWeakStrong:proof}.

\begin{theorem}\label{thm:mixtures:support:strongLInf}
    Suppose that \(\mY\) is compact, that \(\psi\) is continuous and satisfies~\eqref{eq:mixture:likelihoodDecayConditionTheta} for any \((\by,\bgma)\in \mY\times\Gma\), and that the hypotheses of Theorem~7 in Part I hold. Then, for any \(\mPo\in \mP(\Th)^{\mX}\) the event
    \[
    \Lset{\omega\in\Omega:\,\,\sup_{(\bgma,\bx)\in K_{\Gma}\times K_{\mX}}\,\dLInf(\qPo_{\bgma,\bx}, \qGw_{\bgma,\bx}) < \eps}
    \]
    has positive probability for any \(K_{\Gma}\subset\Gma\) compact, \(K_{\mX}\subset\mX\) compact, and \(\eps > 0\).
\end{theorem}

\subsubsection{The Kullback-Leibler divergence}

The KL-divergence defines a premetric on \(\mD(\mY)\) that induces a locally convex topology on \(\mD(\mY)\). This topology depends on which argument is used to define the neighborhood. Due to its connection to the consistency of Bayesian procedures, we consider the neighborhood basis at \(\mPo\in \mD(\mY)\) given by sets of the form
\[
\set{P\in \mD(\mY):\,\, \KL{\po}{p} < \eps}
\]
for \(\eps>0\). The {\em product-KL} topology on \(\mD(\mY)^{\Gma\times\mX}\) is defined as follows. A neighborhood basis for \(\mPo\in \mD(\mY)^{\Gma\times\mX}\) on the product-KL topology is given by sets of the form
\[
\Lset{P\in \mD(\mY)^{\Gma\times\mX}:\,\, \KL{\po_{\bgmai,\bxi}}{p_{\bgmai,\bxi}} < \eps_i}
\]
for some \(\eps_1,\ldots, \eps_n > 0\), \(\bgma_1,\ldots,\bgma_n\in\Gma\) and \(\bx_1,\ldots,\bx_n\in \mX\). In this case, we obtain a result similar to that obtained for the Hellinger distance. In this case, we also assume \(\mY\) is compact. We defer the proof of the following result to Appendix~\ref{thm:mixtures:support:KLWeakStrong:proof}.

\begin{theorem}\label{thm:mixtures:support:weakKL}
    Suppose that \(\mY\) is compact, and that \(\psi\) is continuous, strictly positive, and satisfies~\eqref{eq:mixture:likelihoodDecayConditionTheta} for any \((\by,\bgma)\in \mY\times\Gma\). Furthermore, suppose the hypotheses of Theorem~5 in Part I hold. Then, for any \(\mPo\in \mP(\Th)^{\mX}\) the event
    \[
    \Lset{\omega\in\Omega:\,\, \KL{\po_{\bgmai,\bxi}}{\qGw_{\bgmai,\bxi}} < \eps_i}
    \]
    has positive probability for any \(\eps_1,\ldots, \eps_n > 0\), \(\bgma_1,\ldots,\bgma_n\in\Gma\) and \(\bx_1,\ldots,\bx_n\in \mX\).
\end{theorem}

The stronger {\em compact-KL} topology on \(\mD(\mY)^{\Gma\times\mX}\) can be defined similarly as for the Hellinger and \(\LInf\) distances. In this topology, a  neighborhood basis at \(\mPo\in \mD(\mY)^{\Gma\times\mX}\) is given by
\[
\Lset{P\in \mD(\mY)^{\Gma\times\mX}:\,\, \sup_{(\bgma,\bx)\in K_{\Gma}\times K_{\mX}}\, \KL{\po_{\bgma,\bx}}{p_{\bgma,\bx}} < \eps}
\]
for some \(\eps > 0\), \(K_{\Gma}\subset\Gma\) compact, and \(K_{\mX}\subset\mX\) compact. Note these neighborhoods include sets of the form
\[
\Lset{P\in \mD(\mY)^{\Gma\times\mX}:\,\, \sup_{\bx\in K_{\mX}}\, \KL{\po_{\bgmai,\bx}}{p_{\bgmai,\bx}} < \eps_i,\,i\in\bset{n}}
\]
for \(\eps_1,\ldots,\eps_n > 0\) and \(\bgma_1,\ldots,\bgma_n \in \Gma\). The following result shows that any neighborhood of the image of \(P\in\mP(\Th)^{\mX}\) under the induced mixture on the product-\(\LInf\) topology also contains, with positive probability, the image of a sample path of a DDP or its variants under the same induced mixture. We defer the proof of the following result to Appendix~\ref{thm:mixtures:support:KLWeakStrong:proof}.

\begin{theorem}\label{thm:mixtures:support:strongKL}
    Suppose that \(\mY\) is compact, and that \(\psi\) is continuous, strictly positive, and satisfies~\eqref{eq:mixture:likelihoodDecayConditionTheta} for any \((\by,\bgma)\in \mY\times\Gma\). Furthermore, suppose the hypotheses of Theorem~7 in Part I hold. Then, for any \(\mPo\in \mP(\Th)^{\mX}\) the event
    \[
    \Lset{\omega\in\Omega:\,\, 
        \sup_{(\bgma,\bx)\in K_{\Gma}\times K_{\mX}}\,\KL{\po_{\bgma,\bx}}{\qGw_{\bgma,\bx}} < \eps}
    \]
    has positive probability for any \(\eps > 0\).
\end{theorem}

\subsection{Association structure}
\label{sec:mixtures:association}

As a consequence of Theorem~\ref{thm:mixtures:continuity:uniformAll}, when \(\psi\) is continuous and the DDP or any of its variants have a.s. weakly continuous sample paths, then the induced mixture will have a.s. strongly continuous sample paths. Therefore, the process \(\set{Q^{\Gw}_{\bgma,\bx}(B):\,\, \bgma\in \Gma,\, \bx\in\mX}\) for some fixed \(B\in \mB(\Th)\) has a.s. continuous paths. Following the arguments in Section~4.3 in Part I, for any \(d\in \N\) and \(f:[0,1]^d\to \R\) the process
\[
\set{f(Q^{G}_{\bgma_{1},\bx_1}(B),\ldots, Q^{G}_{\bgma_{d},\bx_d}(B)):\,\, (\bx_1,\ldots,\bx_d)\in \mX^d}
\]
has a.s. continuous sample paths. This also holds for its expectation
\begin{equation*}
    F_B(\bgma_1,\bx_1,\ldots, \bgma_d,\bx_d) := \ev(f(Q^{G}_{\bgma_{1},\bx_1}(B),\ldots, Q^{G}_{\bgma_{d},\bx_d}(B))).
\end{equation*}
In some applications, it is useful to consider the parameter \(\bgma\) as random. Let \((\bgma_1,\ldots,\bgma_d)\) be a random vector defined on \((\Omega,\mF,\mbP)\). Then the process
\[
\set{F_B(\bgma_1,\bx_1,\ldots, \bgma_d,\bx_d):\,\, (\bx_1,\ldots,\bx_d)\in \mX^d}
\]
has a.s. continuous sample paths. By the same arguments as before,
\[
\ev(F_B(\bgma_1,\bx_1,\ldots, \bgma_d,\bx_d)) = \int_{\Gma^d} F_B(\bgma_1,\bx_1,\ldots, \bgma_d,\bx_d)\,d\mPGam(\bgma_1,\ldots,\bgma_d),
\]
where \(\mPGam\) is the probability law of \((\bgma_1,\ldots,\bgma_d)\), is continuous.

\subsection{Posterior consistency}
\label{sec:mixtures:posteriorConsistency}

An important property of dependent Dirichlet processes mixture model is their posterior consistency. To study the asymptotic behavior of dependent Dirichlet processes mixture models, we consider a random sample of size \(n\) given by pairs \((\byi,\bxi)\) for  \(i\in\bset{n}\). As is common in regression settings, we assume that \(\bx_1,\ldots, \bx_n\) contain only exogenous covariates. The exogeneity assumption allows us to focus on the problem of conditional density estimation, regardless of the mechanism generating the predictors, that is, if they are randomly generated or fixed by design~\citep[see, e.g.,][]{barndorff;1973,barndorff;1978,florens;mouchart;rolin;1990}. Let \(\mPo\) be the true probability measure generating the predictors admitting a density \(p^0\) with respect to a measure \(\muX\). By the exogeneity assumption, the true probabi\-li\-ty model for the response variable and predictors takes the form \(m^0(\by,\bx)=p^0(\bx)q^0(\by \mid \bx)\). In this case, both \(p^0\) and \(\set{q^0_{\bx}:\, \bx \in \mX}\) are in free variation, where \(q^0_{\bx}(\by) := q^0(\by\mid\bx)\) denoting a conditional density defined on \(\mY\) for every \(\bx \in \mX\).

Let \(\mw(\by,\bx):=p^0(\bx)\gw_{\bx}(\by)\) be the random joint distribution for the response and predictors arising when \(\set{g_{\bx}: \bx \in \mX}\) is an induced mixture model induced by a \(\th\)DDP. Since the KL divergence between \(m^0\) and  the implied joint distribution \(\mw\) can be bounded as
\begin{align*}
    \KL{m^0}{\mw} &= \int_{\mX}\int_{\mY} m^0(\by,\bx)\log \left(\frac{m^0(\by,\bx)}{\mw(\by,\bx)} \right)d\muY(\by) d\muX(\bx) \\
    &= \int_{\mX} p^0(\bx) \int_{\mY} q^0(\by \mid \bx)\log \left(\frac{q^0(\by \mid \bx)}{\gw_{\bx}(\by)} \right)d\muY(\by) d\muX(\bx)\\
    &\leq \sup_{\bx \in \mX} \int_{\mY} q^0(\by \mid \bx)\log \left( \frac{q^0(\by \mid \bx)}{\gw_{\bx}(\by)}\right)d\muY(\by)
\end{align*} 
when \(\bx\) contains only continuous predictors, it follows that, under the assumptions of Theorem~\ref{thm:mixtures:support:strongKL} when \(\mX\) is compact, for every \(\eps >0\)
\begin{multline*}
    \mbP(\set{\omega\in\Omega:\,  \KL{m^0}{\mw} < \eps})\\
    \geq \mbP\left(\Lset{\omega\in\Omega:\,\sup_{\bx \in \mX} \int_{\mY} q^0(\by \mid \bx)\log \left( \frac{q^0(\by \mid \bx)}{\gw_{\bx}(\by)} \right)d\muY(\by) < \eps}\right) > 0,
\end{multline*}
Thus, by Schwartz's theorem~\citep{schwartz;1965}, it follows that the posterior distribution associated with the random joint distribution induced by any of the proposed models is weakly consistent, that is, the posterior measure of any weak neighborhood, of any joint distribution of the form \(m^0(\by,\bx)=p^0(\bx)q^0(\by \mid \bx)\), converges to 1 as the sample size goes to infinity. This result is summarized in the following theorem.

\begin{theorem}
    \label{thm:mixtures:consistency:weak}
    Suppose that assumptions of Theorem~\ref{thm:mixtures:support:strongKL} hold. Then, the posterior distribution \(\mw(\by,\bx)=p^0(\bx)f(\by|\bx,\Gw_{\bx})\) associated with the random joint distribution induced by the process \(\GX\) where \(p^0\) is the density generating the predictors and 
    \[
    f(\by\mid \bx,\Gw_{\bx})=\int_{\Th}\psi(\by,\bth)d\Gw_{\bx}(\bth),
    \]
    is weakly consistent, under independent sampling, at any joint distribution of the form \(m^0(\by,\bx)=p^0(\bx)f^{0}(\by\mid\bx)\) with
    \[
    f^{0}(\by\mid\bx)=\int_{\Th}\psi(\by,\bth)dP^{0}_{\bx}(\bth)
    \]
    and \(P^0\in C_S(\mX,\muY)\).
\end{theorem}

Although Theorem~\ref{thm:mixtures:consistency:weak} assumes that \(\bx\) contains only continuous predictors, a similar result can be obtained when \(\bx\) contains only predictors with finite support (e.g., categorical, ordinal and discrete predictors) or mixed continuous and predictors with finite support. The theorem can be also extended for more general mixture models induced by the DDP.

\begin{theorem}
    Suppose that assumptions in Theorem~\ref{thm:mixtures:support:strongKL} hold and that \(\mX\) is compact. Then, the posterior distribution associated with the random joint distribution
    \[
    m(\by,\bx,\bgmaw)=p^0(\bx)f(\by\mid\bx,\Gw_{\bx},\bgmaw)
    \]
    where \(p^0\) is the density generating the predictors and 
    \[
    f(\by\mid \bx, \Gw_{\bx},\bgmaw) = \int_{\Th}\psi(\by,\bth,\bgmaw)d\Gw_{\bx}(\bth)
    \]
    is weakly consistent, under independent sampling, at any joint distribution of the form
    \[
    m^{0}(\by,\bx,\bgmao)=p^0(\bx)f^{0}(\by\mid\bx,\bgmao)
    \]
    with
    \[
    f^{0}(\by\mid\bx,\bgmao)=\int_{\Th}\psi(\by,\bth,\bgmao)dP^{0}_{\bx}(\bth),
    \]
    where \(P^0\in C_S(\mX,\muY)\) and \(\bgma_0\in\Gma\).
\end{theorem}

\section{Concluding remarks}

We have defined a DDP for general Polish spaces and introduced some parsimonious variants that may be desirable on specific applications. Furthermore, we provided sufficient conditions for different versions of DDP defined on general Polish spaces to have appealing prior theoretical properties regarding the continuity of their sample paths under different topologies, and the continuity of their autocovariance function and other, more general measures of association. 

These properties are of practical importance because they ensure that different versions of the model can combine and borrow strength across sparse data sources regarding the predictors and, therefore, avoid the need of replicates of the responses for every value of the predictors to obtain adequate estimates of the predictor-dependent probability distributions. 

Furthermore, we studied induced mixture models arising from a DDP or any of its variants. We provided sufficient conditions that ensure the dependent Dirichlet processes mixture model has a continuous density. This case is of practical interest in statistical applications. In addition, we provided sufficient conditions under which dependent Dirichlet processes mixture models have large full or large support, considering different topologies, and study the behavior of the posterior distribution under i.i.d. joint sampling of responses and predictors. The study of stronger consistency results and concentration rates is the subject of ongoing research.   

Finally, the results provided in this article can be easily extended to more general dependent stick-breaking processes.

\section*{Acknowledgments}
The authors are grateful for helpful discussions with Prof. Judith Rousseau. The authors also thank Prof. Rousseau for critical reading on an early version of the manuscript. 
    
\section*{Funding}

A. Iturriaga was supported by ``Becas de Doctorado Nacional de CONICYT.''  A.~Jara’s work was supported by Agencia Nacional de Investigación y Desarrollo (ANID) through the Fondo Nacional de Desarrollo Científico y Tecnológico (FONDECYT) grant No 1220907 and through grant NCN17\_059 from Millennium Science Initiative Program, Millennium Nucleus Center for the Discovery of Structures in Complex Data (MIDAS). C.~A.~Sing~Long was supported by by Agencia Nacional de Investigación y Desarrollo (ANID) through the Fondo Nacional de Desarrollo Científico y Tecnológico (FONDECYT) grant No 1211643 and through grant NCN17\_ 059 from Millennium Science Initiative Program, Millennium Nucleus Center for the Discovery of Structures in Complex Data (MIDAS). 

\appendix

\section{Proof of the main results}\label{apx:proofs}

\subsection{The continuity of the DDP and its variants}

\subsubsection{Proof of Theorem \ref{thm:continuity:weak}}
\label{thm:continuity:weak:proof}

By hypothesis, there is a set \(\Omega_0\subset \Omega\) of full measure such that for any \(\omega\in \Omega_0\) we have: (i) \(\piiw \geq 0\); (ii) \(\sum_{i\in\N}\piiw \equiv 1\); (iii) \(\piiw\) and \(\bthiw\) are continuous on \(\mX\) in the case of a DDP, \(\bthiw\) is continuous on \(\mX\) in the case of a \(w\)DDP, and \(\piiw\) is continuous on \(\mX\) in the case of a \(\th\)DDP. 

First, we prove the statement in the case of a DDP process. Fix \(\omega\in \Omega_0\) and \(\bx_0\in \mX\). Let \(f\in \Cb(\Theta)\) and suppose, without loss, that \(\nrmC{f} \leq 1\). Fix \(\eps > 0\) and let \(N_\eps\in \N\) be such that
\[
\sum_{i>N_\eps} \piixow < \frac{1}{4}\eps. 
\]
Define the set \(U_\eps\) as
\[
U_\eps := \bigcap_{i=1}^{N_\eps}\Lset{\bx\in \mX:\,\, |\piixw - \piixow| < \frac{1}{4N_\eps}\eps\quad\mbox{and}\quad |f(\bthixw)-f(\bthixow))| < \frac{1}{4N_\eps}\eps}.
\]
Since \(\piiw\) and \(f\circ \bthiw\) are continuous on \(\Theta\) we conclude \(U_\eps\) is an open neighborhood of \(\bx_0\). Observe that for any \(\bx\in U_\eps\) we have
\[
\sum_{i > N_\eps} \piixw = 1 - \sum_{i = 1}^{N_\eps}\piixw = \sum_{i>N_\eps} \piixow - \sum_{i = 1}^{N_\eps}(\piixw-\piixow) < \frac{1}{2}\eps.
\]
Let \(\bx\in U_\eps\) and consider the decomposition
\[
\int f(\bth) d\Gxw(\bth) - \int f(\bth) d\Gxow(\bth) = \left(\sum_{i=1}^{N_\eps} + \sum_{i> N_\eps}\right) (\piixw f(\bthixw) - \piixow f(\bthixow)).
\]
The first sum on the right-hand side can be bounded as
\begin{align*}
    \left|\sum_{i=1}^{N_\eps} (\piixw f(\bthixw) - \piixow f(\bthixow))\right| &\leq \sum_{i=1}^{N_\eps} |\piixw - \piixow| + \sum_{i=1}^{N_\eps} |f(\bthixw) - f(\bthixow)|\\
    &< \frac{1}{2}\eps,
\end{align*}
whereas the second sum can be bounded as
\[
\left|\sum_{i>N_\eps} (\piixw f(\bthixw) - \piixow f(\bthixow))\right| \leq \sum_{i>N_\eps} \piixw +  \sum_{i>N_\eps}\piixow
< \frac{1}{2}\eps.
\]
Consequently,
\[
\left|\int f(\bth) d\Gxw(\bth) - \int f(\bth) d\Gxow(\bth)\right| < \eps
\]
for any \(\bx\in U_\eps\). Since \(U_\eps\) is open, the claim follows in the case of a DDP. 

Remark that in the case of a \(w\)DDP or the case of a \(\th\)DDP the previous arguments can be adapted with minor modifications to prove the claim. We omit the details for brevity.\qed

\subsubsection{Proof of Theorem~\ref{thm:continuity:notStrongForDDP}}
\label{thm:continuity:notStrongForDDP:proof}

By hypothesis, there is a set \(\Omega_0\subset \Omega\) of full measure such that for any \(\omega\in \Omega_0\) we have: (i) \(\piiw \geq 0\); (ii) \(\sum_{i\in\N}\piiw \equiv 1\); (iii) \(\piiw\) and \(\bthiw\) are continuous on \(\mX\) in the case of a DDP, \(\bthiw\) is continuous on \(\mX\) in the case of a \(w\)DDP, and \(\piiw\) is continuous on \(\mX\) in the case of a \(\th\)DDP. 

First, we prove the statement in the case of a DDP process. Fix \(\omega\in\Omega_0\) and let \(\eps > 0\). Let \(N_\eps\) be such that
\[
\sum_{i > N_\eps} \piixow < \frac{1}{8}\eps
\]
Define the measurable set
\[
B := \set{\bthixow:\, i\in\bset{N_\eps}}
\]
and let \(f = \mathbb{I}_{B}\). Then
\[
F(\bx):= \int f(\bth) d\Gxw(\bth)
\]
if continuous at \(\bx_0\) by hypothesis. Hence, the set
\[
U_\eps := \Lset{\bx\in \mX:\, |F(\bx) - F(\bxo)| < \frac{1}{4}\eps}\cap\bigcap_{i=1}^{N_\eps}\Lset{\bx\in \mX:\,\, \piixw - \piixow| < \frac{1}{8N_\eps}\eps}
\]
is open. For any \(\bx\in U_\eps\) we have
\[
\sum_{i > N_\eps}\piixw < \frac{1}{4}\eps
\]
and
\begin{align*}
    \left|\sum_{i=1}^{N_\eps}\piixow (\delta_{\bthixw}(B) - 1)\right| &\leq \sum_{i=1}^{N_\eps}\piixw - \piixow| + \left|\sum_{i=1}^{N_\eps}(\piixw \delta_{\bthixw}(B) - \piixow \delta_{\bthixow}(B))\right|\\
    &< \frac{1}{8}\eps + \sum_{i>N_\eps}(\piixow + \piixw) + |F(\bx) - F(\bx_0)|\\
    &< \eps. 
\end{align*}
It is clear that at least one \(\bthixw\) belongs to \(B\). Note \(B\) is a finite set. Hence, by continuity, such \(\bthixw\) must be constant on \(U_\eps\). This proves the claim in the case of a DDP. For the \(w\)DDP the previous arguments can be adapted with minor modifications to prove the claim. We omit the details for brevity.\qed

\subsubsection{Proof of Theorem~\ref{thm:continuity:uniform}}
\label{thm:continuity:uniform:proof}

In the case of a \(\th\)DDP, for every \(\omega\in\Omega\) we have
\[
\sup_{B\in\mB(\Th)}\, |\Gxw(B) - \Gxow(B)| \leq \sum_{i=1}^{\infty} \piixw - \piixow|.
\]
Let \(\Omega_0\subset\Omega\) be a set of full measure such that \(\omega\in \Omega_0\) implies \(\Viw\) is continuous for every \(i\in \N\). Fix \(\omega\in \Omega_0\). Then \(\piiw\) is continuous for every \(i\in \N\). We follow a similar argument as in the proof of  Theorem~\ref{thm:continuity:weak} in Appendix~\ref{thm:continuity:weak:proof}. Let \(\eps >0\) and let \(N_\eps\in \N\) be such that
\[
\sum_{i>N_\eps} \piixow < \frac{1}{4}\eps. 
\]
Define the open neighborhood \(U_\eps\) of \(\bx_0\) as
\[
U_\eps := \bigcap_{i=1}^{N_\eps}\Lset{\bx\in \mX:\,\, \piixw - \piixow| < \frac{1}{4N_\eps}\eps}.
\]
For any \(\bx\in U_\eps\) we have
\[
\sum_{i>N_\eps}\piixw = \sum_{i>N_\eps} \piixow - \sum_{i = 1}^{N_\eps}(\piixw-\piixow) < \frac{1}{2}\eps
\]
whence
\[
\sum_{i=1}^{\infty} |\piixw - \piixow| \leq \sum_{i=1}^{N_\eps} |\piixw - \piixow| + \sum_{i> N_\eps} (\piixw + \piixow) < \eps.
\]
Consequently, for any \(\bx\in U_\eps\) we have
\[
\sup_{B\in\mB(\Th)}\, |\Gxw(B) - \Gxow(B)| < \eps
\]
from where the theorem follows.\qed

\subsection{The support of the DDP and its variants}

We first prove the following auxiliary result.

\begin{lemma}\label{lem:support:approximation}
    Let \(P:\mX\to \mP(\Th)\) and \(f_1,\ldots, f_n\in \LInf(\Th)\). Define
    \[
    F_i(\bx) : = \int_{\Th} f_i(\bth)\, dP_{\bx}(\bth).
    \]
    Let \(K\subset \mX\) be compact. If \(F_1,\ldots, F_n: K\to \R\) are continuous, then for any \(\eps > 0\) there exists \(\olP:K \to \mP(\Th)^{\mX}\) such that:
    \begin{enumerate}[leftmargin=24pt]
        \item{For any \(\bx\in K\) we have
            \[
            \left|\int_{\Th} f_i(\bth)\, dP_{\bx}(\bth) - \int_{\Th} f_i(\bth)\, d\olP_{\bx}(\bth)\right| < \eps
            \]
        }
        \item{For any \(B\in \mB(\Th)\) the map \(\bx\mapsto \olP_{\bx}(B)\) is Lipschitz continuous on \(K\).
        }
        \item{The collection \(\set{\olP_{\bx}:\,\bx\in K}\) is tight.
        }
    \end{enumerate}
\end{lemma}

\begin{proof}[Proof of Lemma~\ref{lem:support:approximation}]
    Let \(\eps > 0\). We begin by constructing a suitable partition of unity in \(K\). Since \(F_1,\ldots, F_n\) are uniformly continuous on \(K\) there exists \(\delta > 0\) such that
    \[
    \forall\, \bx,\bx'\in K,\, i\in\bset{n}:\,\, \dX(\bx,\bx') < \delta\,\,\Rightarrow\,\, |F_i(\bx) - F_i(\bx')| < \eps.
    \]
    Let \(r < \delta/2\). From the open cover \(\set{B(\bx, r)}_{\bx\in K}\) we can extract the finite subcover \(\set{B(\bx_k,r)}_{k=1}^{N_K}\). We construct a continuous partition of the unity on \(K\) surrogate to this cover as follows. Define the continuous functions
    \[
    \olvphi_{k}(\bx) = \max\left(0, 1 - \frac{\dX(\bx, \bx_{k})}{r}\right).
    \]
    for \(k\in\set{1,\ldots, N_K}\). Note that \(\bx\notin B(\bx_k,r)\) implies \(\olvphi_k(\bx) = 0\). It can be verified that
    \[
    \sum_{k=1}^{N_K} \olvphi_{k}(\bx) \geq c_{\min} > 0
    \]
    for any \(\bx \in K\). Therefore, we define the continuous functions
    \[
    \vphi_k(\bx) = \frac{\olvphi_k(\bx)}{\sum_{k=1}^{N_K} \olvphi_{k}(\bx)}
    \]
    which satisfy \(\sum_{k=1}^{N_K} \vphi_k \equiv 1\) over \(K\). Define
    \[
    \olF_i(\bx) := \sum_{k=1}^{N_K} F_i(\bx_k)\vphi_k(\bx)\quad\mbox{and}\quad \olP_{\bx} := \sum_{k=1}^{N_K} \vphi_k(\bx) P_{\bx_k}.
    \]
    These functions satisfy
    \[
    \olF_i(\bx) := \int_{\Theta} f_i(\bth)\left(\sum_{k=1}^{N_K}\vphi_k(\bx) dP_{\bx_k}(\bth)\right) = \int_{\Theta} f_i(\bth) d\olP_{\bx}(\bth). 
    \]
    To prove (a) note that \(\olF_i\) satisfies
    \[
    |\olF_i(\bx) - F_i(\bx)| \leq \sum_{k:\, \bx\in B(\bx_{k},r)} |F_i(\bx_k) - F_i(\bx)|\vphi_k(\bx) < \eps
    \]
    for any \(\bx\in K\). To prove (b) note that for any \(B\in\mB(\Th)\) we have
    \[
    \olP_{\bx}(B) = \sum_{k=1}^{N_K} \vphi_i(\bx) \olP_{\bx_k}(B)
    \]
    which is continuous by construction. To prove it is Lipschitz, note that
    \begin{align*}
        |\vphi_i(\bx) - \vphi_i(\bx')| &\leq \frac{|\olvphi_i(\bx) - \olvphi_i(\bx')|}{\sum_{j=1}^{N_K}\olvphi_j(\bx)} + \frac{\olvphi_j(\bx')}{\sum_{j=1}^{N_K}\olvphi_j(\bx)\sum_{k=1}^{N_K}\olvphi_k(\bx')}\sum_{\ell=1}^{N_K}|\olvphi_{\ell}(\bx) - \olvphi_\ell(\bx')| \\
        &\leq c_{\min}^{-1}(N_K c_{\min}^{-1} + 1) \dX(\bx,\bx').
    \end{align*}
    Finally, to prove (c) note that \(\Th\) is Polish. Hence, the collection \(P_{\bx_1},\ldots, P_{\bx_{N_K}}\) is tight. For any \(\eps > 0\) there exists \(K_\Th \subset \Th\) such that
    \[
    P_{\bx_k}(K_{\Theta}) > 1 - \eps. 
    \]
    However, since \(\vphi_i \geq 0\) we have
    \[
    \olP_{\bx}(B) = \sum_{k=1}^{N_K} \vphi_i(\bx) \olP_{\bx_k}(B) > 1  - \eps,
    \]
    proving the claim.
\end{proof}

\subsubsection{Proof of Theorem~\ref{thm:support:product-weak}}
\label{thm:support:product-weak:proof}

By possibly modifying the sequence \(\eps_{1,1},\ldots,\eps_{n,n}\) we can assume, without loss, that \(\nrmC{f_{i,j}} \leq 1\). Let \(\Omega_0\subset \Omega\) be a set of full measure such that for any \(\omega\in \Omega_0\) we have: (i) \(\piiw \geq 0\); (ii) \(\sum_{i\in\N}\piiw \equiv 1\); (iii) \(\piiw\) and \(\bthiw\) are continuous on \(\mX\) in the case of a DDP, \(\bthiw\) are continuous on \(\mX\) in the case of a \(w\)DDP, or \(\piiw\) are continuous on \(\mX\) in the case of a \(\th\)DDP.

Let \(\eps > 0\) be such that \(\eps < \eps_0:=\min(\eps_{1,1},\ldots, \eps_{n,n})\) and define the event
\[
\Lset{\omega\in\Omega:\,\, \left|\int_\Theta f_{i,j}(\bth) d\Gw_{\bx_j}(\bth) - \int_\Theta f_{i,j}(\bth) dP^0_{\bx_j}(\bth)\right| < \eps_{0},\,\, i,j\in \bset{n}}.
\]
By inspection, the event is measurable. It suffices to show that this event has positive probability. The first step of the proof is the same regardless of the variant. 

Since \(\Th\) is Polish, the collection \(P^0_{\bx_1},\ldots, P^0_{\bx_n}\) is tight. Let \(K_{\Th}\subset \Th\) be a compact set such that \(P^0_{\bx_j}(\Th\setminus K_{\Th}) < \eps\) for \(j\in\set{1,\ldots, n}\). Since \(f_{1,1},\ldots,f_{n,n}\) are uniformly continuous on \(K_{\Th}\), there exists \(\delta>0\) such that
\[
\forall\ \bth,\bth'\in K_{\Th},\, i,j\in\set{1,\ldots,n}:\,\, \dTh(\bth,\bth') < \delta\quad\Rightarrow\quad |f_{i,j}(\bth) - f_{i,j}(\bth')| < \eps.
\]
Let \(r < \delta/2\). Then \(\set{B(\bth,r)}_{\bth\in K_{\Th}}\) is an open cover of \(K_{\Th}\) and we can extract a finite subcover \(\set{B(\bth_k, r)}_{k=1}^{\NTh}\). By possibly removing elements, we may assume no ball is covered by the union of the remanining ones. This subcover induces the partition
\begin{align}
    \label{eq:support:compactThetaPartition}
    \begin{split}
        A_1 &:= K_{\Th} \cap B(\bth_k,r)\\
        A_k &:= K_{\Th} \cap \left(B(\bth_k,r_f) \setminus\bigcup_{\ell
            = 1}^{k-1}B(\bth_{\ell},r)\right)\qquad k\in\set{2,\ldots, N_\Theta}.
    \end{split}
\end{align}
of \(K_\Theta\). Remark that no \(A_k\) is empty and that \(A_k \subset B(\bth_k,r)\). Therefore, no \(f_i\) varies by more than \(2\eps\) over any \(A_k\). Note that
\[
\left|\int_{\Th\setminus K_{\Th}} f_{i,j}(\bth) dP_{\bx_j}^0(\bth)\right| \leq P^0_{\bx_j}(\Th\setminus K_{\Th}) < \eps,
\]
and
\[
\left|\int_{K_{\Th}} f_{i,j}(\bth) dP_{\bx_i}^0(\bth) - \sum_{k=1}^{\NTh} f_{i,j}(\bth_k) P_{\bx_j}^0(A_k)\right| < 2\eps \sum_{k=1}^{\NTh} P_{\bx_j}(A_k) \leq 2\eps.
\]
Therefore,
\begin{multline*}
    \left|\int_{\Th} f_{i,j}(\bth) d\Gw_{\bx_j}(\bth) - \int_{\Th} f_{i,j}(\bth) dP_{\bx_j}^0(\bth)\right| \\
    < 3\eps + \left|\int_{\Th} f_{i,j}(\bth) d\Gw_{\bx_j}(\bth) - \sum_{k=1}^{\NTh} f_{i,j}(\bth_k) P_{\bx_j}^0(A_k)\right|.
\end{multline*}
Since \(f_{1,1},\ldots, f_{n,n}\) are continuous, define the open subsets
\[
U_k := \bigcap_{j=1}^{n}\bigcap_{i=1}^{n}\Lset{\bth\in \Theta:\,\, |f_{i,j}(\bth) - f_{i,j}(\bth_k)| < \frac{1}{N_\Theta}\eps}
\]
and the event
\[
\Omega_\pi :=  \bigcap_{k=1}^{\NTh}\bigcap_{j=1}^{n}\Lset{\omega\in \Omega_0:\, |\piw_{k,\bx_j} - P^0_{\bx_j}(A_k)| < \frac{1}{\NTh}\eps}
\]
which has positive measure by hypothesis. Note that for any \(\omega\in \Omega_\pi\) we have
\[
\sum_{k > \NTh} \piw_{k,\bx_j} = 1 - \sum_{k=1}^{\NTh} P^0_{\bx_j}(A_k) - \sum_{k=1}^{\NTh} (\piw_{k,\bx_j} - P^0_{\bx_j}(A_k)) < 2\eps
\]
for \(i\in\bset{\NTh}\).

The second step of the proof changes slightly in the case of a DDP, a \(w\)DDP, and the \(\th\)DDP. We present it first in the case of a DDP. For \(\omega\in \Omega_\pi\) we have
\begin{multline*}
    \left|\int_{\Th} f_{i,j}(\bth) d\Gw_{\bx_j}(\bth) - \sum_{k=1}^{\NTh} f_{i,j}(\bth_k) P_{\bx_j}^0(A_k)\right| \\ 
    \leq \sum_{k > \NTh}\piw_{k,\bx_j} + \left|\sum_{k = 1}^{\NTh} (\piw_{k,\bx_j} f_{i,j}(\bthw_{k,\bx_j}) - P^0_{\bx_j}(A_k) f_{i,j}(\bth_{k}))\right|\\
    \leq 3\eps + \sum_{k = 1}^{\NTh} P^0_{\bx_j}(A_k) |f_i(\bthw_{k,\bx_j}) - f_{i,j}(\bth_{k})|.
\end{multline*}
Consider the event
\[
\Omega_\theta := \bigcap_{k=1}^{\NTh}\bigcap_{j=1}^{n}\set{\omega\in \Omega_0:\,\, \bthw_{k,\bx_j}\in U_k},
\]
which has positive measure by hypothesis as each \(U_k\) is open. Since the events \(\Omega_\pi\) and \(\Omega_\th\) are independent, the event \(\Omega_\pi\cap\Omega_\th\) has positive measure. Hence, for any \(\omega\in\Omega_\pi\cap\Omega_\th\) we have
\[
\left|\int_{\Th} f_{i,j}(\bth) d\Gw_{\bx_j}(\bth) - \sum_{k=1}^{\NTh} f_{i,j}(\bth_k) P_{\bx_j}^0(A_k)\right| < 4\eps
\]
whence
\[
\left|\int_{\Th} f_{i,j}(\bth) d\Gw_{\bx_j}(\bth) - \int_{\Th} f_{i,j}(\bth) dP_{\bx_j}^0(\bth)\right| < 7\eps
\]
for any \(i,j\in\bset{n}\). This proves the theorem in the case of a DDP. 

To prove the theorem in the case of a \(\th\)DDP, the key inequality is
\[
\left|\int_{\Th} f_{i,j}(\bth) d\Gw_{\bx_j}(\bth) - \sum_{k=1}^{\NTh} f_{i,j}(\bth_k) P_{\bx_j}^0(A_k)\right| < 3\eps + \sum_{k = 1}^{\NTh} P^0_{\bx_j}(A_k) |f_i(\bthw_{k}) - f_{i,j}(\bth_{k})|.
\]
Note that it suffices to consider the event 
\[
\Omega_\theta := \bigcap_{k=1}^{\NTh}\set{\omega\in \Omega_0:\,\, \bthw_{k}\in U_k}
\]
which has positive measure as the \(U_k\) are open. The rest of the argument is essentially the same as in the case of a DDP.

The proof in the case of a \(w\)DDP follows first the same step as in the other cases. However, the remainder of the proof is different and slighly more involved. Let \(M\in\N\) be such that there are integers \(m_{j,k}\in\bset{M}\) such that
\[
\left|P^0_{\bxj}(A_k) - \frac{m_{j,k}}{M}\right| < \frac{\eps}{2\NTh}.
\]
Note that for this choice, for any \(j\in\bset{n}\),
\[
\sum_{k=1}^{\NTh}\frac{m_{j,k}}{M} < \sum_{k=1}^{\NTh}P^0_{\bxj}(A_k) + \frac{1}{2}\eps \leq 1 + \frac{1}{2}\eps
\]
and
\[
\sum_{k=1}^{\NTh}\frac{m_{j,k}}{M} \geq \sum_{k=1}^{\NTh}P^0_{\bxj}(A_k) - \frac{\eps}{2} = P^0_{\bxj}(K_{\Th}) -\frac{\eps}{2} \geq 1 - \frac{3}{2}\eps. 
\]
Define the event 
\[
\Omega_\pi :=  \bigcap_{i=1}^{M}\Lset{\omega\in \Omega_0:\, \piiw \in \left(\frac{1-\eps/2\NTh}{M}, \frac{1}{M}\right] }
\]
which has positive measure by hypothesis. Note that for \(\omega\in \Omega_\pi\) we have
\[
\left(1-\frac{\eps}{2\NTh}\right)\frac{M'}{M}< \sum_{i=1}^{M'} \piiw \leq \frac{M'}{M}
\]
for any \(M'\in\bset{M}\). Hence, for every \(j\) we have
\[
\left|P^0_{\bxj}(A_k) - \sum_{i\in I_{j,k}} \piiw\right| \leq \left|P^0_{\bxj}(A_k) - \frac{m_{j,k}}{M}\right| + \left| \frac{m_{j,k}}{M}- \sum_{i\in I_{j,k}} \piiw\right| < \frac{\eps}{2\NTh}+ \frac{\eps}{2\NTh} = \frac{\eps}{\NTh}
\]
where \(I_{j,k}\subset \bset{M}\) is an arbitrary subset such that \(|I_{j,k}| = m_{j,k}\). Therefore, for every \(j\in \bset{n}\) we let \(\set{I_{j,k}}_{k=1}^{\NTh}\) be a collection of disjoint sets \(\bset{M}\) for which \(|I_{j,k}| = m_{j,k}\) and we let \(I^0_j\) be the complement of their union. Note that
\[
\sum_{i\in I_j^0} \piiw = \sum_{i=1}^M \piiw - \sum_{k=1}^{\NTh} \sum_{i\in I_{j,k}}\piiw < 1 + \frac{1}{2}\eps - P^0_{\bx_j}(K_{\Th}) < \frac{3}{2}\eps
\]
and
\[
\sum_{i > M} \piiw = 1 - \sum_{i = 1}^{M}\piiw < \frac{1}{2\NTh}\eps < \frac{1}{2}\eps.
\]
Consequently, in the case of a \(w\)DDP we obtain the inequality
\begin{multline*}
    \left|\int_{\Th} f_{i,j}(\bth) d\Gw_{\bx_j}(\bth) - \sum_{k=1}^{\NTh} f_{i,j}(\bth_k) P_{\bx_j}^0(A_k)\right| \\\leq \sum_{\ell > M}\piw_{\ell} + \sum_{\ell\in I_j^0} \piw_{\ell} + \sum_{k=1}^{\NTh}\left|\sum_{\ell \in I_{j,k}} \piw_{\ell} f_{i,j}(\bthw_{\ell,\bx_j}) - P^0_{\bx_j}(A_k) f_{i,j}(\bth_{k})\right| \\
    < 2\eps +\sum_{k=1}^{\NTh}\left|\sum_{\ell \in I_{j,k}} \piw_{\ell} (f_{i,j}(\bthw_{\ell,\bx_j}) -  f_{i,j}(\bth_{k}))\right| + \sum_{k=1}^{\NTh}|f_{i,j}(\bth_k)|\left|\sum_{\ell \in I_{j,k}} \piw_{\ell} - P^0_{\bx_j}(A_k)\right|\\
    < 3\eps + \sum_{k=1}^{\NTh}\left|\sum_{\ell \in I_{j,k}} \piw_{\ell} (f_{i,j}(\bthw_{\ell,\bx_j}) -  f_{i,j}(\bth_{k}))\right|.
\end{multline*}
Hence, consider the event
\[
\Omega_\theta := \bigcap_{k=1}^{\NTh}\bigcap_{j=1}^{n}\bigcap_{i\in I_{j,k}}\set{\omega\in \Omega_0:\,\, \bthiw \in U_k},
\]
which has positive measure by hypothesis. By independence, \(\Omega_\pi\cap\Omega_\th\) has positive measure, and for \(\omega\in\Omega_\pi\cap\Omega_\th\) we have
\[
\sum_{k=1}^{\NTh}\left|\sum_{\ell \in I_{j,k}} \piw_{\ell} (f_{i,j}(\bthw_{\ell,\bx_j}) -  f_{i,j}(\bth_{k}))\right| < \frac{\eps}{\NTh}\sum_{k=1}^{\NTh} \sum_{\ell \in I_{j,k}} \piw_{\ell} < \frac{\eps}{\NTh}\left(\eps + \sum_{k=1}^{\NTh} P_{\bxj}^0(A_k)\right) < \eps.
\]
Therefore, 
\[
\left|\int_{\Th} f_{i,j}(\bth) d\Gw_{\bx_j}(\bth) - \sum_{k=1}^{\NTh} f_{i,j}(\bth_k) P_{\bx_j}^0(A_k)\right| < 4\eps,
\]
proving the theorem in the case of a \(w\)DDP.

\qed

\subsubsection{Proof of Theorem~\ref{thm:support:compact-weak}}
\label{thm:support:compact-weak:proof}

By possibly modifying the sequence \(\eps_{1},\ldots,\eps_{n}\) we can assume, without loss, that \(\nrmC{f_i} \leq 1\). Let \(\Omega_0\subset \Omega\) be a set of full measure such that for any \(\omega\in \Omega_0\) we have: (i) \(\piiw \geq 0\); (ii) \(\sum_{i\in\N}\piiw \equiv 1\); (iii) \(\piiw\) and \(\bthiw\) are continuous on \(\mX\) in the case of a DDP, or \(\piiw\) are continuous on \(\mX\) in the case of a \(\th\)DDP. 

For simplicity, we write for \(\omega\in \Omega\)
\[
\Fw_i(\bx) := \int_{\Theta} f_i(\bth)\, d\Gxw(\bth).
\]
By Theorem~\ref{thm:continuity:weak}, \(\Fw_i\) is continuous on \(\mX\) for almost every \(\omega\in\Omega\). By possibly removing from \(\Omega_0\) a null set, we may assume \(\Fw_i\) is continuous for every \(\omega\in\Omega_0\) and every \(i\in\bset{n}\). Furthermore, we write
\[
F_i(\bx) := \int_{\Theta} f_i(\bth) dP_{\bx}(\bth)
\]
which, by hypothesis, is a continuous function on \(\mX\). Since \(\nrmC{f_i}\leq 1\) we have \(|\Fw_i| \leq 1\) and \(|F_i|\leq 1\). 

Let \(\eps_0 < \min(\eps_1,\ldots,\eps_n)\) and define the event
\[
\Omega_{K,\eps_0} := \bigcap_{i=1}^{N}\Lset{\omega\in\Omega_0:\,\, \sup_{\bx\in K}\,|\Fw_i(\bx) - F_i(\bx))| < \eps_0}.
\]
Since \(\Th\) is separable, a compact \(K\subset \Th\) is also separable. By hypothesis \(F_i\) is continuous and by construction \(\Fw\) is continuous for every \(\omega \in \Omega_0\). Hence, the event is measurable, and the theorem follows if we show that this event has positive measure. Let \(\eps>0\) with \(\eps < \eps_0\). Since \(K\) is compact, by Lemma~\ref{lem:support:approximation} there is a function \(\olP: K\to \mP(\Th)\) such that for
\[
\olF_i(\bx) := \int_{\Theta} f_i(\bth) d\olP_{\bx}(\bth)
\]
we have
\[
\sup_{\bx\in K}|F_i(\bx) - \olF_i(\bx)| < \eps. 
\]
Hence, 
\[
|\Fw_i(\bx) - F_i(\bx))| < |\Fw_i(\bx) - \olF_i(\bx)| + \eps. 
\]
for any \(i\in\set{1,\ldots, n}\) and \(\bx\in K\). Since \(\set{\olP_{\bx}:\,\bx\in K}\) is tight by Lemma~\ref{lem:support:approximation}, there exists \(K_{\Th}\subset \Th\) compact such that \(\olP_{\bx}(\Th\setminus K_{\Th}) < \eps\) for any \(\bx\in K\). Hence, 
\[
\left|\olF_i(\bx) - \int_{K_{\Th}} f_i(\bth) \olP_{\bx}(\bth) \right| \leq \olP_{\bx}(\Th\setminus K_{\Th}) < \eps
\]
and
\[
|\Fw_i(\bx) - F_i(\bx)| \leq |\Fw_i(\bx) - \olF_i(\bx)| + \eps \leq \left|\Fw_i(\bx) - \int_{K_{\Th}} f_i(\bth) d\olP_{\bx}(\bth)\right| + 2\eps
\]
on \(K\). Finally, since \(f_1,\ldots, f_n\) are continuous, we can use the same construction as that in the proof of Theorem~\ref{thm:support:product-weak} in Appendix~\ref{thm:support:product-weak:proof} to obtain a partition \(\set{A_k}_{k=1}^{\NTh}\) of \(K_{\Th}\) as in~\eqref{eq:support:compactThetaPartition} where each \(A_k\) is measurable, non-empty, and every \(f_i\) varies at most by \(2\eps\) over any \(A_k\). Then
\[
\left|\int_{K_{\Th}} f_i(\bth) d\olP_{\bx}(\bth) - \sum_{k=1}^{\NTh} f_i(\bth_k) \olP_{\bx}(A_k)\right| \leq \sum_{k=1}^{\NTh}\int_{A_k}|f_i(\bth) - f_i(\bth_k)|\,d\olP_{\bx}(\bth) < \eps
\]
whence
\[
|\Fw_i(\bx) - F_i(\bx)| < \left|\Fw_i(\bx) - \sum_{k=1}^{\NTh} f_i(\bth_k) \olP_{\bx}(A_k) \right| + 3\eps.
\]

Since \(\bx\to \olP_{\bx}(A_k)\) is continuous by Lemma~\ref{lem:support:approximation}, with values on \([0, 1]\), and
\[
1 -\eps < \sum_{k=1}^{\NTh} \olP_{\bx}(A_k) \leq 1
\]
the event
\[
\Omega_\pi:= \bigcap_{k=1}^{\NTh}\Lset{\omega\in \Omega_0:\,\, \sup_{\bx\in K} |\piw_{k,\bx} - \olP_{\bx}(A_k)| < \frac{1}{\NTh}\eps}
\]
is measurable, as \(\bx\mapsto \piw_{k,\bx}\) is continuous for \(\omega\in\Omega_0\) and \(\bx\mapsto \olP_{\bx}(A_k)\) is continuous by Lemma~\ref{lem:mixtures:uniformContinuityAux}, and has positive measure by hypothesis. The proof now proceeds exactly the same as the proof of Theorem~\ref{thm:support:product-weak} in Appendix~\ref{thm:support:product-weak:proof} for the DDP and \(\th\)DDP. We omit the details for brevity. \qed

\subsubsection{Proof of Theorem~\ref{thm:support:product-strong}}
\label{thm:support:product-strong:proof}

Let \(\Omega_0\subset \Omega\) be a set of full measure such that for any \(\omega\in \Omega_0\) we have: (i) \(\piiw \geq 0\); (ii) \(\sum_{i\in\N}\piiw \equiv 1\); (iii) \(\piiw\) and \(\bthiw\) are continuous on \(\mX\) for the DDP,  \(\bthiw\) are continuous on \(\mX\) in the case of a \(w\)DDP, or \(\piiw\) is continuous on \(\mX\) in the case of a \(\th\)DDP.

We prove in detail the statement in the case of a DDP. If \(P^0\in \mP(\Th)^{\mX}|_{\GX}\) then we consider the event 
\[
\Lset{\omega\in\Omega:\,\, \left|\int_\Theta f_{i,j}(\bth) d\Gw_{\bx_j}(\bth) - \int_\Theta f_{i,j}(\bth) dP^0_{\bx_j}(\bth)\right| < \eps_{i,j},\, i,j\in \bset{n}}
\]
which is measurable by inspection. A standard argument shows that, by possibly reducing \(\eps_{1,1},\ldots, \eps_{n,n}\), we can assume \(f_{i,j}\) are simple functions, and that there exists a partition \(\set{A_{k}}_{k=1}^{N_f}\) of \(\Th\) of measurable sets such that
\[
f_{i,j} = \sum_{k = 1}^{N_f} c_{i,j,k}\ind_{A_k}
\]
where \(\ind_{A}\) is the indicator function of the set \(A\) and \(|c_{i,j,k}| \leq 1\) for any \(i,j\in\set{1,\ldots,n}\) and \(k\in \set{1,\ldots, N_f}\). 

Let \(\eps_0 > 0\) be such that \(\eps_0 < \min(\eps_{1,1},\ldots, \eps_{n,n})\). Note that 
\[
\int_\Theta f_{i,j}(\bth) d\Gw_{\bx_j}(\bth) - \int_\Theta f_{i,j}(\bth) dP^0_{\bx_j}(\bth) = \sum_{k=1}^{N_f} c_{i,j,k} (\Gw_{\bxj}(A_k) - P^0_{\bxj}(A_k)). 
\]
Consider the event 
\[
\Omega_\pi := \bigcap_{k=1}^{\Nf}\bigcap_{j=1}^{n}\Lset{\omega\in \Omega_0:\, |\piw_{k,\bx_j} - P^0_{\bx_j}(A_k)| < \frac{1}{\Nf^2}\eps},
\]
which has positive measure by hypothesis. Remark that, in this case, for any \(j\in \bset{n}\) we have
\[
\sum_{i > \Nf} \piw_{i,\bxj} = 1 - \sum_{i=1}^{\Nf}\piw_{i,\bxj} = \sum_{i=1}^{\Nf}(P^0_{\bxj}(A_k) - \piw_{i,\bxj}) < \frac{1}{\Nf}\eps
\]
where we used the fact that \(\set{A_k}_{k=1}^{\Nf}\) is a partition. Hence,
\begin{align*}
    \left|\sum_{k=1}^{N_f} c_{i,j,k} (d\Gw_{\bxj}(A_k) - P^0_{\bxj}(A_k))\right| &\leq  N_f\sum_{i > \Nf}\piw_{i,\bx_j} + \sum_{k=1}^{\Nf} \left|\sum_{i=1}^{N_f} \piw_{i,\bxj} \delta_{\bthw_{i,\bxj}}(A_k) - P^0_{\bxj}(A_k)\right|\\
    &< \eps + \sum_{k=1}^{\Nf} \left|\sum_{i=1}^{N_f} \piw_{i,\bxj} \delta_{\bthw_{i,\bxj}}(A_k) - P^0_{\bxj}(A_k)\right|. 
\end{align*}
Now we use the fact that \(P^0\in \mP(\Th)^{\mX}|_{\GX}\). In this case, \(G^0_{\bx_j}(A_k) = 0\) implies \(P^0_{\bx_j}(A_k) = 0\) and no such terms do not contribute to the sum. We can define the event
\[
\Omega_{\th} := \set{\omega\in \Omega_0:\,\, \bthw_{k,\bx_j}\in A_k\,\,\mbox{and}\,\, G^0_{\bxj}(A_k) >0,\, j\in\bset{n},\, k\in\bset{N_f}}
\]
which has positive measure. By independence \(\Omega_\pi\cap\Omega_\th\) has positive measure. Hence, for \(\omega\in\Omega_\pi\cap\Omega_\th\) we have
\[
\left|\sum_{i=1}^{N_f} \piw_{i,\bxj} \delta_{\bthw_{i,\bxj}}(A_k) - P^0_{\bxj}(A_k)\right| = |\piw_{k,\bxj} - P^0_{\bxj}(A_k)| < \frac{1}{\Nf^2}\eps. 
\]
if \(G^0_{\bxj}(A_k) > 0\). If \(G^0_{\bxj}(A_k) = 0\) then \(P^0_{\bxj}(A_k) = 0\) and
\[
\left|\sum_{i=1}^{N_f} \piw_{i,\bxj} \delta_{\bthw_{i,\bxj}}(A_k) - P^0_{\bxj}(A_k)\right| \leq \sum_{k:\, G^0_{\bxj}(A_k) = 0}\piw_{i,\bxj} < \frac{1}{\Nf}\eps.
\]
Therefore,
\[
\left|\int_\Theta f_{i,j}(\bth) d\Gw_{\bx_j}(\bth) - \int_\Theta f_{i,j}(\bth) dP^0_{\bx_j}(\bth)\right| < 2\eps,
\]
proving the theorem in the case of a DDP. 

To prove the theorem in the case of a \(\th\)DDP, the argument is the same up to the inequality 
\[
\left|\sum_{k=1}^{N_f} c_{i,j,k} (d\Gw_{\bxj}(A_k) - P^0_{\bxj}(A_k))\right| < \eps + \sum_{k=1}^{\Nf} \left|\sum_{i=1}^{N_f} \piw_{i,\bxj} \delta_{\bthiw}(A_k) - P^0_{\bxj}(A_k)\right|. 
\]
In this case, we define the event 
\[
\Omega_{\th} := \set{\omega\in \Omega_0:\,\, \bthw_{k}\in A_k\,\,\mbox{and}\,\, G^0(A_k) >0,\, k\in\bset{N_f}}.
\]
Since \(P^0\in \mP(\Th)^{\mX}|_{\GX}\) this event has positive measure. The proof then follows the same steps as those in the case of a DDP. 

Finally, in the case of a \(w\)DDP we follow a similar argument as that on the proof of Theorem~\ref{thm:support:product-weak} in Appendix~\ref{thm:support:product-weak:proof}. By choosing a suitable \(M\) such that there are integers \(m_{j,k} \in \set{1,\ldots, M}\) such that
\[
\left|P^0_{\bxj}(A_k) - \frac{m_{j,k}}{M}\right| < \frac{\eps}{2\Nf}
\]
we can define the event 
\[
\Omega_\pi :=  \bigcap_{i=1}^{M}\Lset{\omega\in \Omega_0:\, \piiw \in \left(\frac{1-\eps/2\Nf}{M}, \frac{1}{M}\right] },
\]
which has positive measure by hypothesis, and for every \(j\in \set{1,\ldots,n}\) we can define a collection \(\set{I_{j,k}}_{k=1}^{\Nf}\) of disjoint sets \(\set{1,\ldots,M}\) for which \(|I_{j,k}| = m_{j,k}\). We let \(I^0_j\) be the complement of their union. Note that for \(\omega\in\Omega_\pi\) we have
\[
\sum_{i\in I_j^0} \piiw < \frac{3}{2}\eps\quad\mbox{and}\quad\sum_{i > M} \piiw< \frac{1}{2}\eps.
\]
Hence, we obtain the inequality 
\begin{align*}
    |\Gw_{\bxj}(A_k) - P^0_{\bxj}(A_k)| &\leq \sum_{\ell > M} \piw_{\ell} + \left|\sum_{\ell=1}^M \piw_{\ell} \delta_{\bthw_{\ell,\bxj}}(A_k) - P^0_{\bxj}(A_k)\right|\\
    &< \frac{3}{2\Nf}\eps + \sum_{\ell\in I^0_j} + \left|\sum_{k=1}^{\Nf}\sum_{\ell\in I_{j,k}} \piw_{\ell} \delta_{\bthw_{\ell,\bxj}}(A_k) - P^0_{\bxj}(A_k)\right|\\
    &<\frac{2}{\Nf}\eps + \left|\sum_{k=1}^{\Nf}\sum_{\ell\in I_{j,k}} \piw_{\ell} \delta_{\bthw_{\ell,\bxj}}(A_k) - P^0_{\bxj}(A_k)\right|.
\end{align*}
Since \(G^0_{\bxj}(A_j)= 0\) implies \(P^0_{\bxj}(A_k) = 0\) it suffices to consider the event
\[
\Omega_\theta := \set{\omega\in \Omega_0:\,\, \bthw_{i,\bxj} \in A_k\,\,\mbox{and}\,\, G^0_{\bxj}(A_k),\, k\in\bset{N_f}, j\in \bset{n},\, i\in I_{j,k}}.
\]
By hypothesis, this has positive measure and, by independence, so does \(\Omega_\pi\cap\Omega_\th\). The proof then follows the same arguments as those in the proof of Theorem~\ref{thm:support:product-weak} in Appendix~\ref{thm:support:product-weak:proof} in the case of a \(w\)DDP. We omit the details for brevity. \qed

\subsubsection{Proof of Theorem~\ref{thm:support:compact-strong}}
\label{thm:support:compact-strong:proof}

The proof is similar to the proof of Theorem~\ref{thm:support:compact-weak} in Appendix~\ref{thm:support:compact-weak:proof} with minor modifications. Let \(\Omega_0\subset \Omega\) be a set of full measure such that for any \(\omega\in \Omega_0\) we have: (i) \(\piiw \geq 0\); (ii) \(\sum_{i\in\N}\piiw \equiv 1\); (iii) \(\piiw\) and \(\bthiw\) are continuous on \(\mX\) in the case of a DDP, or \(\piiw\) is continuous on \(\mX\) in the case of a \(\th\)DDP. Furthermore, by possibly reducing \(\Omega_0\) by a null set, we may assume
\[
\bx \mapsto \int f_i(\bth) d\Gxw(\bth) 
\]
is continuous for every \(\omega\in\Omega_0\) and \(i\in\bset{n}\).

We prove the theorem in detail in the case of a DDP. Suppose \(C_S(\mX,\mP(\Th))\cap\mP(\Th)^{\mX}|_{\GX}\) is non-empty, as otherwise there is nothing to prove. Our goal is to show that for \(P^0\in C_S(\mX,\mP(\Th))\cap\mP(\Th)^{\mX}|_{\GX}\) the event 
\[
\Lset{\omega\in\Omega:\,\, \sup_{\bx\in K}\, \left|\int_\Theta f_{i}(\bth) d\Gw_{\bx}(\bth) - \int_\Theta f_{i}(\bth) dP^0_{\bx}(\bth)\right| < \eps_{i},\, i\in \bset{n}}
\]
has positive probability. As argued in the proof of Theorem~\ref{thm:support:compact-weak} in Appendix~\ref{thm:support:compact-weak:proof}, we can assume there exists a partition \(\set{A_{k}}_{k=1}^{\Nf}\) of \(\Th\) of measurable sets such that
\[
f_{i} = \sum_{k = 1}^{N_f} c_{i,k}\ind_{A_k}
\]
where \(|c_{i,k}| \leq 1\) for any \(i\in\bset{n}\) and \(k\in \bset{\Nf}\). Let \(\eps_0 > 0\) be such that \(\eps_0 < \min(\eps_{1},\ldots, \eps_{n})\). Note that
\[
\int_\Theta f_{i}(\bth) d\Gw_{\bx}(\bth) - \int_\Theta f_{i}(\bth) dP^0_{\bx}(\bth) = \sum_{k=1}^{N_f} c_{i,j,k} (\Gw_{\bx}(A_k) - P^0_{\bx}(A_k)). 
\]
By hypothesis, if \(G^0(A_k) = 0\) then \(P^0_{\bx}(A_k) = 0\) for any \(\bx\in K\). Hence, without loss, we can assume \(G^0(A_k) > 0\). Note that in this case, we have
\[
\sum_{k=1}^{\Nf}P^0_{\bx}(A_k) = 1. 
\]
Consider the event 
\[
\Omega_\pi := \bigcap_{k=1}^{\Nf}\bigcap_{j=1}^{n}\Lset{\omega\in \Omega_0:\, \sup_{\bx\in K}\, |\piw_{k,\bx} - P^0_{\bx}(A_k)| < \frac{1}{\Nf^2}\eps}.
\]
Since \(P^0\in C_S(\mX,\mP(\Th))\) we see that \(\bx\mapsto P_{\bx}^0(A_k)\) is continuous whence \(\Omega_\pi\) has positive measure by hypothesis. Remark that, in this case, for any \(j\in \bset{n}\) we have
\[
\sum_{i > \Nf} \piw_{i,\bx} = 1 - \sum_{i=1}^{\Nf}\piw_{i,\bx} = \sum_{i=1}^{\Nf}(P^0_{\bx}(A_k) - \piw_{i,\bxj}) < \frac{1}{\Nf}\eps
\]
where we used the fact that \(\set{A_k}_{k=1}^{\Nf}\) is a partition. Hence,
\[
\left|\sum_{k=1}^{N_f} c_{i,j,k} (d\Gw_{\bxj}(A_k) - P^0_{\bxj}(A_k))\right| < \eps + \sum_{k=1}^{\Nf} \left|\sum_{i=1}^{N_f} \piw_{i,\bxj} \delta_{\bthw_{i,\bxj}}(A_k) - P^0_{\bxj}(A_k)\right|. 
\]
Since \(P^0\in \mP(\Th)^{\mX}|_{\DDP}\), the event
\[
\Omega_{\th} := \set{\omega\in \Omega_0:\,\, \bthw_{k,\bx}\in A_k\,\,\mbox{and}\,\, G^0_{\bx}(A_k) >0,\,\bx\in K,\, k\in\set{1,\ldots,\Nf}}
\]
has positive measure. The proof then proceeds exactly as the  proof of Theorem~\ref{thm:support:compact-weak} in Appendix~\ref{thm:support:compact-weak:proof}. We omit the proof in the case of a \(\th\)DDP for brevity. \qed

\subsection{Association structure of the DDP and its variants}

\subsubsection{Proof of Theorem~\ref{thm:continuity:association}}
\label{thm:continuity:association:proof}

We prove the theorem in the case of a DDP as the arguments are the same in the case of a \(w\)DDP. To prove the theorem we proceed as follows. We first show that for any \(d\in \N\)
\[
(\bx_1,\ldots, \bx_d) \mapsto \ev\left(\prod_{k = 1}^dG_{\bx_k}(B)\right)
\]
is continuous. Then, by leveraging Stone-Weierstrass's theorem~\cite{Reed1980}, we can approximate any continuous \(f\) uniformly over the hypercube by a polynomial. Since the expectation of this polynomial is continuous by our first claim, the theorem follows. 

Define the sequence of functions \(h_n : \mX^d\times \Omega \to \R\) as
\[
h_n(\bx_1,\ldots,\bx_d,\omega) = \prod_{k = 1}^d\sum_{i_k=1}^{n} \piw_{i_k,\bx_{k}}\delta_{\bthw_{i_k,\bx_k}}(B) = \sum_{i_1,\ldots, i_d = 1}^n \prod_{k=1}^{d} \piw_{i_k,\bx_k} \delta_{\bthw_{i_{k},\bx_{k}}}(B).
\]
Since the summands are a.e. non-negative, \(h_n \leq h_{n+1}\). Furthermore, \(h_n\in [0, 1]\) for any \(n\in \N\), and
\[
\lim_{n\to\infty}h_n(\bx_1,\ldots,\bx_d,\omega) = \prod_{k=1}^d \Gw_{\bx_k}(B)
\]
for a.e. \(\omega\in\Omega\) and \(\bx_1,\ldots,\bx_d\in \mX\). 

We first control the expectation for fixed \(\bx_1,\ldots,\bx_d\in \mX\). Write \(h_n^{\bx}(\omega) = h_n(\bx_1,\ldots,\bx_d,\omega)\) for simplicity, and define the function \(g_n:\mX^d\to \R\) as
\begin{multline}
    \label{eq:pf:association:auxA}
    g_n(\bx_1,\ldots,\bx_d) = \ev(h_n^{\bx}) \\
    = \sum_{i_1,\ldots, i_m = 1}^n \ev\left(\prod_{k=1}^{m} \pi_{i_k,\bx_k}\right)\prob\left(\Lset{\omega\in\Omega:\, \bth_{i_1,\bx_1}\in B,\ldots, \bth_{i_d,\bx_d}\in B}\right)
\end{multline}
where in the last equality we used the hypothesis of independence. Since the sequence \(\set{h_n^{\bx}}\) is monotone non-decreasing, by the monotone convergence theorem we have
\[
\lim_{n\to\infty}g_n(\bx_1,\ldots,\bx_d) = g_\infty(\bx_1,\ldots,\bx_d) := \ev\left(\prod_{k = 1}^dG_{\bx_k}(B)\right).
\]
Hence, the sequence \(\set{g_n}\) converges pointwise to \(g_\infty\) over \(\mX^d\). To conclude \(g_\infty\) is continuous, we will show that \(\set{g_n}\) is a sequence of continuous functions, to then use a uniform approximation argument to show the continuity of \(g_\infty\) near any \((\bx_1,\ldots,\bx_d)\in \mX^d\).

To show \(g_n\) is continuous, note that for a.e. \(\omega\) the functions \(\bx\to \piixw\) are continuous for any \(i\in \N\). Therefore, for any \(i_1,\ldots, i_d\in \N\) the function
\[
(\bx_1,\ldots, \bx_d) \to \ev\left(\prod_{k=1}^{d} \pi_{i_k,\bx_k}\right)
\]
is continuous on \(\mX^d\). Similarly, 
\[
(\bx_1,\ldots, \bx_d) \to \prob\left(\Lset{\omega\in\Omega:\, \bth_{i_1,\bx_1}\in B,\ldots, \bth_{i_d,\bx_d}\in B}\right)
\]
is continuous by hypothesis.

We now show continuity of \(g_\infty\) near any \((\bx_1,\ldots, \bx_d)\in \mX^d\) using a uniform approximation. Fix \((\bx_1,\ldots, \bx_d)\in \mX^d\) and let \(\eps > 0\). Remark that for \(n > m\) we have the bound
\[
|h_{n}(\bx_1',\ldots,\bx_d',\omega) - h_{m}(\bx_1',\ldots,\bx_d',\omega)| \leq 
\sum_{k=1}^d \sum_{i_k = m+1}^{n}\piw_{i_k,\bx_k'}
\]
whence
\[
|g_{n}(\bx_1',\ldots,\bx_d') - g_m(\bx_1',\ldots,\bx_d')| \leq \sum_{k=1}^d \sum_{i_k = m+1}^{n}\ev(\pi_{i_k,\bx_k'}).
\]
Hence, we can control the differences on the left-hand side by controlling the expectations on the right-hand side. By similar arguments as those used in the proof of Theorem~\ref{thm:continuity:weak} in Appendix~\ref{thm:continuity:weak:proof}, for any \(\delta < \eps/3\) we can find \(N_\delta\in \N\) such that
\[
\sum_{i_k > N_\delta} \ev(\pi_{i_k,\bx_k}) < \frac{1}{2d}\delta. 
\]
for every \(k\in \set{1,\ldots, d}\). We define the open neighborhood
\[
U_\delta := \bigcap_{k=1}^d\bigcap_{i_k=1}^{N_\eps}\Lset{(\bx'_1,\ldots,\bx_d')\in \mX^d:\,\, |\ev(\pi_{i_k,\bx_k'}) - \ev(\pi_{i_k,\bx_k})| < \frac{1}{2d N_\delta}\delta}
\]
of \((\bx_1,\ldots,\bx_d)\). For \((\bx'_1,\ldots,\bx_d')\in U_{\delta}\) we have
\[
\sum_{i_k > N_\delta} \ev(\pi_{i_k,\bx_k'}) = \sum_{i_k>N_\delta} \ev(\pi_{i_k,\bx_k}) - \sum_{i_k=1}^{N_\delta}(\ev(\pi_{i_k,\bx_k}) -  \ev(\pi_{i_k,\bx_k'})) < \frac{1}{d}\delta.
\]
Therefore, for \(n > m > N_\delta\) we have
\[
\sup_{(\bx'_1,\ldots,\bx_d')\in U_{\delta}}\, |g_{n}(\bx_1',\ldots,\bx_d') - g_m(\bx_1',\ldots,\bx_d')| \leq \sup_{(\bx'_1,\ldots,\bx_d')\in U_{\delta}}\, \sum_{k=1}^d \sum_{i_k > N_\delta}\ev(\pi_{i_k,\bx_k'}) < \delta.
\]
In particular, this implies
\[
\sup_{(\bx'_1,\ldots,\bx_d')\in U_{\eps}}\, |g_{n}(\bx_1',\ldots,\bx_d') - g_{\infty}(\bx_1',\ldots,\bx_d')| < \delta
\]
for any \(n > N_\delta\). Hence, it suffices to choose \(n > N_\eps\) to define the open neighborhood
\[
W_\eps := \Lset{(\bx_1',\ldots,\bx_d')\in U_\delta:\, |g_n(\bx_1',\ldots,\bx_d') - g_n(\bx_1,\ldots,\bx_d)| < \frac{1}{3}\eps}
\]
of \((\bx_1,\ldots,\bx_d)\). Then, for any \((\bx_1',\ldots,\bx_d')\in W_\eps\) we have
\[
|g_{\infty}(\bx_1',\ldots,\bx_d') - g_{\infty}(\bx_1,\ldots,\bx_d)| < 2\delta + \frac{1}{3}\eps < \eps 
\]
whence \(g_\infty\) is continuous.

Now, let \(f:[0, 1]^d \to \R\) be a continuous function. Since \([0, 1]^d\) is compact, \(f\) is bounded, and the function
\[
F(\bx_1,\ldots, \bx_d) = \ev(f(G_{\bx_1}(B),\ldots, G_{\bx_d}(B)))
\]
is well-defined. We now show it is continuous. Let \((\bx_1,\ldots,\bx_d)\in \mX^d\) and fix \(\eps > 0\). Since \([0, 1]^d\) is compact, \(f\) can be approximated uniformly by a polynomial \(p:[0, 1]^d\to \R\). If
\[
p(z_1,\ldots, z_d) = \sum_{|\bgma|=0}^{N} c_{\bgma} \bz^{\bgma}\quad\mbox{and}\quad \sup_{\bz\in [0,1]^d}\, |f(\bz) - p(\bz)| < \frac{1}{3}\eps
\]
then
\[
P(\bx_1,\ldots,\bx_d) := \ev(p(G_{\bx_1}(B),\ldots, G_{\bx_d}(B))) = \sum_{|\bgma|=0}^{N} c_{\bgma} \ev\left(\prod_{k=1}^d \prod_{i_k=1}^{\gamma_k} G_{\bx_k}(B)\right)
\]
is a continuous function by our previous result. Furthermore, if we define the open neighborhood
\[
U_{\eps} = \Lset{(\bx_1',\ldots, \bx'_d):\, |P(\bx_1',\ldots,\bx_d') - P(\bx_1,\ldots,\bx_d)|<\frac{1}{3}\eps}
\]
of \((\bx_1,\ldots, \bx_d)\), then for any \((\bx_1',\ldots,\bx'_d)\in U_\eps\) we have
\begin{multline*}
    |F(\bx_1',\ldots,\bx_d') - F(\bx_1,\ldots,\bx_d)| \\
    \leq \ev(|f(G_{\bx_1'}(B),\ldots, G_{\bx_d'}(B)) - f(G_{\bx_1}(B),\ldots, G_{\bx_d}(B))|) < \eps
\end{multline*}
proving the claim.\qed

\subsection{The regularity of the probability density function of induced mixtures}

\subsubsection{Proof of Lemma~\ref{lem:mixtures:continuousDensityAll}}
\label{lem:mixtures:continuousDensityAll:proof}

We assume that \(P\) is fixed and we drop the superscript on \(\rho^P\). Let \((\byo,\bgmao,\bxo)\in \mY\times\Gma\times \mX\) and \(\eps > 0\). By hypothesis, there exists an open neighborhood \(U_{\byo}\times U_{\bgmao}\) of \((\byo,\bgmao)\) and a compact set \(K_{\btho}\subset \Th\) such that
\[
(\by,\bgma,\bth)\in U_{\byo}\times U_{\bgmao}\times K_{\btho}^{c}:\,\, \psi(\by,\bgma,\bth) < \frac{1}{4}\eps. 
\]
Furthermore, since \(\Th\) is Polish and \(P_{\bxo}\) is finite, there exists \(K_{\bxo}\subset \Th\) compact such that
\[
P_{\bxo}(\Th\setminus K_{\bxo}) < \eps.
\]
Define the compact set \(K_{\Th} := K_{\btho}\cup K_{\bxo}\). For \(\bth\in K_{\Th}\) let \(\delta_{\bth} > 0\) be such that
\[
\max\set{\dY(\by',\byo),\dGma(\bgma',\bgmao),\dTh(\bth',\bth)} < \delta_{\bth}\quad\Rightarrow\quad |\psi(\by',\bgma',\bth') - \psi(\byo,\bgmao,\bth)| < \frac{1}{8}\eps. 
\]
Let \(r_{\bth} > 0\) be such that \(2r_{\bth} < \delta_{\bth}\). Then \(\set{B(\bth,r_{\bth})}_{\bth\in K_{\Th}}\) is an open cover of \(K_{\Th}\). Hence, we can extract a finite subcover \(\set{B(\bthk,r_{\bthk})}_{k=1}^{N}\). Let \(r_0 >0 \) be such that 
\[
r_{0} < \min\set{r_{\bthk}}_{k=1}^N\quad\mbox{and}\quad B(\byo,r_0)\times B(\bgmao,r_0)\subset U_{\byo}\times U_{\bgmao}
\]
and define
\[
U_{\Th} := \bigcup_{k=1}^{N}B(\bthk,r_{\bthk}).
\]
Then there exists a continuous function \(h:\Th\to [0, 1]\) such that \(h \equiv 0\) on \(K_{\Th}\) and \(h\equiv 1\) on \(U_{\Th}^{c}\). Let
\[
\psi^0(\by,\bgma,\bx) = h(\bth)\psi(\by,\bgma,\bth)\quad\mbox{and}\quad \psi^R(\by,\bgma,\bx) = (1-h(\bth))\psi(\by,\bgma,\bth)
\]
and define \(\rho^0\) and \(\rho^R\) similarly. Then, for any
\[
(\by,\bgma,\bth) \in B(\byo,r_0)\times B(\bgmao,r_0) \times U_{\Th}
\]
there exists \(\bthk\in K_{\Th}\) such that
\begin{multline*}
    |\psi^0(\by,\bgma,\bth) - \psi^0(\byo,\bgmao,\bth)| \\
    < |\psi(\by,\bgma,\bth) - \psi(\byo,\bgmao,\bthk)| + |\psi(\byo,\bgmao,\bthk) - \psi(\byo,\bgmao,\bth)| < \frac{1}{4}\eps. 
\end{multline*}
For \((\by,\bgma)\in B(\byo,r_0)\times B(\bgmao,r_0)\) we have
\[
|\rho(\by,\bgma,\bx) - \rho(\byo,\bgmao,\bxo)| < |\rho(\by,\bgma,\bx) - \rho(\byo,\bgmao,\bx)| + |\rho(\byo,\bgmao,\bx) - \rho(\byo,\bgmao,\bxo)|.
\]
The second term can be controlled using the weak continuity of \(P\). In fact, it suffices to consider the open set
\[
U_{\bxo} := \Lset{\bx\in \mX:\,\, |\rho(\byo,\bgmao,\bx) - \rho(\byo,\bgmao,\bxo)| < \frac{1}{4}\eps}
\]
For the first term, consider first
\[
|\rho^0(\by,\bgma,\bx) - \rho^0(\byo,\bgmao,\bx)| < \int_{\overline{U_{\Th}}} (\psi^0(\by,\bgma,\bth) - \psi^0(\byo,\bgmao,\bth)) dP_{\bx}(\bth) < \frac{1}{4}\eps.
\]
For the second, note that
\[
\rho^R(\by,\bgma,\bx) = \int \psi^R(\by,\bgma,\bth) dP_{\bx}(\bth) \leq \int_{K_{\Th}^{c}} \psi(\by,\bgma,\bth) dP_{\bx}(\bth) < \frac{1}{8}\eps. 
\]
Consequently, for \((\by,\bgma,\bx)\in B(\byo,r_0)\times B(\bgmao,r_0) \times U_{\bxo}\) we have
\[
|\rho(\by,\bgma,\bx) - \rho(\byo,\bgmao,\bxo)| < \rho^R(\by,\bgma,\bx) + \rho^R(\byo,\bgmao,\bx) + \frac{1}{4}\eps + \frac{1}{4}\eps < \eps,
\]
proving the claim. \qed

\subsubsection{Proof of Lemma~\ref{lem:mixtures:continuousDensity}}
\label{lem:mixtures:continuousDensity:proof}

Consider the sequence of functions \(\set{q_n}_{n\in\N}\) for \(\qn : \mY\times \Gma \times \mX\times \Omega \to \Rp\) given by
\[
\qnwx(\by) = \sum_{i=1}^n \piixw\psi(\by, \bgma, \bthixw).
\]
Let \(\Omega_0\subset \Omega\) be a set of full measure such that: (i) \(\piiw \geq 0\); (ii) \(\sum_{i\in \N}\piiw \equiv 1\); and (iii) both \(\piiw\) and \(\bthiw\) are continuous on \(\mX\) for any \(i\in\N\). 
If we restrict the functions to \(\mY\times\Gma\times \mX\times \Omega_0\) then \(\qn \geq 0\) and \(\set{\qn}\) is a monotone non-decreasing sequence. Hence,
\[
\qw_{\infty, \bx}(\by) := \lim_{n\to\infty} \qnwx(\by)
\]
is well-defined for a.e. \(\omega\). By the monotone convergence theorem
\[
\int \qw_{\infty, \bx}(\by)d\muY(\by) \equiv 1. 
\]

To prove continuity for a.e. \(\omega\), fix \(\omega\in\Omega_0\). Let \(\byo\in \mY\), \(\bxo\in\mX\) and \(\bgmao\in \Gma\). By hypothesis, there exists an open neighborhood \(U_{\by_0}\) of \(\by_0\) and \(U_{\bgmao}\) of \(\bgmao\) such that
\[
\forall\, (\by,\bgma,\bth) \in U_{\by_0}\times U_{\bgmao}\times \Th:\quad \psi(\by,\bgma,\bth) \leq c_f < \infty
\]
for some \(c_f > 0\). Let \(\eps > 0\) and let \(N_\eps\in \N\) be such that 
\[
\sum_{i > N_\eps} \piixow < \frac{1}{8c_f}\eps.
\]
Define the open set \(U_{\pi} \subset \mX\)
\[
U_\pi := \bigcap_{i=1}^{N_\eps}\Lset{\bx\in\mX:\,\, |\piixw- \piixow| < \frac{1}{8c_f N_\eps}\eps}.
\]
For \(\bx\in U_\pi\) we have
\[
\sum_{i > N_\eps} \piixw= 1 - \sum_{i=1}^{N_\eps} \piixow + \sum_{i=1}^{N_\eps} (\piixw- \piixow) <\frac{1}{8c_f}\eps + \frac{1}{8c_f}\eps < \frac{1}{4c_f}\eps
\]
Now, define the open set \(U_{f,x}\subset \mY\times\Gma\times \mX\) be
\[
U_{f,x} := \bigcap_{i=1}^{N_\eps} \Lset{(\by,\bgma,\bx)\in U_{\by_0}\times U_{\bgmao}\times \mX:\,\, |\psi(\by,\bgma,\bthixw) -\psi(\by_0,\bgmao,\bthixow)| < \frac{1}{4c_f N_\eps}}.
\]

This set is open by continuity of \(\bthiw\). Then, for any \((\by,\bgma,\bx)\in U_{f,x}\) 
\[
\sum_{i > N_\eps}  \piixw\psi(\by, \bgma, \bthixw) < c_f\sum_{i > N_\eps} \piixw< \frac{1}{4}\eps. 
\]
Therefore,
\begin{multline*}
    |\qw_{\infty,\bx,\bgma}(\by) - \qw_{\infty,\bxo,\bgmao}(\by_0)| \\
    \leq \frac{1}{2}\eps + \sum_{i=1}^{N_\eps}|\piixw- \piixow||\psi(\by,\bgma,\bthixw)| + \sum_{i=1}^{N_\eps}\piixow |\psi(\by,\bgma,\bthixw) - \psi(\by_0,\bgmao,\bthixow)|\\
    < \frac{1}{2}\eps + \frac{1}{4}\eps + \frac{1}{4}\eps = \eps,
\end{multline*}
proving the lemma. \qed

\subsection{The continuity of mixtures induced by dependent Dirichlet processes}

\def\cY{c_{\mY}}

\subsubsection{Proof of Theorem~\ref{thm:mixtures:continuity:uniformAll}}
\label{thm:mixtures:continuity:uniformAll:proof}

Fix \(P\in C_W(\mX,\mP(\Th))\). To simplify notation, we drop the superscript on \(M^P\) and \(\rho^P\). Let \(\eps >0\). We will prove that there exist a neighborhood of \((\bgmao,\bxo)\) such that
\[
|Q_{\bgma,\bx}(B) - Q_{\bgmao,\bxo}(B)| < \eps
\]
uniformly on \(B\). The strategy is to find suitable compact subsets of \(\mY\) and \(\Th\) where the measures \(M\) and \(P_{\bxo}\) are concentrated near \((\bgmao,\bxo)\). Then, we leverage the continuity of \(\psi\) and the weak continuity of \(P\).

First, since \(\mY\) is Polish there exists \(K_{\mY}\subset \mY\) compact such that
\[
M_{\bgmao,\bxo}(\mY\setminus K_{\mY})  < \frac{1}{16}\eps. 
\]
Since \(\muY\) is locally finite, there exists \(\cY > 0\) such that
\[
\max\set{1, \muY(K_{\mY})} < \cY < \infty.
\]

Second, since \(\Th\) is Polish, there exists \(K_{\Th}\subset \Th\) compact such that
\[
P_{\bxo}(\Th\setminus K_{\Th}) < \frac{1}{16}\eps.
\]

We now construct a suitable cover for \(K_{\mY}\times\set{\bgmao}\times K_{\Th}\). Let \(\delta_{\by,\bth} > 0\) be such that
\[
\max\set{\dY(\by',\by), \dGma(\bgma',\bgmao), \dTh(\bth',\bth)} < \delta_{\by,\bth}\quad\Rightarrow\quad |\psi(\by',\bgma',\bth') - \psi(\by,\bgmao,\bth)| < \frac{1}{64\cY}\eps.
\]
Similarly, let \(r_{\by,\bth} > 0\) be \(2 r_{\by,\bth} < \delta_{\by,\bth}\). Then, for every \(\by\in K_{\mY}\) the collection
\[
\set{B(\by, r_{\by,\bth})\times  B(\bgmao,r_{\by,\bth})\times B(\bth,r_{\by,\bth})}_{\bth\in K_{\Th}}
\]
is an open cover of the compact set \(\set{\by}\times\set{\bgmao}\times K_{\Th}\). Hence, there exists a finite subcover
\[
\set{B(\by, r_{\by,\bth_\ell})\times  B(\bgmao,r_{\by,\bth_\ell})\times B(\bth_\ell,r_{\by,\bth_\ell})}_{\ell=1}^{N_{\by}}.
\]
Let \(r_{\by} > 0\) be such that
\[
r_{\by} < \min\set{r_{\by,\bth_1},\ldots, r_{\by,\bth_{N_{\by}}}}.
\]
Define the open neighborhood of \(K_{\Th}\)
\[
K_{\Th}^{r_{\by}} := \set{\bth\in \Th:\,\, \dTh(\bth, K_{\Th}) < r_{\by}}
\]
and the open set
\[
W_{\by} := B(\by,r_{\by})\times B(\bgmao,r_{\by})\times K_{\Th}^{r_{\by}}.
\]
Note that
\[
\set{\by}\times\set{\bgmao}\times K_{\Th} \subset W_{\by}.
\]
Furthermore, for any \((\by',\bgma',\bth')\in W_{\by}\) there exists, by construction, \(\bth''\in K_{\Th}\) such that
\[
\dTh(\bth',\bth'') < r_{\by}.
\]
Hence, for some \(\ell\in\set{1,\ldots, N_{\by}}\) we have
\[
\dTh(\bth',\bth_\ell) \leq \dTh(\bth',\bth'') + \dTh(\bth'',\bth_\ell) < r_{\by} + r_{\by,\bth_\ell} < \delta_{\by,\bth_\ell}.
\]
Furthermore, for the same choice of \(\bth_\ell\) we also have \(\dY(\by',\by) < \delta_{\by,\bth_\ell}\) and \(\dGma(\bgma',\bgmao) < r_{\by,\bth_\ell}\). Hence, we deduce that
\begin{multline*}
    |\psi(\by',\bgma',\bth') - \psi(\by,\bgmao,\bth')| \\
    \leq |\psi(\by',\bgma',\bth') - \psi(\by,\bgmao,\bth_\ell)| + |\psi(\by,\bgmao,\bth_{\ell}) - \psi(\by,\bgmao,\bth')|
    < \frac{1}{32\cY}\eps.
\end{multline*}

The collection \(\set{W_{\by}}_{\by\in K_{\mY}}\) is an open cover of the compact set \(K_{\mY}\times\set{\bgmao}\times K_{\Th}\). Hence, we can extract a subcover \(\set{W_{\byk}}_{k=1}^{N_{\mY}}\). By possibly removing elements, we may assume no ball is covered by the union of the remanining ones. Note that \(\set{B(\byk,r_{\byk})}_{k=1}^{N_{\mY}}\) is a cover for \(K_{\mY}\) with the same property. We partition of \(K_{\mY}\) into the sets
\begin{align*}
    A_1 &:= K_{\mY} \cap B(\byk,r_{\byk})\\
    A_k &:= K_{\mY} \cap \left(B(\byk,r_{\byk}) \setminus\bigcup_{\ell
        = 1}^{k-1}B(\by_{\ell},r_{\by_\ell})\right)\qquad k\in\set{2,\ldots, N_{\mY}}.
\end{align*}

Additionally, we let 
\[
U_{\bgmao} := \bigcap_{k=1}^{N_{\mY}} B(\bgmao,r_{\by_k}).
\]

Finally, consider the open set
\[
U_\Th := \set{\bth\in\Th:\,\, \dTh(\bth,K_{\Th}) < \min\set{r_{\by_1},\ldots, r_{\by_{N_{\mY}}}}} \subset \bigcap_{k=1}^{N_{\mY}} K_{\Th}^{r_{\byk}}.
\]
Then, there exists a continuous function \(h:\Th\to [0,1]\) such that \(h|_{K_{\Th}} \equiv 1\) and \(h|_{U_{\Th}^{c}}\equiv 0\). From \(K_{\Th}\) and \(U_\Th\) we obtain the decomposition,
\[
\psi^R(\by,\bgma,\bth):= (1-h(\bth))\psi(\by,\bgma,\bth)\quad\mbox{and}\quad \psi^0(\by,\bgma,\bth) = h(\bth)\psi(\by,\bgma,\bth).
\]
We define \(\rho^R, \rho^0, M^R\) and \(M^0\) similarly. By construction \(\psi^0\) is supported on \(\overline{U_\Th}\) and
\[
P_{\bxo}(K_{\Th}\setminus\overline{U_\Th}) \leq P_{\bxo}(K_{\Th}\setminus K_{\Th}) < \frac{1}{16}\eps. 
\]
Furthermore, for \((\by,\bgma,\bth)\in B(\by_k,r_{\by_k})\times U_{\bgmao}\times K_{\Th}^{r_{\by_k}}\)
\[
|\psi_0(\by,\bgma,\bth) - \psi_0(\by_k,\bgmao,\bth)| \leq h(\bth) |\psi(\by,\bgma,\bth) - \psi(\by_k,\bgmao,\bth)| < \frac{1}{32\cY}\eps.
\]

We now show that
\[
M^0_{\bgma,\bx}(K_{\mY}\cap B) := \int_{K_{\mY}}\int \psi_0(\by,\bgma,\bth) dP_{\bx}(\bth) d\muY(\by)
\]
can be approximated by a continuous function near \((\bgmao,\bxo)\). Consider the decomposition
\begin{multline*}
    M^0_{\bgma,\bx}(B) = \sum_{k=1}^{N_{\mY}} \muY(A_k\cap B)\int \psi_0(\by_k,\bgmao,\bth) dP_{\bx}(\bth) \\
    -\:\sum_{k=1}^{N_{\mY}} \int_{A_k\cap B}\int_{\overline{U_\Th}} (\psi^0(\by,\bgma,\bth) - \psi^0(\by_k,\bgmao,\bth)) dP_{\bx}(\bth)d\muY(\by).
\end{multline*}
Since
\[
\left|\int_{A_k\cap B}\int_{\overline{U_\Th}} (\psi^0(\by,\bgma,\bth) - \psi^0(\by_k,\bgmao,\bth)) dP_{\bx}(\bth)d\muY(\by)\right| < \frac{1}{32\cY} \eps \muY(A_k\cap B)
\]
whence
\[
\left|M^0_{\bgma,\bx}(K_{\mY}\cap B) - \sum_{k=1}^{N_{\mY}} \muY(A_k\cap B)\int \psi^0(\by_k,\bgmao,\bth) dP_{\bx}(\bth)\right| < \frac{1}{32}\eps. 
\]
Therefore
\begin{multline*}
    |M^0_{\bgma,\bx}(K_{\mY}\cap B) - M^0_{\bgmao,\bxo}(K_{\mY}\cap B)| \\
    < \frac{1}{16}\eps + \sum_{k=1}^{N_{\mY}}\muY(A_k\cap B) |\rho^0(\by_k,\bgmao,\bx) - \rho^0(\by_k,\bgmao,\bx_0)|.
\end{multline*}
It suffices to choose 
\[
U_{\bxo} := \bigcap_{k=1}^{N_{\mY}}\Lset{\bx\in \mX:\,\, \left|\int \psi^0(\by_k,\bgmao,\bth) dP_{\bx}(\bth) - \int \psi^0(\by_k,\bgmao,\bth) dP_{\bxo}(\bth)\right| < \frac{1}{16}\eps}
\]
as an open neighborhood for \(\bxo\). Hence, we obtain the bound
\[
|M^0_{\bgma,\bx}(K_{\mY}\cap B) - M^0_{\bgmao,\bxo}(K_{\mY}\cap B)| < \frac{1}{8}\eps
\]
which is uniform over \(B\) for  \((\bgma,\bx) \in U_{\bgmao}\times U_{\bxo}\). 

We now show the remaining terms do not concentrate too much mass. By construction,
\[
M_{\bgmao,\bxo}(K_{\mY}) = M^0_{\bgmao,\bxo}(K_{\mY}) + M^R_{\bgmao,\bxo}(K_{\mY}) \geq 1 - \frac{1}{16}\eps.
\]
Furthermore, by Fubini's theorem,
\begin{align*}
    |M_{\bgmao,\bxo}(K_{\mY}) - M^0_{\bgmao,\bxo}(K_{\mY})| &< \int_{K_{\mY}} \int_{\Th\setminus K_{\Th}} \psi^R(\by,\bgmao,\bth) dP_{\bxo}(\bth)d\muY(\by) \\
    &\leq  \int_{\Th\setminus K_{\Th}} \int_{K_{\mY}}\psi^R(\by,\bgmao,\bth)d\muY(\by) dP_{\bxo}(\bth)\\
    &\leq P_{\bxo}(\Th\setminus K_{\Th})\\
    &< \frac{1}{16}\eps,
\end{align*}
where we used the fact that \(1- h\) is supported on \(\Th\setminus K_{\Th}\). It follows that
\[
M^0_{\bgmao,\bxo}(K_{\mY}) > 1 - \frac{1}{8}\eps\quad\Rightarrow\quad M^R_{\bgmao,\bxo}(K_{\mY}) \leq \frac{1}{8}\eps. 
\]
Note then that
\[
M^0_{\bgma,\bx}(K_{\mY}) > M^0_{\bgmao,\bxo}(K_{\mY}) - \frac{1}{8}\eps > 1 - \frac{1}{4}\eps.   
\]
whence
\[
1 - M^0_{\bgma,\bx}(K_{\mY}) = M^R_{\bgma,\bx}(K_{\mY}) + M_{\bgma,\bx}(\mY\setminus K_{\mY}) < \frac{1}{4}\eps.
\]
Finally, from the decomposition
\begin{multline*}
    M_{\bgma,\bx}(B) = M_{\bgma,\bx}(K_{\mY}\cap B) + M_{\bgma,\bx}(B\setminus K_{\mY}) \\
    = M_{\bgma,\bx}^0(K_{\mY}\cap B) + M_{\bgma,\bx}^R(K_{\mY}\cap B) + M_{\bgma,\bx}(B\setminus K_{\mY})
\end{multline*}
we deduce
\begin{multline*}
    |M_{\bgma,\bx}(B) - M_{\bgmao,\bxo}(B)|  \leq |M_{\bgma,\bx}^0(K_{\mY}\cap B) - M_{\bgmao,\bxo}^0(K_{\mY}\cap B)|\\
    +\: M_{\bgma,\bx}^R(K_{\mY}\cap B) + M_{\bgma,\bx}(B\setminus K_{\mY}) +  M_{\bgmao,\bxo}^R(K_{\mY}\cap B) + M_{\bgmao,\bxo}(B\setminus K_{\mY})\\
    <\frac{1}{8}\eps + \frac{1}{4}\eps + \frac{1}{4}\eps < \eps
\end{multline*}
for \((\bgma,\bx)\in U_{\bgmao}\times U_{\bxo}\) uniformly in \(B\). This proves the theorem. \qed

\subsection{The support of mixtures induced by dependent Dirichlet processes}

To characterize the support of induced mixtures in different topologies, we will use repeatedly the following lemma. It provides uniform control on the behavior of induced mixtures uniformly over weakly continuous \(P:\mX\to \mP(\Th)\).

\begin{lemma}\label{lem:mixtures:uniformContinuityAux}
    Suppose \(\kmix\) is continuous and that for every \(\byo\in \mY\), \(\bgmao\in \mG\) and \(\eps > 0\) there exists \(K_0\subset \Th\) compact such that
    \[
    (\by,\bgma,\bth)\in U_{\byo}\times U_{\bgmao}\times K_0^{c}\quad\Rightarrow\quad \psi(\by,\bgma,\bth) < \eps. 
    \]
    The following assertions are true.
    \begin{enumerate}[leftmargin=24pt]
        \item{Let \(K_{\mY}G \subset\mY\times \mG\) be compact. Then, for every \(\eps > 0\) there exists \(U_{K_{\mY}G}\supset K_{\mY}G\) open, and \(\delta > 0\) such that for any \((\by,\bgma),(\by',\bgma')\in U_{K_{\mY}G}\) we have
            \[
            d_{\mY\times\mG}((\by,\bgma),(\by',\bgma')) < \delta\,\,\Rightarrow\,\, \sup_{\bx\in\mX}\left|\int_{Th} \kmix(\by',\bgma',\bth) dP_{\bx}(\bth) - \int_{Th} \kmix(\by',\bgma',\bth) dP_{\bx}(\bth) \right| < \eps
            \]
            uniformly over \(P:\mX\to \mP(\Th)\) weakly continuous.
        }
        \item{Let \(K_{\Gma\times\mX} \subset \mG\times\mX\) be compact, and let \(P:\mX\to \mP(\Th)\) be weakly continuous. Then, for every \(\eps > 0\) there exists \(K_{\mY}\subset \mY\) compact and \(U_{K_{\Gma\times\mX}}\supset K_{\Gma\times\mX}\) open such that
            \[
            \forall\,(\bgma,\bx) \in U_{K_{\Gma\times\mX}}:\quad Q^P_{\bgma,\bx}(\mY\setminus K_{\mY}) < \eps. 
            \]
        }
    \end{enumerate}
\end{lemma}

\begin{proof}[Proof of Lemma~\ref{lem:mixtures:uniformContinuityAux}]
    \noindent{\em Proof of 1.} Let \(\eps'\in (0, \eps)\). We first prove the result for \(K_{\mY}G = \set{(\byo,\bgmao)}\). By hypothesis, there exists neighborhoods \(U_{\byo}\) and \(U_{\bgmao}\) of \(\byo\) and \(\bgmao\) respectively, and \(K_{\Th}\subset \Th\) compact such that
    \[
    (\by',\bgma',\bth')\in U_{\byo}\times U_{\bgmao}\times K_{\Th}^{c}:\quad \psi(\by',\bgma',\bth') < \frac{1}{8}\eps'.
    \]
    For \(\btho\in K_{\Th}\) let \(\delta_{\btho} >0\) be such that
    \[
    d_{\mY\times\mG\times\Th}((\by,\bgma,\bth),(\byo,\bgmao,\btho)) < \delta_{\btho}\,\,\Rightarrow\,\, |\kmix(\by,\bgma,\bth) - \kmix(\byo,\bgmao,\btho)| < \frac{1}{8}\eps'.
    \]
    Let \(2r_{\btho} < \delta_{\btho}\). Then \(\set{B(\bth,r_{\bth})}_{\bth\in K_{\Th}}\) is an open cover of \(K_{\Th}\) from which we can extract a finite subcover \(\set{B(\bthk,r_{\bthk})}_{k=1}^{N}\). Define the neighborhoods
    \begin{align*}
        U'_{\byo} &:= U_{\byo}\cap\bigcup_{k=1}^{N} B(\byo,r_0),\\
        U'_{\bgmao} &:= U_{\bgmao}\cap\bigcup_{k=1}^{N} B(\bgmao,r_{\bthk}),\\
        U_{K_{\Th}} &:= \bigcup_{k=1}^{N} B(\bthk,r_{\bthk}),
    \end{align*}
    of \(\byo,\bgmao\) and \(K_{\Th}\) respectively, and let \(h:\Th\to [0,1]\) be a continuous map such that \(h\equiv 1\) on \(K_{\Th}\) and \(h\equiv 0\) on \(U_{K_{\Th}}^{c}\). Then, for any \(\by\in U'_{\byo}\) and \(\bgma\in U'_{\bgmao}\) we have
    \[
    \int_{Th} (1 - h(\bth))\psi(\by,\bgma,\bth)\, dP_{\bx}(\bth) \leq \int_{K_{\Th}^{c}}\psi(\by,\bgma,\bth)\, dP_{\bx}(\bth) < \frac{1}{8} \eps P_{\bx}(K_{\Th}^{c}) < \frac{1}{8}\eps'
    \]
    for any weakly continuous \(P:\mX\to \mP(\Th)\) and \(\bx \in \mX\). Hence, define
    \[
    \kmix_{K_{\Th}}(\by,\bgma,\bth) := h(\bth) \kmix(\by,\bgma,\bth).
    \]
    Then, for any \((\by,\bgma,\bth) \in U'_{\byo}\times U'_{\bgmao}\times U_{K_{\Th}}\) there exists \(\bthk\) such that
    \begin{multline*}
        |\kmix_{K_{\Th}}(\by,\bgma,\bth) - \kmix_{K_{\Th}}(\byo,\bgmao,\bth)| \\
        < |\kmix(\by,\bgma,\bth) - \kmix(\byo,\bgmao,\bthk)| + |\kmix(\byo,\bgmao,\bthk) - \kmix(\byo,\bgmao,\bth)| < \frac{1}{4}\eps'.
    \end{multline*}
    It follows that for any \((\by,\bgma) \in U'_{\byo}\times U'_{\bgmao}\) we have
    \begin{multline*}
        \left|\int_{\Th} \kmix(\by,\bgma,\bth) dP_{\bx}(\bth) - \int_{\Th} \kmix(\byo,\bgmao,\bth) dP_{\bx}(\bth)\right| \\
        < \frac{1}{4}\eps' + \int_{U_{K_{\Th}}} |\kmix_{K_{\Th}}(\by,\bgma,\bth) - \kmix_{K_{\Th}}(\byo,\bgmao,\bth)| dP_{\bx}(\bth) 
        < \frac{1}{2}\eps'
    \end{multline*}
    for any weakly continuous \(P:\mX\to \mP(\Th)\) and \(\bx\in \mX\). From now on, we let \(r_{(\byo,\bgmao)} > 0\) be such that
    \[
    B((\byo,\bgmao), r_{(\byo,\bgmao)}) \subset U'_{\byo}\times U'_{\bgmao}.
    \]
    
    We now consider the case for an arbitrary compact set \(K_{\mY\times\Gma}\). From the open cover \(\set{B((\by,\bgma),r_{(\by,\bgma)}/2)}_{(\by,\bgma)\in K_{\mY}G}\) we can extract a finite subcover \(\set{B((\byk,\bgmak),r_{(\byk,\bgmak)}/2)}_{k=1}^{N}\). Remark the radius of the cover is half of that obtained in the previous step. Define
    \[
    U_{K_{\mY}G} := \bigcup_{k=1}^N B((\byk,\bgmak),r_{(\byk,\bgmak)}/2)
    \]
    and
    \[
    \delta := \frac{1}{4}\min\set{r_{(\byk,\bgmak)}:\, k\in\bset{N}}.
    \]
    Let \((\by',\bgma'),(\by,\bgma)\in U_{K_{\mY}G}\) be such that
    \[
    d_{\mY\times\mG}((\by',\bgma'),(\by,\bgma)) < \delta.
    \]
    Then, there exists \((\byk,\bgmak)\in K_{\mY}G\) such that
    \[
    d_{\mY\times\mG}((\by,\bgma),(\byk,\bgmak)) < \frac{1}{2} r_{(\byk,\bgmak)}
    \]
    This implies that
    \[
    d_{\mY\times\mG}((\by',\bgma'),(\byk,\bgmak)) < d_{\mY\times\mG}((\by',\bgma'),(\by,\bgma)) + \frac{1}{2}r_{(\byk,\bgmak)} <  r_{(\byk,\bgmak)},
    \]
    from where it follows that \((\by',\bgma'),(\by,\bgma)\in B((\byk,\bgmak), r_{\byk,\bgmak})\). In particular,
    \begin{multline*}
        \left|\int_{\Th} \kmix(\by',\bgma',\bth) dP_{\bx}(\bth) - \int_{\Th} \kmix(\by,\bgma,\bth) dP_{\bx}(\bth)\right| \\
        \leq \left|\int_{\Th} \kmix(\by',\bgma',\bth) dP_{\bx}(\bth) - \int_{\Th} \kmix(\byk,\bgmak,\bth) dP_{\bx}(\bth)\right| \\
        +\: \left|\int_{\Th} \kmix(\by,\bgma,\bth) dP_{\bx}(\bth) - \int_{\Th} \kmix(\byk,\bgmak,\bth) dP_{\bx}(\bth)\right| < \frac{1}{2}\eps' + \frac{1}{2}\eps' = \eps'.
    \end{multline*}
    for any \(P:\mX\to \mP(\Th)\) and \(\bx\in \mX\). Since the supremum over \(\mX\) can be at most \(\eps'\) with \(\eps' < \eps\), this proves the lemma.
    
    \noindent{\em Proof of 2.} The hypothesis allow us to conclude from Theorem~\ref{thm:mixtures:continuity:uniformAll} that for any weakly continuous \(P:\mX\to \mP(\Th)\) the map \(Q^P:\mG\times\mX\to \mP(\mY)\) is strongly continuous. Hence, for every \((\bgma,\bx)\in K_{\Gma\times\mX}\) let \(K_{\bgma,\bx}\subset \mY\) be a compact set such that
    \[
    Q^P_{\bgma,\bx}(\mY\setminus K_{\bgma,\bx}) < \frac{1}{2}\eps. 
    \]
    Then, let \(U_{\bgma,\bx}\subset \mG\times \mX\) be an open neighborhood of \((\bgma,\bx)\) such that
    \[
    \forall\, (\bgma',\bx') \in U_{\bgma,\bx}:\quad  |Q^P_{\bgma',\bx'}(\mY\setminus K_{\bgma,\bx}) - Q^P_{\bgma.\bx}(\mY\setminus K_{\bgma,\bx})| < \frac{1}{2}\eps. 
    \]
    Hence, \(\set{U_{\bgma,\bx}}_{(\bgma,\bx)\in K_{\Gma}\times K_{\mX}}\) is an open cover of \(K_{\Gma\times\mX}\) from which we can extract a finite subcover \(\set{U_{\bgmak,\bxk}}_{k=1}^{N}\). Let
    \[
    U_{K_{\Gma\times\mX}} := \bigcup_{k=1}^{N} U_{\bgmak,\bxk}.
    \]
    Then that for every \((\bgma,\bx)\in U_{K_{\Gma\times\mX}}\) there exists \((\bgmak,\bxk)\) such that
    \[
    Q^P_{\bgma,\bx}(\mY\setminus K_{\bgmak,\bxk}) \leq \frac{1}{2}\eps + |Q^P_{\bgma,\bx}(\mY\setminus K_{\bgmak,\bxk}) - Q^P_{\bgmak,\bxk}(\mY\setminus K_{\bgmak,\bxk})| < \eps.
    \]
    Hence, we can choose the compact set
    \[
    K_{\mY} := \bigcup_{k=1}^{N} K_{\bgmak,\bxk},
    \]
    proving the claim.
\end{proof}

As a consequence of the lemma, if the hypotheses of Theorem~\ref{thm:continuity:weak} hold, then for every \(\omega\)  on a set of full measure we have
\[
\left|\int_{\Th} \kmix(\by',\bgma',\bth) d\Gxw(\bth) - \int_{\Th} \kmix(\by',\bgma',\bth) d\Gxw(\bth) \right| < \eps
\]
for any \(\bx\in \mX\). 

The lemma allows us to use essentially the same argument to characterize the support both in the product and compact-open topologies as follows. The statements in the case of the compact-open topology involve the supremum over a compact set \(K_{\Gma}\times K_{\mX}\subset\Gma\times \mX\) and \(\mPo:\mX\to\mP(\Th)\) weakly continuous. The statements in the case of the product topology can be reduced to this as follows. First, in the case of the product topology we need to consider a finite set \(\set{(\bgmai,\bxi):\,  i\in\bset{n}}\). We will see this is equivalent to bounding a supremum over the compact set
\[
K_{\Gma\times\mX} := \set{(\bgmai,\bxi):\, i\in\bset{n}}
\]
or by considering first the compact sets
\[
K_{\Gma} := \set{\bgmai:\, i\in\bset{n}}\quad\mbox{and}\quad K_{\mX} := \set{\bxi:\, i\in\bset{n}}
\]
and defining \(K_{\Gma\times\mX} :=K_{\Gma}\times K_{\mX}\). Second, in the case of the product topology we make no assumptions about the continuity of \(\mPo:\mX\to \mP(\Th)\). However, since only its values on the finite set \(\set{\bxi:\, i\in\bset{n}}\) are relevant, we can leverage Lemma~1 in Part I to replace \(\mPo\) by its weakly continuous interpolant \(\ol{\mPo}\). Once we have performed this reduction, Lemma~\ref{lem:mixtures:uniformContinuityAux} will allow us to prove the desired results using analogous arguments for the product and compact-open topologies.

\subsubsection{Proof of Theorems~\ref{thm:mixtures:support:weakHellinger} and~\ref{thm:mixtures:support:strongHellinger}}
\label{thm:mixtures:support:hellingerWeakStrong:proof}

In the case of the product-Hellinger topology, we define the compact sets \(K_{\Gma} := \set{\bgmai:\, i\in\bset{n}}\) and \(K_{\mX} := \set{\bxi:\, i\in\bset{n}}\). Furthermore, we let \(K_{\Gma\times\mX} =K_{\Gma}\times K_{\mX}\) and we let \(\eps_0 < \min\set{\eps_{1},\ldots,\eps_{n}}\). As indicated before, over the finite set \(K_{\mX}\) we can assume without loss that \(\mPo\) is weakly continuous. For the compact-Hellinger topology, we define \(K_{\Gma\times\mX} =K_{\Gma}\times K_{\mX}\) and let \(\eps_0 < \eps\). 

This reduction allows us to consider the event
\[
\Lset{\omega\in\Omega: \sup_{(\bgma,\bx)\in K_{\Gma\times\mX}} H(\rho^{\Gw}_{\bgma,\bx}, \rho^{\mPo}_{\bgma,\bx}) < \eps_0}.
\]
Both Theorem~\ref{thm:mixtures:support:weakHellinger} and~\ref{thm:mixtures:support:strongHellinger} follow if we show the above event has positive probability. 

Note that
\[
H(\rho^{\Gw}_{\bgma,\bx}, \rho^{\mPo}_{\bgma,\bx}) < \eps_0\,\,\Leftrightarrow\,\, 1 - \int_{\mY} \rho^{\Gw}(\by,\bgma,\bx)^{1/2} \rho^{\mPo}(\by,\bgma,\bx)^{1/2} d\muY(\by) < 2\eps_0^2
\]
Let \(\eps\in (0, 1)\) be such that \(\eps < 2\eps_0^2\). From Lemma~\ref{lem:mixtures:uniformContinuityAux} there exists a compact set \(K_{\mY}\subset \mY\) such that
\[
\sup_{(\bgma,\bx)\in K_{\Gma}\times K_{\mX}} Q^P_{\bgma,\bx}(\mY\setminus K_{\mY}) < \frac{1}{4}\eps.
\]
Since \(\muY\) is locally finite, we can assume without loss that \(\muY(K_{\mY}) < \infty\). Define the compact set \(K_{\mY}G = K_{\mY}\times K_{\Gma}\). By Lemma~\ref{lem:mixtures:continuousDensityAll} there exists \(\delta > 0\) such that for any \((\by',\bgma'),(\by,\bgma)\in K_{\mY\times\Gma}\) we have
\begin{multline*}
    d_{\mY\times\mG}((\by',\bgma'),(\by,\bgma)) < \delta\quad\Rightarrow \\ 
    \left|\int_{\Th} \kmix(\by',\bgma',\bth) dP_{\bx}(\bth) - \int_{\Th} \kmix(\by,\bgma,\bth) dP_{\bx}(\bth) \right| < \frac{1}{16\muY(K_{\mY})}\eps^2
\end{multline*}
uniformly over weakly continuous \(P:\mX\to \mP(\Th)\) and \(\bx\in \mX\). In particular,
\[
\left|\left(\int_{\Th} \kmix(\by',\bgma',\bth) dP_{\bx}(\bth)\right)^{1/2} - \left(\int_{\Th} \kmix(\by,\bgma,\bth) dP_{\bx}(\bth)\right)^{1/2} \right| < \frac{1}{4\muY(K_{\mY})^{1/2}}\eps.
\]

Let \(2r < \delta\). We construct a finite open cover for \(K_{\mY}G\) as follows. First, from the open cover \(\set{B(\by, r)}_{\by \in K_{\mY}}\) we can extract a finite subcover \(\set{B(\byk, r)}_{k=1}^N\). Without loss, we can assume it is minimal, and we can partition \(K_{\mY}\) in terms of the measurable sets of positive measure
\begin{align*}
    A_1 &:= K_{\mY} \cap B(\byk,r)\\
    A_k &:= K_{\mY} \cap \left(B(\byk,r) \setminus\bigcup_{\ell
        = 1}^{k-1}B(\byl ,r)\right)\qquad k\in\set{2,\ldots, N}.
\end{align*}
Second, from the open cover \(\set{B(\bgma, r)}_{\bgma\in K_{\Gma}}\) we can extract a finite subcover \(\set{B(\bgmal, r)}_{\ell=1}^M\). Note that \(\set{B(\byk,r)\times B(\bgmal,r):\, k\in\bset{N},\,\ell\in\bset{M}}\) is an open cover for \(K_{\mY}G\). Hence, for any \((\bgma,\bx)\in K_{\Gma\times\mX}\) we can write
\begin{align*}
    1 - \frac{1}{4}\eps &\leq \int_{K_{\mY}} \qPo(\by,\bgma,\bx)\, d\muY(\by)\\
    & = \sum_{k=1}^{N} \int_{A_k}(\qPo(\by,\bgma,\bx)^{1/2} - \qPo(\byk,\bgmal,\bx)^{1/2}) \qPo(\by,\bgma,\bx)^{1/2}d\muY(\by) \\
    &+\:\sum_{k=1}^{N} \int_{A_k}(\qPo(\byk,\bgmal,\bx)^{1/2} - \qGw(\byk,\bgmal,\bx)^{1/2}) \qPo(\by,\bgma,\bx)^{1/2} d\muY(\by) \\
    &+\:\sum_{k=1}^{N} \int_{A_k} (\qGw(\byk,\bgmal,\bx)^{1/2}-\qGw(\by,\bgma,\bx)^{1/2}) \qPo(\by,\bgma,\bx)^{1/2} d\muY(\by)\\
    &+\:\int_{K_{\mY}} \qGw(\by,\bgma,\bx)^{1/2} \qPo(\by,\bgma,\bx)^{1/2} d\muY(\by)
\end{align*}
where \(\bgma\in B(\bgmal,r)\). The first sum can be bounded as
\begin{multline*}
    \int_{A_k}(\qPo(\by,\bgma,\bx)^{1/2} - \qPo(\byk,\bgmal,\bx)^{1/2}) \qPo(\by,\bgma,\bx)^{1/2}d\muY(\by) \\
    \leq \frac{1}{4\muY(K_{\mY})^{1/2}}\eps \int_{K_{\mY}} \qPo(\by,\bgma,\bx)^{1/2}d\muY(\by) \leq \frac{1}{4\muY(K_{\mY})^{1/2}}\eps\muY(K_{\mY})^{1/2}Q^P(K_{\mY})^{1/2} \leq \frac{1}{4}\eps. 
\end{multline*}
The third sum is a.s. bounded by the same arguments. Hence,
\begin{multline*}
    1 - \int_{\mY} \qGw(\by,\bgma,\bx)^{1/2}\qPo(\by,\bgma,\bx)^{1/2}d\muY(\by) \\
    < \frac{3}{4}\eps + \sum_{k=1}^{N} \int_{A_k}(q^P(\byk,\bgmal,\bx)^{1/2} - \qGw(\byk,\bgmal,\bx)^{1/2}) q^Po(\by,\bgma,\bx)^{1/2} d\muY(\by).
\end{multline*}
To prove the theorems, it remains to bound the integral in the right-hand side. 

The hypotheses of Theorem~\ref{thm:mixtures:support:weakHellinger} allow us to apply Theorem~5 in Part I to show the event
\begin{multline*}
    \left\{\omega\in\Omega: \left|\int_{\Th} \kmix(\byk,\bgmal,\bth)d\Gw_{\bxi}(\bth) - \int_{\Th} \psi(\byk,\bgmai,\bth)dP^0_{\bxi}(\bth)\right| < \frac{\eps}{4\muY(K_{\mY})^{1/2}},\,\right. \\\left. \phantom{\int_{\Th}} k\in\bset{N},\,\ell\in\bset{M},\, i\in\bset{n}\right\},
\end{multline*}
has positive probability. This proves Theorem~\ref{thm:mixtures:support:weakHellinger}.

The hypotheses of Theorem~\ref{thm:mixtures:support:strongHellinger} allow us followsto apply Theorem~7 in Part I to show the event 
\begin{multline*}
    \left\{\omega\in\Omega: \sup_{\bx\in K_{\mX}}\left|\int_{\Th} \psi(\byk,\bgmal,\bth)d\Gw_{\bx}(\bth) - \int_{\Th} \psi(\byk,\bgmal,\bth)d\mPo_{\bx}(\bth)\right| < \frac{\eps}{8\muY(K_{\mY})^{1/2}},\right.\\
    \left. \phantom{\int_{\Th}}\,\,k\in\bset{N},\,\ell\in\bset{M}\right\},
\end{multline*}
has positive probability. This proves Theorem~\ref{thm:mixtures:support:weakHellinger}.\qed

\subsubsection{Proof of Theorems~\ref{thm:mixtures:support:weakLInf} and~\ref{thm:mixtures:support:strongLInf}}
\label{thm:mixtures:support:LInfWeakStrong:proof}

In the case of the product-\(\LInf\) topology we can define the compact set \(K_{\Gma\times\mX} \subset K_{\Gma}\times K_{\mX}\) as in Appendix~\ref{thm:mixtures:support:hellingerWeakStrong:proof} and let \(\eps_0 < \min\set{\eps_{1},\ldots,\eps_{n}}\). As indicated before, over this finite set we can assume without loss that \(\mPo\) is weakly continuous. For the compact-\(\LInf\) topology, we define \(K_{\Gma\times\mX} =K_{\Gma}\times K_{\mX}\) and let \(\eps_0 < \eps\). 

This reduction allows us to consider the event
\[
\Lset{\omega\in\Omega: \sup_{(\bgma,\bx)\in K_{\Gma\times\mX}} \nrmLInf{\qGw_{\bgma,\bx} - q^Po_{\bgma,\bx}} < \eps_0}.
\]
Both Theorem~\ref{thm:mixtures:support:weakLInf} and~\ref{thm:mixtures:support:strongLInf} and follow is we show the above event has positive probability. Let \(\eps\in [0, 1)\) be such that \(\eps < \eps_0\). 

Let \(K_{\mY}G = \mY\times K_{\Gma}\) which, by hypothesis, is compact. From Lemma~\ref{lem:mixtures:uniformContinuityAux} there exists \(\delta > 0\) such that for any \((\by',\bgma'),(\by,\bgma)\in K_{\mY}G\) we have
\[
d_{\mY\times\mG}((\by',\by), (\bgma',\bgma)) < \delta\quad\Rightarrow\quad \left|\int_{\Th} \kmix(\by',\bgma',\bth) dP_{\bx}(\bth) - \int_{\Th} \kmix(\by,\bgma,\bth) dP_{\bx}(\bth)\right| < \frac{1}{4}\eps
\]
uniformly over weakly continuous \(P:\mX\to \mP(\Th)\) and \(\bx\in \mX\). Let \(r >0\) be such that \(2 r < \delta\). From the open cover \(\set{B((\by,\bgma),r)}_{(\by,\bgma)\in K_{\mY}G}\) we can extract a finite subcover \(\set{B((\byk,\bgmak), r)}_{k=1}^{N}\). Hence, for any \((\by,\bgma)\in K_{\mY}G\) there exists \(k\in \bset{N}\) such that
\[
\left|\int_{\Th} \kmix(\by,\bgma,\bth) dP_{\bx}(\bth)\right| < \frac{1}{4}\eps + \left|\int_{\Th} \kmix(\byk,\bgmak,\bth) dP_{\bx}(\bth)\right|
\]
uniformly over \(P:\mX\to \mP(\Th)\) weakly continuous and \(\bx\in \mX\). In particular,
\begin{multline*}
    \left|\int_{\Th} \kmix(\by,\bgma,\bth) d\Gw_{\bx}(\bth) - \int_{\Th} \kmix(\by,\bgma,\bth) dP_{\bx}(\bth)\right|\\
    \leq \frac{1}{2}\eps + \left|\int_{\Th} \kmix(\byk,\bgmak,\bth) d\Gw_{\bx}(\bth) - \int_{\Th} \kmix(\byk,\bgmak,\bth) dP_{\bx}(\bth)\right|.
\end{multline*}

The hypotheses of Theorem~\ref{thm:mixtures:support:weakLInf} and the fact that \(\bgmak \in\set{\bgma_1,\ldots,\bgma_n}\) allows us to apply Theorem~5 in Part I to show that the event
\[
\Lset{\omega\in\Omega: \left|\int_{\Th} \kmix(\byk,\bgmai,\bth)d\Gw_{\bxi}(\bth) - \int_{\Th} \kmix(\byk,\bgmai,\bth)dP^0_{\bxi}(\bth)\right| < \frac{1}{2}\eps,\,\, i\in \bset{n},\,k\in\bset{N}}
\]
has positive probability. This proves Theorem~\ref{thm:mixtures:support:weakLInf}. The hypotheses of Theorem~\ref{thm:mixtures:support:strongLInf} allow us to apply Theorem~7 in Part I to show that the event
\[
\Lset{\omega\in\Omega: \sup_{\bx\in K_{\mX}}\left|\int_{\Th} \kmix(\byk,\bgmak,\bth)d\Gw_{\bx}(\bth) - \int_{\Th} \kmix(\byk,\bgmak,\bth)dP^0_{\bx}(\bth)\right| < \frac{1}{2}\eps,\,\,k\in\bset{N}},
\]
has positive probability. This proves Theorem~\ref{thm:mixtures:support:strongLInf}. \qed

\subsubsection{Proof of Theorems~\ref{thm:mixtures:support:weakKL} and~\ref{thm:mixtures:support:strongKL}}
\label{thm:mixtures:support:KLWeakStrong:proof}

In the case of the product-KL topology we proceed as for the product-\(\LInf\) topology. We define the compact sets \(K_{\Gma} := \set{\bgmai:\, i\in\bset{n}}\) and \(K_{\mX} := \set{\bxi:\, i\in\bset{n}}\). Furthermore, we let \(K_{\Gma\times\mX} =K_{\Gma}\times K_{\mX}\) and we let \(\eps_0 < \min\set{\eps_{1},\ldots,\eps_{n}}\). As indicated before, over this finite set we can assume without loss that \(\mPo\) is weakly continuous. For the compact-KL topology, we define \(K_{\Gma\times\mX} =K_{\Gma}\times K_{\mX}\) and let \(\eps_0 < \eps\). 

The hypothesis \(\kmix > 0\) implies that \(q^{\mPo} > 0\). Furthermore, since \(\mY\times K_{\Gma}\times K_{\mX}\) is compact, there exists \(c_{\max} > 0\) such that \(q^Po \leq c_{\max}\). Let \(\eps' > 0\) be such that \(\eps' < \eps_0 / (1 + \eps_0)\) and consider the event
\[
\Lset{\omega\in \Omega:\,\, \sup_{(\by,\bgma,\bx)\in K_{\mY}\times K_{\Gma}\times K_{\mX}}\,|\qGw(\by,\bgma,\bx) - q^Po(\by,\bgma,\bx)| < \eps'c_{\max}}.
\]
The hypotheses allows us to apply Theorem~\ref{thm:mixtures:support:weakLInf} or Theorem~\ref{thm:mixtures:support:strongLInf} respectively to prove this event has positive probability. Furthermore, on this event we have
\[
\left|\frac{\qGw(\by,\bgma,\bx)}{q^Po(\by,\bgma,\bx)} - 1\right| < \eps'.
\]
Since for \(t > -1\) we have
\[
\frac{t}{1 + t} \leq \log(1 + t) \leq t,
\]
we deduce that
\[
\left|\log\left(\frac{q^Po(\by,\bgma,\bx)}{\qGw(\by,\bgma,\bx)}\right)\right| < \frac{\eps'}{1 - \eps'} < \eps_0. 
\]
In particular,
\[
\KL{q^{\mPo}}{\qGw} = \int_{\mY} q^Po(\by,\bgma,\bx)\log\left(\frac{q^{\mPo}(\by,\bgma\bx)}{\qGw(\by,\bgma,\bx)}\right)d\muY(\by) < \eps_0.
\]
Consequently, the event
\[
\Lset{\omega\in \Omega:\,\, \sup_{(\bgma,\bx)\in K_{\Gma}\times K_{\mX}}\,\KL{q^Po_{\bgma,\bx}}{\qGw_{\bgma,\bx}} < \eps_0}
\]
has positive probability, proving the theorem.\qed

\bibliographystyle{unsrtnat}
\bibliography{references}  






\end{document}